\documentclass[preprint]{elsarticle}

\author[1]{Chad Nester\corref{cor1}}

\author[1,2]{Niels Voorneveld}

\affiliation[1]{
  organization={Tallinn University of Technology},
  addressline={Akadeemia Tee 21/1},
  postcode={12618},
  city={Tallinn},
  country={Estonia}
}

\affiliation[2]{
  organization={Cybernetica AS},
  addressline={Mäealuse 2/1},
  postcode={12618},
  city={Tallinn},
  country={Estonia}
}

\cortext[cor1]{Corresponding Author}

\title{Protocol Choice and Iteration for the Free Cornering\tnoteref{t1,t2}}

\tnotetext[t1]{Both authors were supported by the ESF funded Estonian IT Academy research measure (project 2014-2020.4.05.19-0001).}

\tnotetext[t2]{Niels Voorneveld was supported by the Estonian Research Council, grant no. PRG1780.}

\usepackage{amsmath,amsthm,amssymb}
\usepackage{stmaryrd,enumerate,mathpartir}
\usepackage{graphicx,graphbox,hyperref}
\usepackage{tikz-cd}
\usepackage{csquotes}
\usepackage{xypic}

\newtheorem{proposition}{Proposition}

\newtheorem{lemma}{Lemma}
\newtheorem{conjecture}{Conjecture}

\theoremstyle{definition}
\newtheorem{definition}{Definition}
\newtheorem{example}{Example}

\newtheorem{remark}{Remark}

\newcommand{\A}{\mathbb{A}}

\newcommand{\D}{\mathbb{D}}
\newcommand{\C}{\mathbb{C}}
\newcommand{\bh}{\mathbf{H}}
\newcommand{\bv}{\mathbf{V}}

\newcommand{\getl}{\mathsf{get}_\mathsf{L}}
\newcommand{\putr}{\mathsf{put}_\mathsf{R}}
\newcommand{\getr}{\mathsf{get}_\mathsf{R}}
\newcommand{\putl}{\mathsf{put}_\mathsf{L}}

\newcommand{\ev}{\mathsf{ev}}

\newcommand{\Dst}[1]{\mathsf{S}(#1)}

\newcommand{\corner}[1]{{\text{}^\ulcorner_\llcorner\!{#1}\!_\lrcorner^\urcorner}}
\newcommand{\ex}[1]{{#1}^{\circ\bullet}}

\newcommand{\cells}[4]{({\scriptstyle #1} {{\scriptstyle #2} \atop {\scriptstyle #3}} {\scriptstyle #4})}
\newcommand{\floatcells}[4]{\left({\scriptstyle #1} {{\scriptstyle #2} \atop {\scriptstyle #3}} {\scriptstyle #4}\right)}

\newcommand{\alice}{\texttt{Alice}}
\newcommand{\bob}{\texttt{Bob}}

\newcommand{\ip}{\rotatebox[origin=c]{180}{$\pi$}}
\newcommand{\from}{\leftarrow}

\makeatletter
\providecommand{\leftsquigarrow}{%
  \mathrel{\mathpalette\reflect@squig\relax}%
}
\newcommand{\reflect@squig}[2]{%
  \reflectbox{$\m@th#1\rightsquigarrow$}%
}
\makeatother

\newcommand{\pl}{p_0}
\newcommand{\pr}{p_1}

\begin{document}

\begin{abstract}
  We extend the free cornering of a symmetric monoidal category, a double categorical model of concurrent interaction, to support branching communication protocols and iterated communication protocols. We validate our constructions by showing that they inherit significant categorical structure from the free cornering, including that they form monoidal double categories. We also establish some elementary properties of the novel structure they contain. Further, we give a model of the free cornering in terms of strong functors and strong natural transformations, inspired by the literature on computational effects. 
\end{abstract}


\maketitle

\section{Introduction}
While there are many theories of concurrent computation, none may yet claim to be the \emph{canonical} such theory. In the words of Abramsky~\cite{Abramsky2006}:
\begin{displayquote}
  It is too easy to cook up yet another variant of process calculus or algebra; there are too few constraints \ldots The mathematician Andr\'{e} Weil apparently compared finding the right definitions in algebraic number theory --- which was like carving adamantine rock --- to making definitions in the theory of uniform spaces \ldots which was like sculpting with snow. In concurrency theory we are very much at the snow-sculpture end of the spectrum. We lack the kind of external reality \ldots which is hard and obdurate, and resistant to our definitions.
\end{displayquote}

This motivates the search for categorical models of concurrency, with category theory playing the role of a suitably stubborn external reality against which to test our definitions. In this we adopt the perspective suggested by Maddy~\cite{Maddy2019} on the relationship of category theory to set theory. Loosely, set theory serves among other things as a \emph{generous arena} in which the full array of mathematical construction techniques are permissible, but with no way of discerning the mathematically promising structures from the rest (sculpting with snow). The role of category theory is to provide \emph{essential guidance}, with the idea being that the mathematically promising structures are precisely those which fit smoothly into the category-theoretic landscape (carving rock). Whether or not category theory can provide this sort of essential guidance in the theory of concurrent computation has not yet been conclusively established, but the idea is a compelling one, and we work under the assumption that it can and will. 

 This paper concerns the free cornering of a symmetric monoidal category, a categorical model of concurrent interaction proposed by Nester~\cite{Nester2021a,Nester2023}. The model builds on the resource-theoretic interpretation of symmetric monoidal categories as a kind of process theory (see e.g.,~\cite{Coecke14}) by defining the \emph{free cornering} of a given symmetric monoidal category to be a certain double category in which the processes represented by the base are augmented with \emph{corner cells}. The cells of the free cornering admit interpretation as \emph{interacting processes}, with the corner cells embodying a notion of message passing. The corner cells are precisely what is required to make the free cornering into a \emph{proarrrow equipment}, which is a kind structured double category that plays a fundamental role in formal category theory. We remark that for the structure of a proarrow equipment to coincide with a notion of message passing would seem to be essential guidance of the highest quality. That said, the free cornering is currently far from being a canonical model of concurrent computation. Much work remains to be done, and many connections remain unexplored.

 More specifically, this paper addresses the connection of the free cornering to notions of \emph{session type}. Roughly, session types are to communication protocols what data types are to structured data. Much like the way that data types constrain the possible data values a process operates on, session types constrain the behaviour of a process so that it conforms to the corresponding communication protocol. In the free cornering of a symmetric monoidal category, process interaction is governed by the \emph{monoid of exchanges}, whose elements may be understood as a basic sort of session type. The communication protocols expressible in this way are those in which messages of a predetermined type are sent and received in a predetermined sequence. Viewed this way, the monoid of exchanges is missing a number of features that systems of session types are expected to have. In particular, \emph{branching protocols} --- in which one of the participants chooses which of two possible continuations of the current protocol will happen and the other participant must react --- and \emph{iterated protocols}, in which all or part of the protocol is carried out some number of times based on choices made by the participants.

In an effort to rectify the situation, we construct the \emph{free cornering with choice} and \emph{free cornering with iteration} of a distributive monoidal category. The free cornering with choice supports branching communication protocols in addition to those of the free cornering, and the free cornering with iteration supports iterated communication protocols in addition to those of the free cornering with choice. To ask that the base category is distributive monoidal is to ask that it supports a kind of sequential branching, analogous to the ``if then else'' statements present in many programming languages. In the free cornering with choice, this sequential branching structure interacts nontrivially with branching communication protocols, and is what allows a process to decide which branch of a protocol to select based on its inputs. 

We prove some elementary results concerning the monoidal category of \emph{horizontal cells} of the free cornering with choice and the free cornering with iteration. The objects of the category of horizontal cells correspond to communication protocols. We show that the category of horizontal cells of the free cornering with choice has binary products and coproducts given by branching communication protocols. Further, we characterize iterated communication protocols in the category of horizontal cells of the free cornering with iteration by showing that they arise as initial algebras (or dually, final coalgebras) of a suitable endofunctor, and that process iteration forms a monad (or dually, a comonad). 

An important feature of the free cornering is the existence of well-behaved \emph{crossing cells}, which among other things carry the structure of a \emph{monoidal double category} in the sense of Shulman~\cite{Shulman2010}. We extend the construction of crossing cells in the free cornering to construct crossing cells in the free cornering with choice and the free cornering with iteration, and show that these crossing cells remain well-behaved. It follows that the free cornering with choice and free cornering with iteration also form monoidal double categories. We view the presence of these well-behaved crossing cells as a kind of sanity check. 

As an additional sanity check, we construct a model of the free cornering, and extend it to give a model of the free cornering with choice and the free cornering with iteration. Specifically, from a cartesian closed category we construct a double category of \emph{stateful transformations}, in which cells are given by strong natural transformations between strong endofunctors on the base category (see e.g.~\cite{McDermott2022}). This double category is a model of the free cornering in the sense that there is a structure-preserving double functor from the free cornering of the base category into the category of stateful transformations. We show that under additional assumptions on the base category, the double category of stateful transformations gives a model of the free cornering with choice and the free cornering with iteration in this sense. The existence of such a model is reassuring in that it tells us the axioms of the free cornering with choice and free cornering with iteration do not collapse. Strong functors play an important role in the categorical semantics of effectful computations (see e.g,\cite{Moggi91,Plotkin01}), and so the double category of stateful transformations may be of independent interest. 

 \paragraph{Contributions}
 The central contributions of this paper are the construction of the free cornering with choice (Definitions~\ref{def:choice-exchanges} and~\ref{def:free-cornering-choice}) and the free cornering with iteration (Definitions~\ref{def:iteration-exchanges} and~\ref{def:iter-doub}). In this we include the results validating the two constructions, specifically the contents of Sections~\ref{subsec:choice-properties}, \ref{subsec:choice-crossing}, \ref{subsec:iteration-properties}, and~\ref{subsec:iteration-crossing}. A further contribution is the construction of the double category of stateful transformations together with the fact that it models the free cornering (Section~\ref{subsec:stateful-transformations}), free cornering with choice (Section~\ref{subsec:stateful-choice}), and free cornering with iteration (Section~\ref{subsec:stateful-iteration}) in the presence of suitable assumptions on the base category. Finally, Lemma~\ref{lem:coherent-crossing} is a minor contribution to the theory of the free cornering (without choice or iteration): while in previous work on the free cornering it has been part of the definition of the crossing cells, in recapitulating the material on crossing cells we realised that it was in fact a consequence of the slightly weaker definition used here. 

 \paragraph{Related Work}
 Double categories first appear in~\cite{Ehr63}. Free double categories are considered in~\cite{Daw02} and again in~\cite{Fio08}. The idea of a proarrow equipment first appears in~\cite{Woo82}, albeit in a rather different form. Proarrow equipments have subsequently appeared under many names in formal category theory (see e.g.,~\cite{Shulman2008,Grandis2004}). The string diagrams for double categories and proarrow equipments that we will use without comment are given a detailed treatment in~\cite{Mye16}. The free cornering was introduced in~\cite{Nester2021a}, and has been developed in~\cite{Nester2023,Nester2022,Boisseau2022}. Session types were introduced by Honda~\cite{Honda1993} and the idea has since been developed by a number of authors. While the purpose of this paper is to develop more sophisticated session types in the free cornering, we are primarily influenced not by the literature on session types after Honda but by the \emph{logic of message passing} of Cockett and Pastro~\cite{Cockett2009}, in which process communication is modelled categorically by \emph{linear actegories} (the semantics of a kind of augmented linear logic). A particular point of difference between these two lines of research is that in the logic of message passing, like in the free cornering, the protocols that our types describe are \emph{two-sided}, requiring a left and right participant. In the literature on session types after Honda the protocol types are \emph{one-sided}, and two participants must each conform to a session type \emph{dual} to that of the other if they wish to interact. The connection between session types and linear logic is explored from slightly different angles in~\cite{Wadler2014} and~\cite{Caires2010}, and all of this seems to have been heavily influenced by the early work of Bellin and Scott~\cite{Bellin1994}. In our use of distributive monoidal categories to model branching programs and datatypes we follow Walters~\cite{Walters1989}.
 The definitions of protocols in the model of stateful transformations are based on the theory of strong functors and monads for describing computational effects by Moggi~\cite{Moggi91}.
 Stateful transformations themselves are inspired by \emph{stateful runners}~\cite{Uustalu2015}, and \emph{interaction laws} as described in~\cite{Katsumata20}.
 
 \paragraph{Organisation} In Section~\ref{sec:corner-classical} we give an introduction to single-object double categories (Section~\ref{subsec:single-double}); recapitulate the construction of the free cornering of a symmetric monoidal category (Section~\ref{subsec:free-cornering}); recall the construction of crossing cells in the free cornering along with certain properties of crossing cells (Section~\ref{subsec:crossing-cells}); and introduce the double category of stateful transformations and its relationship to the free cornering (Section~\ref{subsec:stateful-transformations}). Section~\ref{sec:corner-choice} concerns the free cornering with choice. We introduce distributive monoidal categories and discuss the way in which they model branching sequential processes (Section~\ref{subsec:distributive-branching}); introduce the free cornering with choice of a distributive monoidal category together with its interpretation (Section~\ref{subsec:free-choice}); establish a few elementary properties of the resulting single-object double category (Section~\ref{subsec:choice-properties}); show that the construction of crossing cells in the free cornering extends to the free cornering with choice, and that the attendant properties of crossing cells hold in the larger setting (Section~\ref{subsec:choice-crossing}); and show that when the base cartesian closed category is distributive the category of stateful transformations gives a model of the free cornering with choice (Section~\ref{subsec:stateful-choice}). Section~\ref{sec:corner-iteration} concerns the free cornering with iteration, and its organisation is similar to that of Section~\ref{sec:corner-choice}. We introduce the free cornering with iteration of a distributive monoidal category together with its interpretation (Section~\ref{subsec:free-iteration}); establish a few elementary properties of the resulting single-object double category, in particular that our notion of iteration is (co)monadic (Section~\ref{subsec:iteration-properties}); show that the construction of crossing cells in the free cornering with choice extends to the free cornering with iteration, and that the attendant properties of crossing cells conitnue to hold in the larger setting (Section~\ref{subsec:iteration-crossing}); and show that when we consider only the part of the double category of stateful transformations given by certain \emph{container functors} on the base category $\mathsf{Set}$ this gives a model of the free cornering with iteration (Section~\ref{subsec:stateful-iteration}). We conclude and discuss a number of directions for future work in Section~\ref{sec:concluding-remarks}. 

 \paragraph{Prerequisites}
 We assume some familiarity with elementary category theory (see e.g.,~\cite{MacLane1971}), cartesian closed categories (see e.g.,~\cite{Lambek86}), and in particular with symmetric monoidal categories and their string diagrams (see e.g.,~\cite{Selinger2010}). While knowledge of the theory of double categories would certainly be helpful, it is not strictly required, and we provide a brief technical introduction in Section~\ref{subsec:single-double} that covers everything we will need for our development. The sections concerning the double category of stateful transformations make heavy use of the notion of tensorial strength and the associated notion of strong natural transformation. We give the necessary definitions in Section~\ref{subsec:stateful-transformations}, but prior familiarity would, of course, be helpful (see e.g.,~\cite{Cockett1991}).

\section{The Free Cornering}\label{sec:corner-classical}
The aim of this section is to introduce the free cornering of a symmetric monoidal category. We begin by recapitulating some basic double category theory in the single-object case, which occupies Section~\ref{subsec:single-double}. This done, in Section~\ref{subsec:free-cornering} we recapitulate the free cornering construction and its interactive interpretation. In Section~\ref{subsec:crossing-cells} we recall the \emph{crossing cells} of the free cornering. We recall certain important properties of the crossing cells, the continuing validity of which will serve as a kind of litmus test for the soundness of our notions of choice and iteration in the sequel. Finally, in Section~\ref{subsec:stateful-transformations} we construct a double category of \emph{stateful transformations} over a cartesian closed category, and moreover show that it is a model of the free cornering of that category, possessing corner cells and crossing cells in an interesting fashion. This will serve as a running example throughout the paper. 

Before we begin, we must briefly discuss strictness and notation. We write composition of arrows in a category in \emph{diagrammatic order}. That is, the composite of $f : A \to B$ and $g : B \to C$ is written $fg : A \to C$. While we may write $g \circ f : A \to C$, we will \emph{never} write $gf : A \to C$. Moreover, in this paper we consider only \emph{strict} monoidal categories, and in our development the term ``monoidal category'' should be read as ``strict monoidal category''. That said, we imagine that our results will hold in some form for arbitrary monoidal categories via the coherence theorem for monoidal categories~\cite{MacLane1971}. Similarly, our double categories are what some authors call \emph{strict double categories}. The braiding maps in a symmetric monoidal category will be written $\sigma_{A,B} : A \otimes B \to B \otimes A$. Further notational conventions will be introduced as needed.

\subsection{Single-Object Double Categories}\label{subsec:single-double}
In this section we set up the rest of our development by recalling the theory of single-object double categories, being those double categories $\D$ with exactly one object. In this case $\D$ consists of a \emph{horizontal edge monoid} $\D_H = (\D_H, \otimes, I)$, a \emph{vertical edge monoid} $\D_V = (\D_V, \otimes, I)$, and a collection of \emph{cells}
\[
\includegraphics[height=1.7cm,align=c]{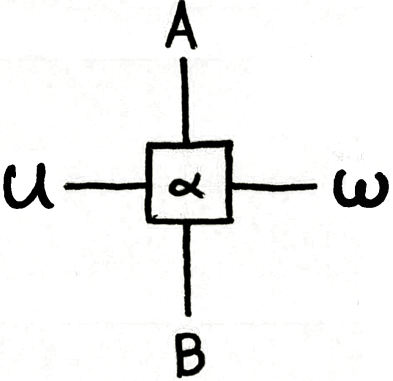}
\]
where $A,B \in \D_H$ and $U,W \in \D_V$. We write $\D\cells{U}{A}{B}{W}$ for the \emph{cell-set} of all such cells in $\D$, and write $\alpha : \D\cells{U}{A}{B}{W}$ to indicate the membership of $\alpha$ in a cell-set. When $\D$ is clear from context, we write $\cells{U}{A}{B}{W}$ instead of $\D\cells{U}{A}{B}{W}$. Given cells $\alpha : \cells{U}{A}{B}{V}$ and $\beta : \cells{V}{A'}{B'}{W}$ for which the right boundary of $\alpha$ matches the left boundary of $\beta$ we may form a cell $\alpha \vert \beta : \cells{U}{A \otimes A'}{B \otimes B'}{W}$ -- their \emph{horizontal composite} -- and similarly if the bottom boundary of $\alpha : \cells{U}{A}{C}{W}$ matches the top boundary of $\beta : \cells{U'}{C}{B}{W'}$ we may form $\frac{\alpha}{\beta} : \cells{U \otimes U'}{A}{B}{W \otimes W'}$ -- their \emph{vertical composite} -- with the boundaries of the composite cell formed from those of the component cells using the binary operation associated with the appropriate monoid (both written $\otimes$). We depict horizontal and vertical composition, respectively, as in:
\begin{mathpar}
\includegraphics[height=1.7cm,align=c]{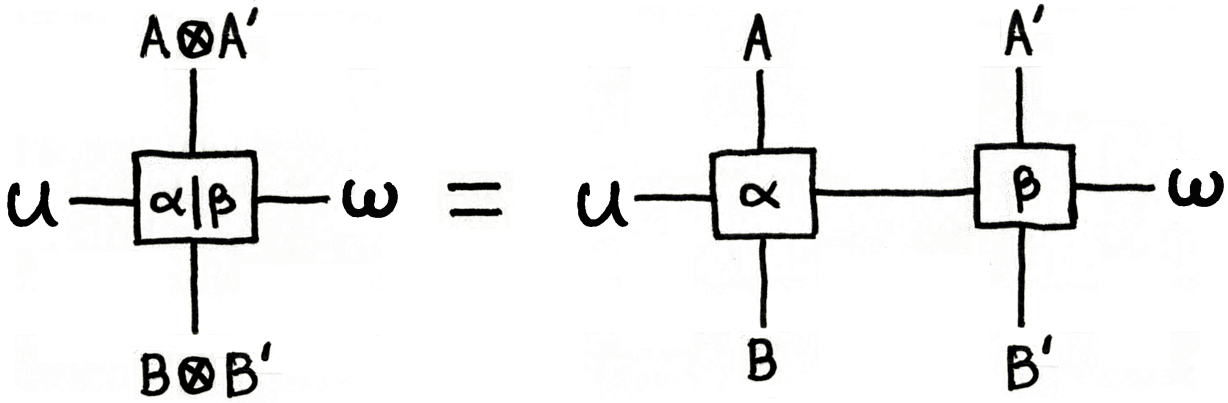}

\text{and}

\includegraphics[height=2.3cm,align=c]{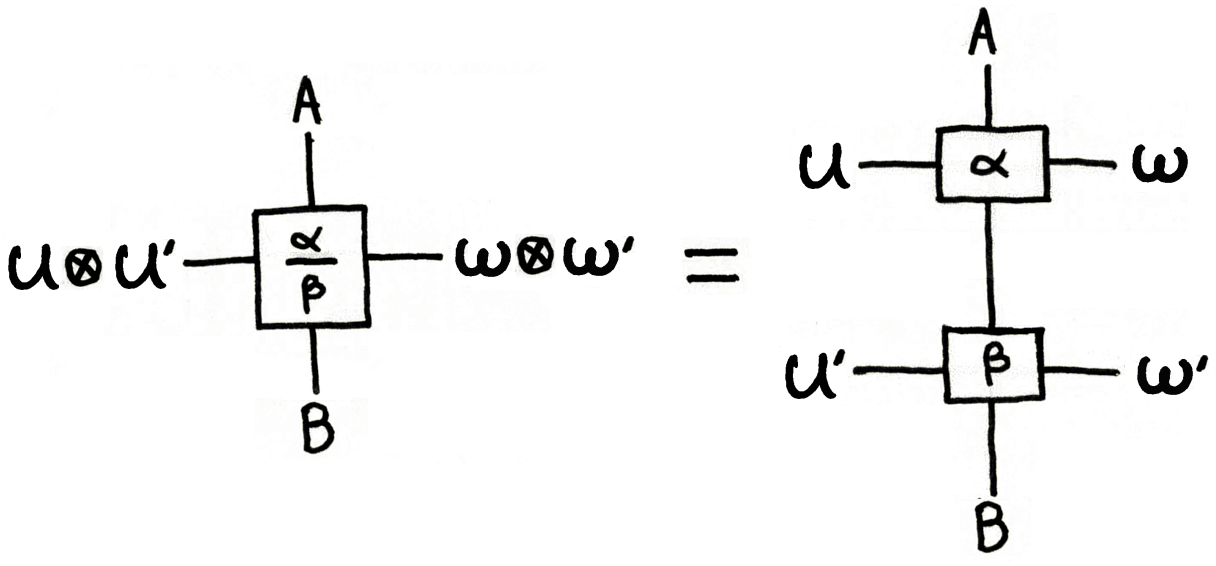}
\end{mathpar}
Horizontal and vertical composition of cells are required to be associative and unital. We write $id_U : \cells{U}{I}{I}{U}$ and $1_A : \cells{I}{A}{A}{I}$ for units of horizontal and vertical composition, respectively. We omit wires of sort $I$ in our depictions of cells, allowing us to depict horizontal and vertical identity cells, respectively, as in:
\begin{mathpar}
\includegraphics[height=1.7cm,align=c]{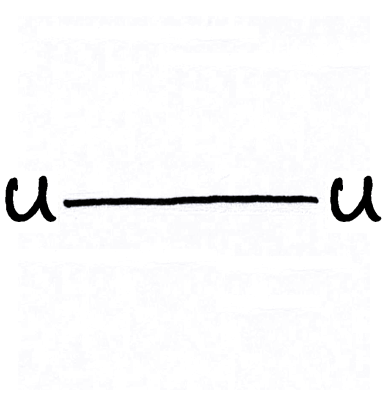}

\text{and}

\includegraphics[height=1.7cm,align=c]{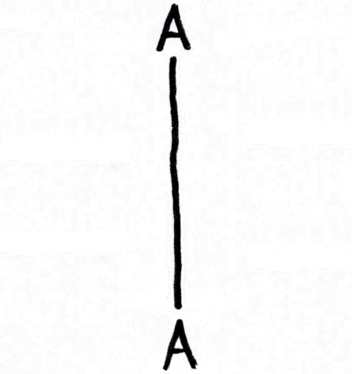}
\end{mathpar}
Finally, the horizontal and vertical identity cells of type $I$ must coincide -- we call this cell $\square_I = 1_I = id_I : \cells{I}{I}{I}{I}$ and depict it as empty space, see below on the left -- and vertical and horizontal composition must satisfy the interchange law. That is, $\frac{\alpha}{\beta}\vert \frac{\gamma}{\delta} = \frac{\alpha \vert \gamma}{\beta \vert \delta}$, allowing us to unambiguously interpret the diagram below on the right:
\begin{mathpar}
\includegraphics[height=1.7cm,align=c]{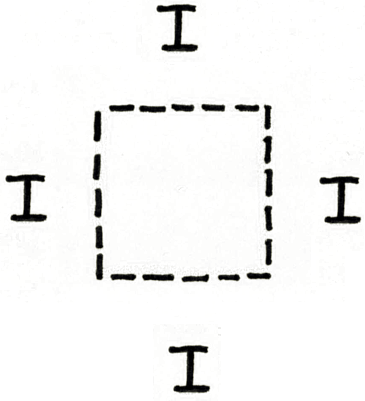}

\phantom{\text{and}}

\includegraphics[height=2.3cm,align=c]{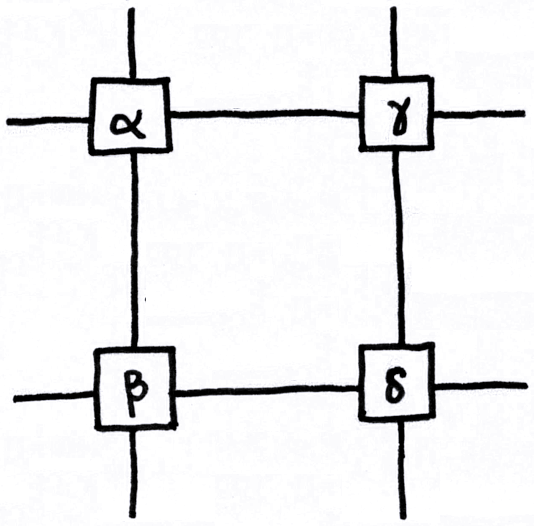}
\end{mathpar}

Every single-object double category $\D$ defines monoidal categories $\bv \D$ and $\bh \D$, consisting of the cells for which the $\D_V$ and $\D_H$ valued boundaries respectively are all $I$, as in:
\begin{mathpar}
\includegraphics[height=1.7cm,align=c]{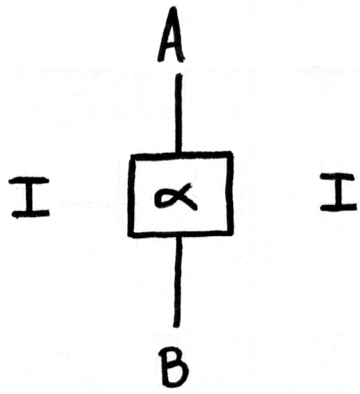}

\text{and}

\includegraphics[height=1.7cm,align=c]{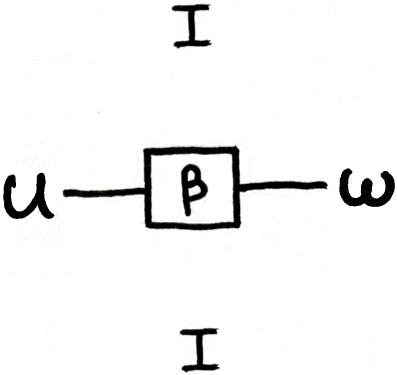}
\end{mathpar}
That is, the collection of objects of $\bv \D$ is $\D_H$, composition in $\bv \D$ is vertical composition of cells, and the tensor product in $\bv \D$ is given by horizontal composition:
\[
\includegraphics[height=1.7cm]{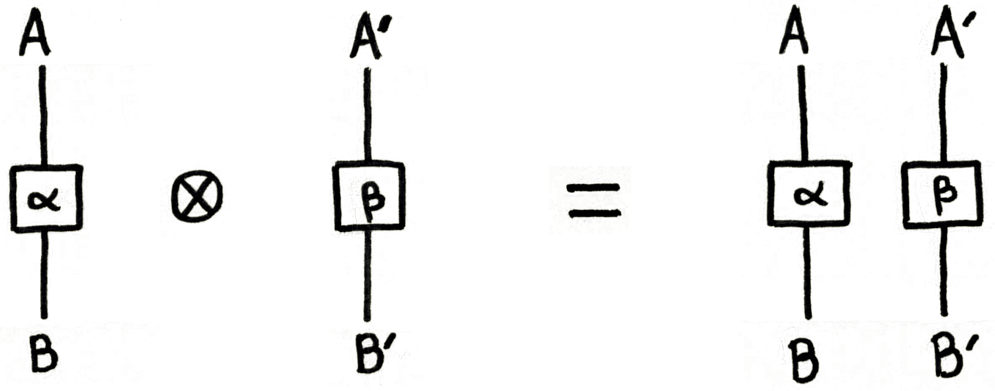}
\]
In this way, $\bv \D$ forms a monoidal category, which we call the category of $\emph{vertical cells}$ of $\D$. Similarly, $\bh \D$ is also a monoidal category (with collection of objects $\D_V$) which we call the \emph{horizontal cells} of $\D$.

\subsection{The Free Cornering}\label{subsec:free-cornering}
In this section we introduce the free cornering of a symmetric monoidal category. It is useful to frame this construction in terms of the resource-theoretic understanding of symmetric monoidal categories~\cite{Coecke14}. That is, objects are understood of as collections of resources. The tensor product $A \otimes B$ of two objects is the collection consisting of $A$ and $B$, and the unit $I$ is the empty collection. Morphisms are understood as \emph{transformations}, with $f : A \to B$ being understood as a way to transform the resources of $A$ to the resources of $B$. We adopt this perspective here and use the associated vocabulary to elucidate our development. 

We begin with the monoid of exchanges over a symmetric monoidal category:
\begin{definition}\label{def:monoid-exchanges} Let $\A$ be a symmetric monoidal category. Define the monoid $\ex{\A}$ of \emph{$\A$-valued exchanges} to be the free monoid on the set of polarized objects of $\A$, as in $\ex{\A} = (\A_0 \times \{ \circ, \bullet \})^*$. Explicitly, $\ex{\A}$ has elements given by:
  \begin{mathpar}
    \inferrule{A \in \A_0}{A^\circ \in \ex{\A}}

    \inferrule{A \in \A_0}{A^\bullet \in \ex{\A}}

    \inferrule{\text{}}{I \in \ex{\A}}

    \inferrule{U \in \ex{\A} \\ W \in \ex{\A}}{U \otimes W \in \ex{\A}}
  \end{mathpar}
  subject to the following equations:
  \begin{mathpar}
  I \otimes U = U

  U \otimes I = U

  (U \otimes W) \otimes V = U \otimes (W \otimes V)
  \end{mathpar}
  We may omit brackets as in $A^\circ \otimes B^\circ \otimes C^\bullet$, as associativity of $\otimes$ ensures that this denotes an element of $\ex{\A}$ unambiguously.
\end{definition}
The $\A$-valued exchanges are interpreted as follows: each $X_1 \otimes \cdots \otimes X_n \in \ex{\A}$ involves a left participant and a right participant giving each other resources in sequence, with $A^\circ$ indicating that the left participant should give the right participant an instance of $A$, and $A^\bullet$ indicating the opposite. For example say the left participant is $\alice$ and the right participant is $\bob$. Then we can picture the exchange $A^\circ \otimes B^\bullet \otimes C^\bullet$ as:
\[
\alice \rightsquigarrow
\includegraphics[height=1cm,align=c]{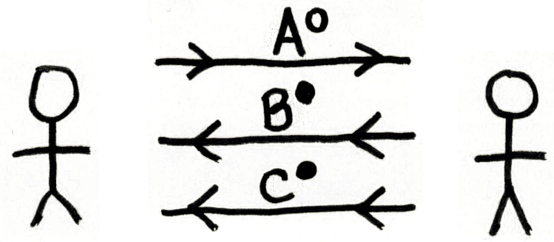}
\leftsquigarrow \bob
\]
These exchanges happen \emph{in order}. For example the exchange pictured above demands that first $\alice$ gives $\bob$ an instance of $A$, then $\bob$ gives $\alice$ an instance of $B$, and then finally $\bob$ gives $\alice$ an instance of $C$.

The monoid of $\A$-valued exchanges plays an important role in the free cornering of $\A$, which we introduce presently:
\begin{definition}[\cite{Nester2021a}]\label{def:free-cornering}
  Let $\A$ be a monoidal category. We define the \emph{free cornering} of $\A$, written $\corner{\A}$, to be the free single-object double category with horizontal edge monoid $(\A_0, \otimes, I)$, vertical edge monoid $\ex{\A}$, and generating cells and equations consisting of:
  \begin{itemize}
  \item For each $f : A \to B$ of $\A$ a cell $\corner{f} : \corner{\A}\cells{I}{A}{B}{I}$ subject to equations:
  \begin{mathpar}
    \corner{fg} = \frac{\corner{f}}{\corner{g}}

    \corner{1_A} = 1_A
    
    \corner{f \otimes g} = \corner{f} \mid \corner{g}
  \end{mathpar}
  One way to understand these is to notice that they allow us to interpret string diagrams denoting morphisms of $\A$ as cells of $\corner{\A}$ unambiguously:
  \begin{mathpar}
    \includegraphics[height=1cm,align=c]{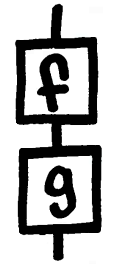}
    
    \includegraphics[height=1cm,align=c]{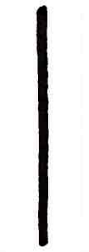}

    \includegraphics[height=1cm,align=c]{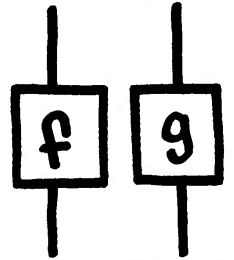}
  \end{mathpar}
  We write $\corner{f} = f$ when it is clear in context that $f$ denotes a cell of $\corner{\A}$. 
  \item For each object $A$ of $\A$, \emph{corner cells} $\getl^A : \corner{\A}\cells{A^\circ}{I}{A}{I}$, $\putr^A : \corner{\A}\cells{I}{A}{I}{A^\circ}$, $\getr^A : \corner{\A}\cells{I}{I}{A}{A^\bullet}$, and $\putl^A : \corner{\A}\cells{A^\bullet}{A}{I}{I}$, which we depict as follows:
  \[
  \includegraphics[height=1.7cm,align=c]{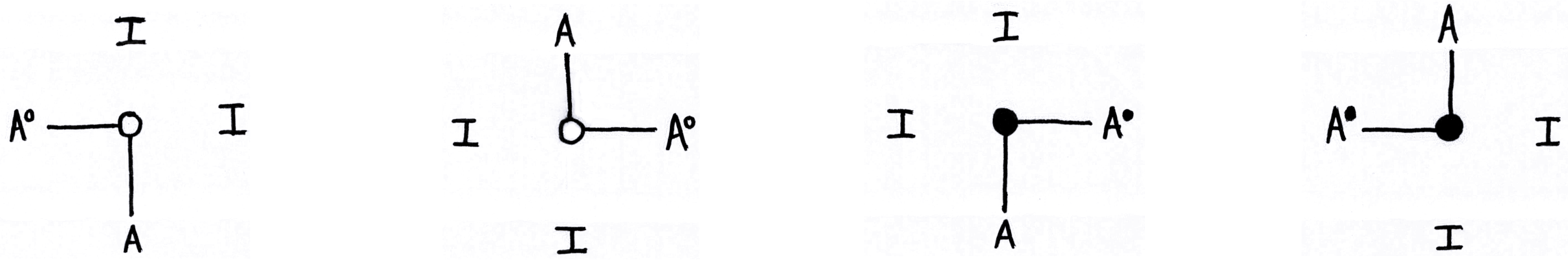}
  \]
  The corner cells are subject to the \emph{yanking equations}:
  \[
  \includegraphics[height=1cm,align=c]{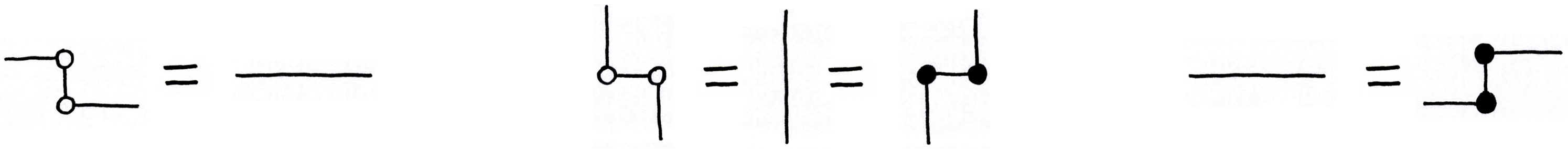}
  \]
  Intuitively, the corner cells send and receive resources along the left and right boundaries. The yanking equations allow us to carry out exchanges between horizontally composed cells, and tell us that being exchanged has no effect on the resources involved. 
  \end{itemize}
\end{definition}
For a precise development of free double categories see~\cite{Fio08}. Briefly, cells are formed from the generating cells by horizontal and vertical composition, subject to the axioms of a double category in addition to any generating equations. The corner structure has been heavily studied under various names including \emph{proarrow equipment}, \emph{framed bicategory}, \emph{connection structure}, and \emph{companion and conjoint structure}. A good resource is the appendix of~\cite{Shulman2008}.
 
Cells of $\corner{\A}$ can be understood as \emph{interacting} morphisms of $\A$. Each cell is a method of obtaining the resources of bottom boundary from those of the top boundary by participating in $\A$-valued exchanges along the left and right boundaries in addition to using the resource transformations supplied by $\A$. For example, if the morphisms of $\A$ describe the procedures involved in baking bread, we might have the following cells of $\corner{\A}$:
\begin{mathpar}
\includegraphics[height=1.7cm,align=c]{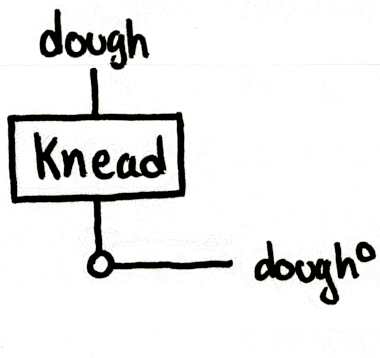}

\includegraphics[height=1.7cm,align=c]{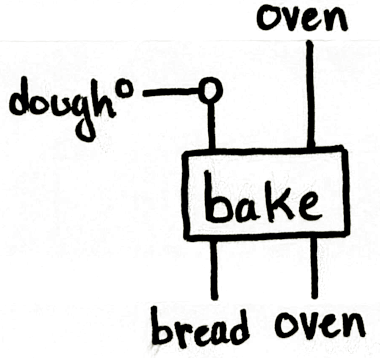}

\includegraphics[height=3cm,align=c]{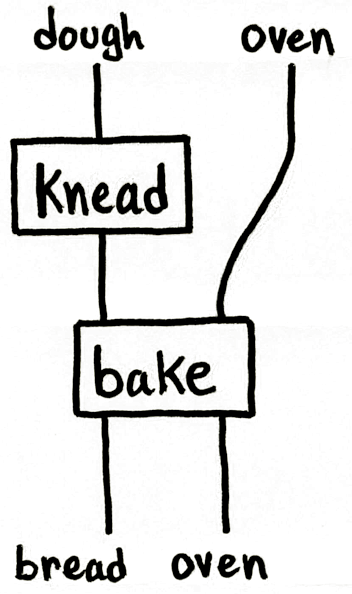}
\end{mathpar}
The cell on the left describes a procedure for transforming \texttt{dough} into nothing by \texttt{knead}ing it and sending the result away along the right boundary, and the cell in the middle describes a procedure for transforming an \texttt{oven} into \texttt{bread} and an \texttt{oven} by receiving \texttt{dough} along the left boundary and then using the \texttt{oven} to \texttt{bake} it. Composing these cells horizontally results in the cell on the right via the yanking equations. In this way the free cornering models process interaction, with the corner cells capturing the flow of information across components.

\subsection{Crossing Cells}\label{subsec:crossing-cells}
In this section we recall \emph{crossing cells}, an interesting bit of structure that exists in the free cornering of any symmetric monoidal category. We recall a few results concerning the crossing cells, which are quite well-behaved. Our purpose in doing so is mainly to extend these results later on when we add choice and iteration to the free cornering. The crossing cells will remain well-behaved, which is a sign that our notions of choice and iteration are formally coherent.

\begin{definition}[\cite{Nester2021a}]\label{def:crossing-cells}
  Let $\A$ be a symmetric monoidal category. For each $A \in \corner{\A}_H$ and each $U \in \corner{\A}_V$ we define a crossing cell $\chi_{U,A}$, pictured as in:
\[
\includegraphics[height=1.7cm]{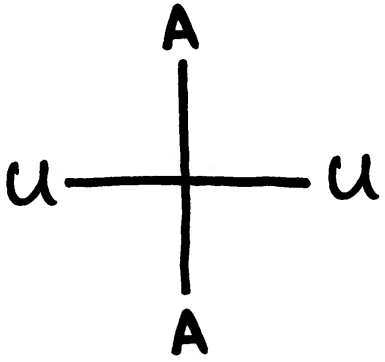}
\]
inductively as follows: define $\chi_{A^\circ,B}$ and $\chi_{A^\bullet,B}$ as in the diagrams below on the left and right, respectively:
\[
\includegraphics[height=1.7cm]{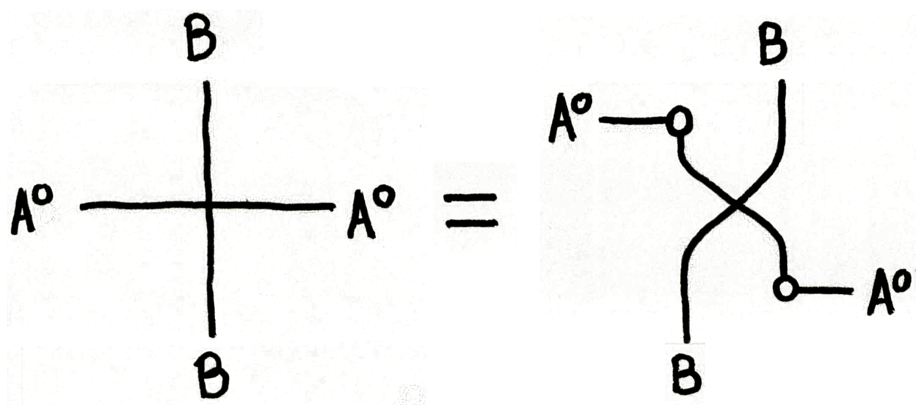}
\hspace{2cm}
\includegraphics[height=1.7cm]{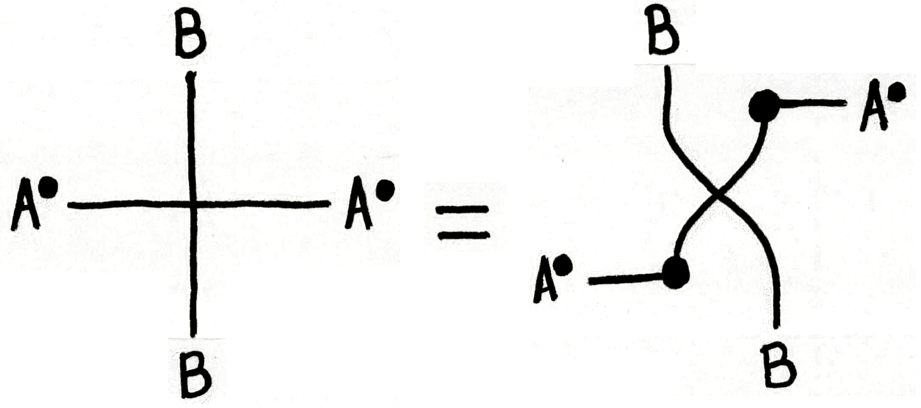}
\]
further, define $\chi_{I,A} = 1_A$ and $\chi_{U \otimes W,A} = \frac{\chi_{U,A}}{\chi_{W,A}}$, as in:
\begin{mathpar}
\includegraphics[height=1.6cm]{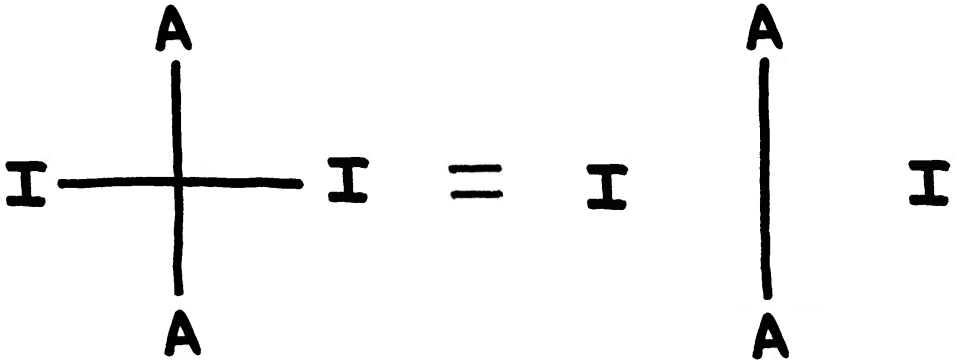}

\includegraphics[height=1.6cm]{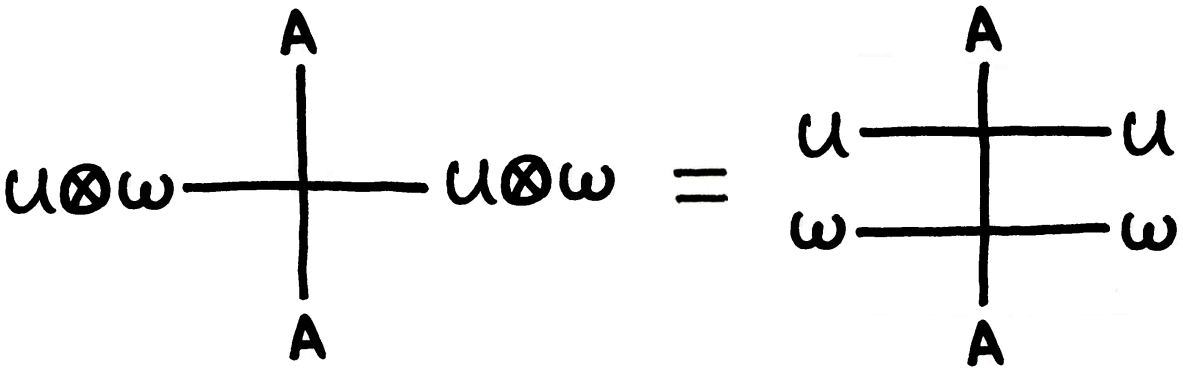}
\end{mathpar}
\end{definition}

We note that this definition differs slightly from that given in~\cite{Nester2021a, Nester2023}. In previous work on the free cornering, the definition of crossing cells included the assumption that they were coherent with respect to horizontal composition. We show that in fact, this can be derived:

\begin{lemma}\label{lem:coherent-crossing}
  Let $\A$ be a symmetric monoidal category. For $U \in \ex{A}$ and $A,B \in \A_0$ the following equations hold in $\corner{\A}$:
  \begin{enumerate}[(i)]
  \item $\chi_{U,{A \otimes B}} = \chi_{U,A} \mid \chi_{U,B}$
  \item $\chi_{U,I} = id_U$
  \end{enumerate}
\end{lemma}
\begin{proof}
  \begin{enumerate}[(i)]
    \item By structural induction on $U$. In case $U = C^\circ$ we have:
    \begin{mathpar}
      \includegraphics[height=1.6cm,align=c]{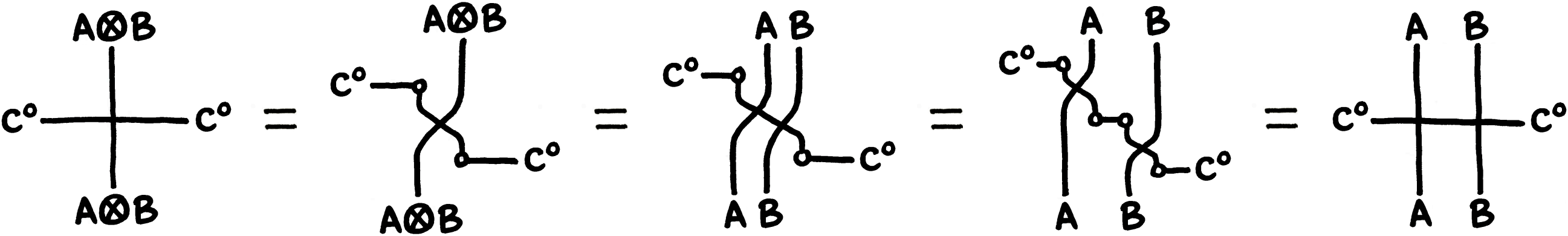}
    \end{mathpar}
    as required. The case for $U = C^\bullet$ is similar. If $U = I$ then we have:
    \[ \chi_{I,A \otimes B} = \corner{1_{A \otimes B}} = \corner{1_A} \mid \corner{1_B} = \chi_{I,A} \mid \chi_{I,B}\] 
    For the inductive case $U \otimes W$ we have
    \[ \chi_{U \otimes W,A \otimes B} = \frac{\chi_{U,A \otimes B}}{\chi_{W,A \otimes B}} = \frac{\chi_{U,A} \mid \chi_{U,B}}{\chi_{W,A} \mid \chi_{W,B}} = \frac{\chi_{U,A}}{\chi_{W,A}} \mid \frac{\chi_{U,B}}{\chi_{W,B}} = \chi_{U \otimes W,A} \mid \chi_{U \otimes W,B}\]
    The claim follows.
  \item By induction on the structure of $U$. If $U = A^\circ$ then we have:
    \begin{mathpar}
      \includegraphics[height=1.6cm,align=c]{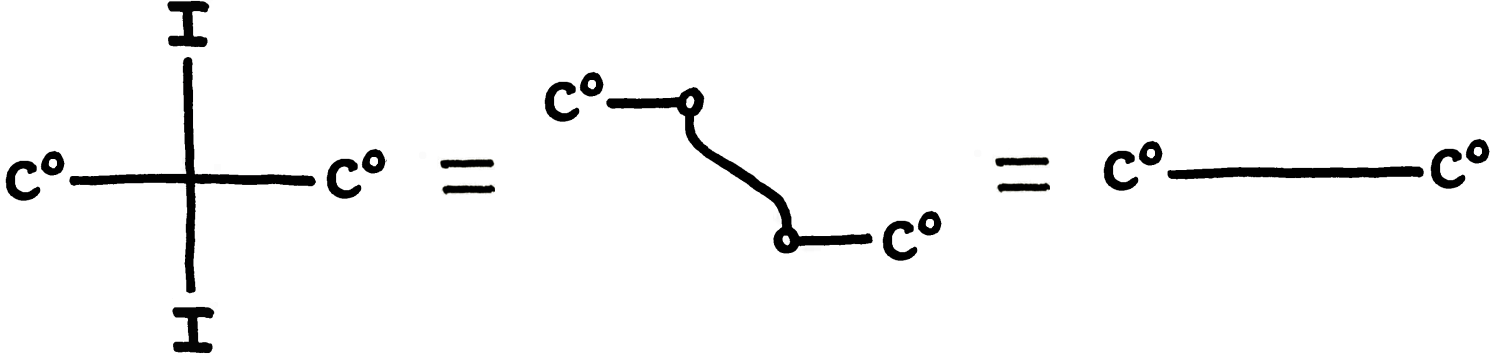}
    \end{mathpar}
    as required. The case for $U = A^\bullet$ is similar. If $U = I$ then we have:
    \[
    \chi_{U,I} = \chi_{I,I} = 1_I = id_I
    \]
    For $U \otimes W$, we have:
    \[
    \chi_{U \otimes W,I} = \frac{\chi_{U,I}}{\chi_{W,I}} = \frac{id_U}{id_W} = id_{U \otimes W}
    \]
    The claim follows.
  \end{enumerate}
\end{proof}

Additionally, the crossing cells carry interesting categorical structure. The core technical lemma underpinning this structure is as follows:
\begin{lemma}[\cite{Nester2021a}]\label{lem:crossing-swaps}
  For any cell $\alpha$ of $\corner{\A}$ we have
  \[
  \includegraphics[height=1.2cm]{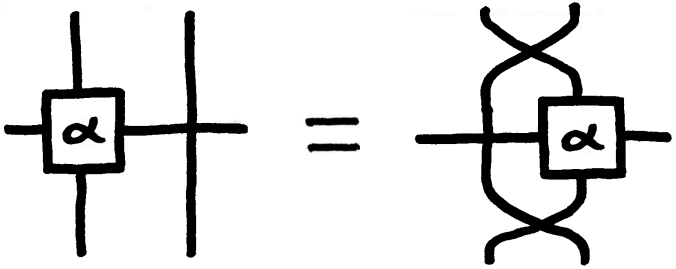}
  \]
\end{lemma}
We recapitulate the proof of this, as we will refer to it later on, when we extend the above lemma to the setting with choice and iteration.
\begin{proof}
  By structural induction on cells of $\corner{\A}$. For the $\circ$-corners we have:
  \[
  \includegraphics[height=1.2cm]{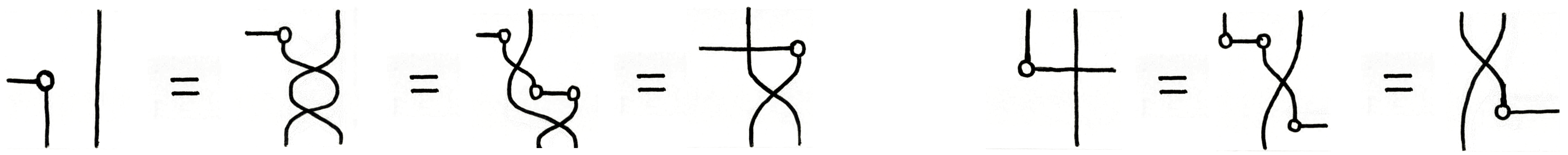}
  \]
  and for the $\bullet$-corners, similarly:
  \[
  \includegraphics[height=1.2cm]{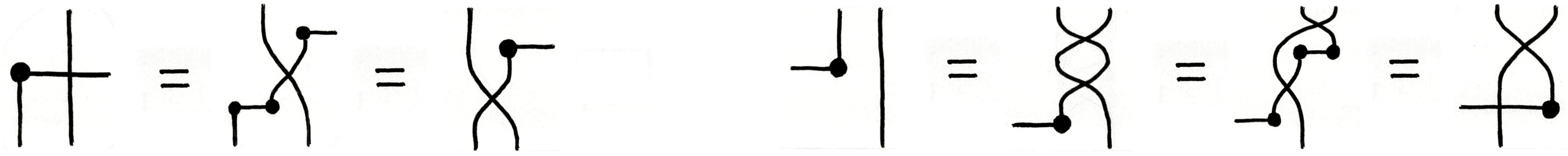}
  \]
  the final base cases are the $\corner{f}$ maps:
  \[
  \includegraphics[height=1.3cm]{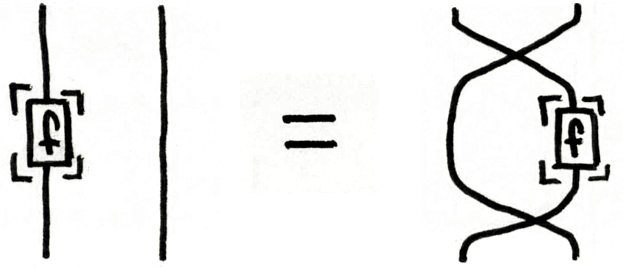}
  \]
  There are two inductive cases. For vertical composition, we have:
  \[
  \includegraphics[height=2.4cm,align=c]{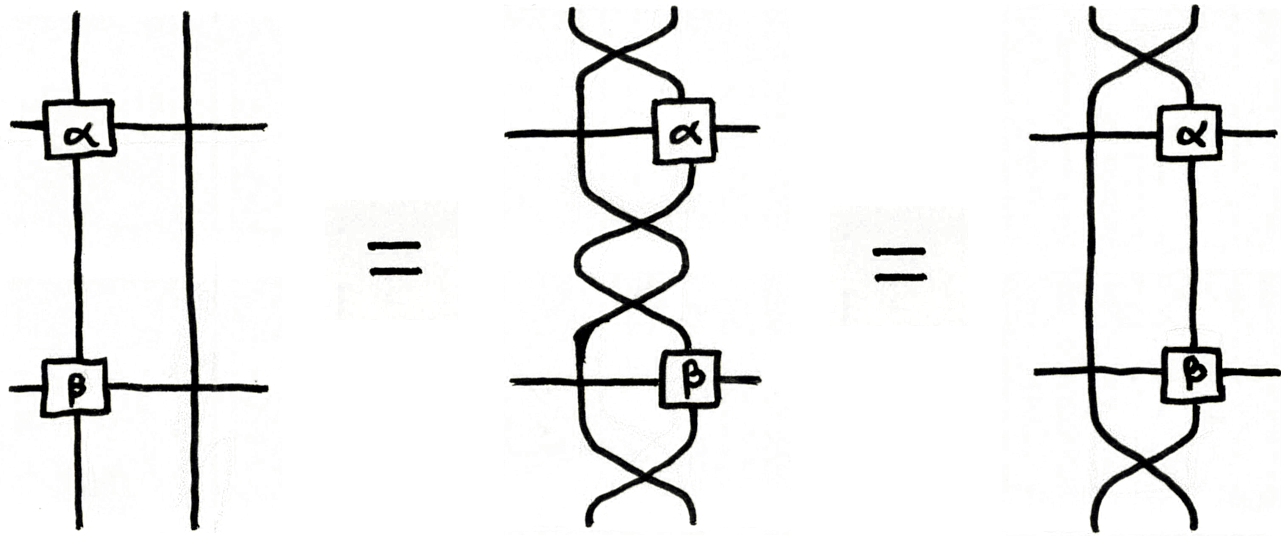}
  \]
  Horizontal composition is similarly straightforward, and the claim follows by induction.
\end{proof}

A particularly interesting consequence of Lemma~\ref{lem:crossing-swaps} is that for any symmetric monoidal category $\A$, $\corner{\A}$ is a monoidal double category in the sense of Shulman~\cite{Shulman2010}. That is, a pseudo-monoid object in the strict 2-category $\bv\mathsf{DblCat}$ of double categories, lax double functors, and vertical transformations. 
\begin{lemma}[\cite{Nester2021a}]\label{lem:monoidaldouble}
  If $\A$ is a symmetric monoidal category then $\corner{\A}$ is a monoidal double category. 
\end{lemma}
\begin{proof}
  We give the action of the tensor product on cells:
  \[
    \includegraphics[height=1.7cm]{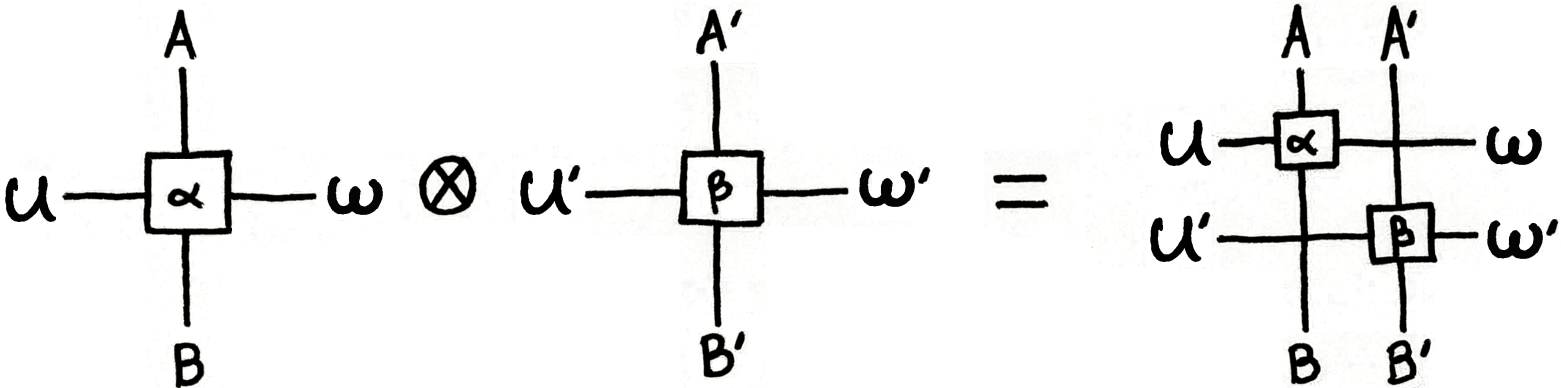}
  \]
  This defines a pseudofunctor, with the component of the required vertical transformation given by exchanging the two middle wires as in:
  \[
    \includegraphics[height=2.5cm]{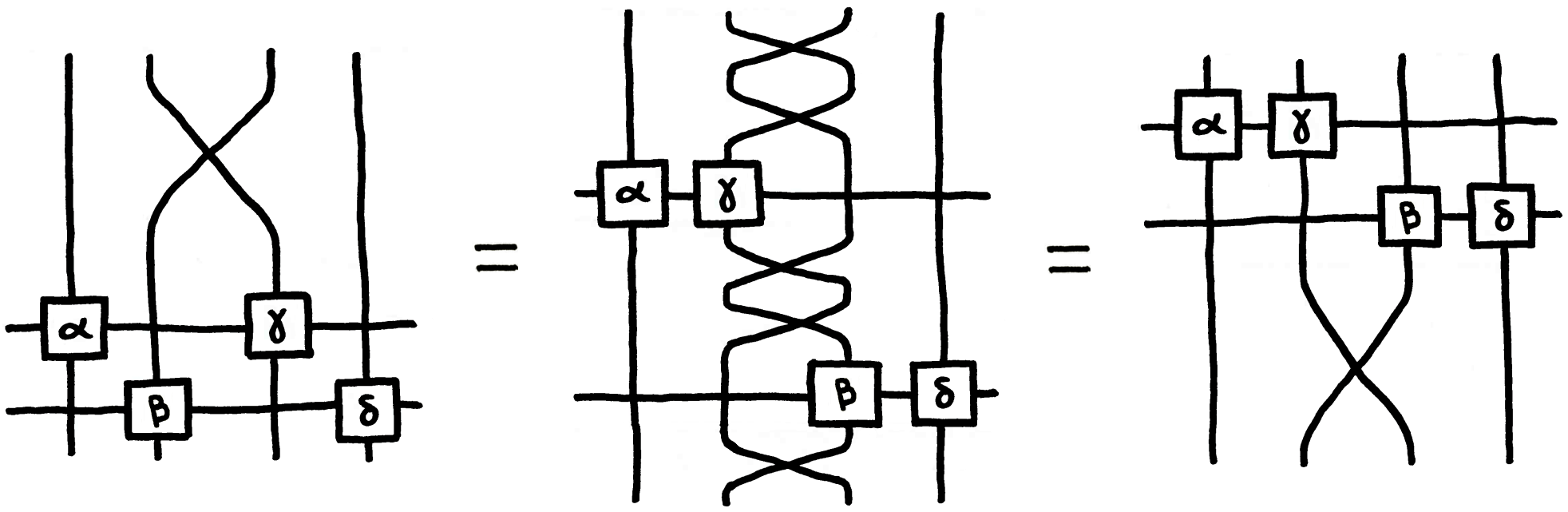}
  \]
  Notice that $\otimes$ is strictly associative and unital, in spite of being only pseudo-functorial. 
\end{proof}

This concludes our treatment of crossing cells in the free cornering. We proceed to give a model of the free cornering that we have recovered from the mathematical wilderness. 
\subsection{A Model: Stateful Transformations}\label{subsec:stateful-transformations}
In this section we construct a single-object double category of \emph{stateful transformations}, named for their resemblance to the \emph{stateful runners} studied by Uustalu in the context of monadic computational effects~\cite{Uustalu2015}. Our interest in the double category of stateful transformations is that it gives a model of the free cornering, exemplifying the corner cells and crossing structure in a more familiar setting. Stateful transformations will serve as a running example throughout our development. First, we require the notion of strong  functor. Recall:
\begin{definition}\label{def:tensorial-strength}
  Let $(\C,\otimes,I)$ be a monoidal category, and let $F : \C \to \C$ be a functor. Then a \emph{tensorial strength for $F$} consists of a natural transformation:
  \[
  \tau^F_{X,Y} : FX \otimes Y \to F(X \otimes Y)
  \]
  satisfying $\tau^F_{X,I} = 1_{FX}$ and also:
  \begin{mathpar}
    \begin{tikzcd}
      FX \otimes Y \otimes Z \ar[d,"\tau^F_{X,Y} \otimes 1_Z"'] \ar[dr,"\tau^F_{X,Y \otimes Z}"] \\
      F(X \otimes Y) \otimes Z \ar[r,"\tau^F_{X \otimes Y,Z}"'] & F(X \otimes Y \otimes Z)
    \end{tikzcd}
  \end{mathpar}
  A \emph{strong functor} $(F,\tau^F) : \C \to \C$ consists of a functor $F : \C \to \C$ together with a tensorial strength $\tau^F$ for $F$. 
\end{definition}
Let $\C^\C_\tau$ be the collection of strong functors $\C \to \C$. Then $\C^\C_\tau$ forms a monoid $(\C^\C_\tau,\circ,I)$. Given two strong functors $(F,\tau^F),(G,\tau^G) : \C \to \C$ we define $(F,\tau^F) \circ (G,\tau^G) = (G \circ F, \tau^{G\circ F})$ where $\tau^{G\circ F}_{X,Y} = \tau^G_{FX,Y}G(\tau^F_{X,Y})$. The unit $I$ is given by the identity functor $1_\C$ with strength $\tau^{1_\C}_{X,Y} = 1_{X \otimes Y}$. We write $F$ instead of $(F,\tau^F)$ when confusion is unlikely. The accompanying notion of natural transformation is:
\begin{definition}\label{def:strong-transformation}
  Let $(\C,\otimes,I)$ be a monoidal category, and let $(F,\tau^F),(G,\tau^G) : \C \to \C$ be strong functors. A \emph{strong natural transformation} $\alpha : (F,\tau^F) \to (G,\tau^G)$ is a natural transformation $\alpha : F \to G$ satisfying:
  \begin{mathpar}
    \begin{tikzcd}
      FX \otimes Y \ar[r,"\alpha_X \otimes 1_Y"] \ar[d,"\tau^F_{X,Y}"'] & GX \otimes Y \ar[d,"\tau^G_{X,Y}"] \\
      F(X \otimes Y) \ar[r,"\alpha_{X\otimes Y}"'] & G(X \otimes Y)
    \end{tikzcd}
  \end{mathpar}
\end{definition}

We will be concerned with strong functors over a cartesian closed category\footnote{In fact, for the purposes of this section it suffices to assume that our category is merely \emph{monoidal closed}. We make the stronger assumption of cartesian closure for continuity with later sections in which it is truly necessary.}  $(\C,\otimes,I)$. We write $X^A$ for the exponential, $\mathsf{ev}^B_A : B^A \otimes A \to A$ for the evaluation maps, and $\lambda[f] : B \to C^A$ for the name of $f : B \otimes A \to C$. We reiterate that here monoidal structure is assumed to be strict, and that this includes the cartesian monoidal structure in cartesian closed categories. Given an object $A$ of $\C$, we define endofunctors $A^{\circ} = (- \otimes A)$ and $A^{\bullet} = (-)^A$ of $\C$. The tensorial strengths for $A^\circ$ and $A^\bullet$ have components as in:
\begin{mathpar}
  \tau^{A^\circ}_{X,Y} = 1_X \otimes \sigma_{A,Y} : X \otimes A \otimes Y \to X \otimes Y \otimes A

  \tau^{A^\bullet}_{X,Y} = \lambda[(1_{X^A} \otimes \sigma_{Y,A})(\mathsf{ev}^A_X \otimes 1_Y)] : X^A \otimes Y \to (X \otimes Y)^A 
\end{mathpar}
Note that $A^{\circ}$ is left adjoint to $A^{\bullet}$.

We may now assemble the double category of stateful transformations:
\begin{definition}\label{def:stateful-transformations}
  Let $(\C,\otimes,I)$ be a cartesian closed category. The single-object double category $\mathsf{S}(\C)$ of \emph{stateful transformations in $\C$} has horizontal edge monoid $(\C_0,\otimes,I)$ given by the cartesian product structure of $\C$, and has vertical edge monoid $(\C^\C,\circ,I)$ the monoid of strong endofunctors on $\C$. The cells $\alpha : \mathsf{S}(\C)\cells{U}{A}{B}{W}$ are strong natural transformations:
  \[
  \alpha : (A^\circ \circ U, \tau^{A^\circ \circ U}) \to (W \circ B^\circ,\tau^{W \circ B^\circ})
  \]
  so in particular the components are of the form:
  \[
  \alpha_X : UX \otimes A \to W(X \otimes B)
  \]
  For horizontal composition, if we have $\alpha : \cells{F}{A}{B}{G}$ and $\beta : \cells{G}{A'}{B'}{H}$ then their horizontal composite $(\alpha \mid \beta) : \cells{F}{A \otimes A'}{B \otimes B'}{H}$ is given by:
  \[
  (\alpha | \beta)_X = (\alpha_X \otimes 1_{A'})(\beta_{X \otimes B}) : FX \otimes A \otimes A' \to H(X \otimes B \otimes B')
  \]
  \noindent and the horizontal identity cells $id_F : \cells{F}{I}{I}{F}$ are given by:   
  \[
  (id_F)_X = 1_{FX} : FX \otimes I = FX \to FX = F(X \otimes I)
  \]

  For vertical composition, if we have $\alpha : \cells{F}{A}{B}{G}$ and $\beta : \cells{H}{B}{C}{K}$ then their vertical composite $\frac{\alpha}{\beta} : \cells{F \circ H}{A}{C}{G \circ K}$ is given by
  \[
  \left(\frac{\alpha}{\beta}\right)_X = \alpha_{HX}G(\beta_X) : F(H(X)) \otimes A \to G(K(X \otimes C))
  \]
  \noindent and the vertical identity cells $1_A : \cells{I}{A}{A}{I}$ are given by:
  \[
  (1_A)_X = 1_{X \otimes A} : 1_\C(X) \otimes A = X \otimes A \to X \otimes A = 1_\C(X \otimes A)
  \]
\end{definition}
We show that this is indeed a double category:
\begin{lemma}
  $\mathsf{S}(\C)$ is well-defined.
\end{lemma}
\begin{proof}
  Horizontal and vertical composition are associative and unital because both composition of natural transformations and the cartesian product structure are associative and unital. The rest of the requirements on a double category are easily seen to hold, with the most involved being interchange, which we show holds explicitly. Given $\alpha : \cells{F}{A}{B}{G}$, $\beta : \cells{G}{A'}{B'}{H}$, $\gamma : \cells{F'}{B}{C}{G'}$ and $\delta : \cells{G'}{B'}{C'}{H'}$, then the interchange law $(\frac{\alpha}{\gamma} \vert \frac{\beta}{\delta}) = (\frac{\alpha \vert \beta}{\gamma \vert \delta})$ holds as follows:
  \[
  \xymatrix@C=0em{
    &
    & 	F(F'(X)) \otimes A \otimes A' 
    \ar^{\alpha_{F'(X)}\otimes A'}[d] 
    \ar^{(\alpha | \beta)_{F'X}}@/^1.5pc/[ddr] 
    \ar_{(\frac{\alpha}{\gamma})_X \otimes A'}@/_1.5pc/[ddl]
    \ar `r[rr] `[dddd]^-{(\frac{\alpha | \beta}{\gamma | \delta})_X}  [dddd]
    \ar `l[ddll] `[dddd]_{(\frac{\alpha}{\gamma} | \frac{\beta}{\delta})_X}  [dddd]
    &	& \\
    &	& 	G(F'(X) \otimes B) \otimes A' 
    \ar^{G(\gamma_X) \otimes A'}[dl] 
    \ar_{\beta_{F'X \otimes B}}[dr] 
    &	& \\
    &	G(G'(X \otimes C)) \otimes A'
    \ar^{\beta_{G'(X \otimes C)}}[dr]
    \ar_{(\frac{\beta}{\delta})_{X \otimes C}}@/_1.5pc/[ddr]
    & 	\text{}
    & 	H(F'(X) \otimes B \otimes B')
    \ar_{H(\gamma_X \otimes B')}[dl]
    \ar^{H(\gamma | \delta)}@/^1.5pc/[ddl] 
    & \\
    &	& 	H(G'(X \otimes C) \otimes B') 
    \ar^{H(\delta_{X \otimes C})}[d]
    &	& \\
    &	& 	H(H'(X \otimes C \otimes C')) 
    &	&
  }
  \]
  where the middle diamond commutes by naturality of $\beta$ and the rest of the diagram is obtained by unfolding definitions. 
\end{proof}

A first observation concerning $\mathsf{S}(\C)$ is that for each $f : A \to B$ of $\C$ there is a cell $[f] : \mathsf{S}(\C)\cells{I}{A}{B}{I}$ given by $[f]_X : (1_X \otimes f) : X \otimes A \to X \otimes B$, and that this defines an embedding $[-] : \C \to \bv\,\mathsf{S}(\C)$. Moreover, the category of strong endofunctors of $\C$ and strong natural transformations embeds into $\bh\,\mathsf{S}(\C)$, since a strong natural transformation $\alpha : (F,\tau^F) \to (G,\tau^G)$ is equivalently a cell $\alpha : \mathsf{S}(\C)\cells{F}{I}{I}{G}$. 

We define $\circ$-corner cells $\putr^A : \mathsf{S}(\C)\cells{I}{A}{I}{A^\circ}$ and $\getl^A : \mathsf{S}(\C)\cells{A^\circ}{I}{A}{I}$ to have identity maps as components, as in:
\begin{mathpar}
  (\putr^A)_X = 1_{X \otimes A} : 1_\C(X) \otimes A = X \otimes A \to  X \otimes A =  A^{\circ}(X \otimes I)

  (\getl^A)_X = 1_{X \otimes A} :  A^{\circ}(X) \otimes I = X \otimes A \to X \otimes A = 1_\C(X \otimes A)
\end{mathpar}
That the yanking equations hold is immediate. Next, we define $\bullet$-corner cells $\putl^A : \mathsf{S}(\C)\cells{A^\bullet}{A}{I}{I}$ and $\getr^A : \mathsf{S}(\C)\cells{I}{I}{A}{A^\bullet}$ using the closed structure, as in:

\begin{mathpar}
(\putl^A)_X = \ev^A_X : A^{\bullet}(X) \otimes A = X^A \otimes A \to X = 1_\C(X \otimes I)
  
(\getr^A)_X = \lambda[1_{X \otimes A}] : 1_\C(X) \otimes I = X \to (X \otimes A)^A = A^{\bullet}(X \otimes A)
\end{mathpar}
In other words, $(\putl^A)_X : A^{\circ}(A^{\bullet}X) \to X$ is and $(\getr^A)_X : X \to A^{\bullet}(A^{\circ}X)$ are the unit and counit of the adjunction $A^{\circ} \dashv A^{\bullet}$ given by the cartesian closed structure. The yanking equations are then the triangle equations of this adjunction. Explicitly:
\begin{mathpar}
  (\getr^A \mid \putl^A)_X = (\lambda[1_{X \otimes A}] \otimes 1_A) \ev^A_{X \otimes A} = 1_{X \otimes A} = (id_{A^\bullet})_X
  
  \left(\frac{\getr^A}{\putl^A}\right)_X = \lambda[1_{X^A \otimes A}] (\ev^A_X)^A = \lambda[1_{X^A \otimes A} \ev^A_X] = 1_{X^A} = (1_A)_X
\end{mathpar}
That the corner cells constitute strong natural transformations is straightforward, if slightly tedious, to verify. It follows that $\mathsf{S}(\C)$ is a proarrow equipment.

The fact that $\mathsf{S}(\C)$ is constructed from \emph{strong} functors is closely connected to the nature of crossing cells there. Given a strong functor $(F,\tau^F) : \C \to \C$ and an object $A$ of $\C$ the tensorial strength of $F$ defines a crossing cell $\chi_{F,A} : \mathsf{S}(\C)\cells{F}{A}{A}{F}$ with components $(\chi_{F,A})_X = \tau^F_{X,A} : FX \otimes A \to F(X \otimes A)$. These crossing cells are coherent with respect to both horizontal and vertical composition in $\mathsf{S}(\C)$ in the sense that:
\begin{mathpar}
  \chi_{F,I} = id_F

  \chi_{F,A \otimes B} = \chi_{F,A} \mid \chi_{F,B}

  \chi_{I,A} = 1_A

  \chi_{F \circ G,A} = \frac{\chi_{F,A}}{\chi_{G,A}}
\end{mathpar}
The crossing cells $\chi_{B^\circ,A}$ and $\chi_{B^\bullet,A}$ obtained in this manner are equal to those defined in terms of the corner cells and braiding, as in Definition~\ref{def:crossing-cells}.

Moreover, the equation from Lemma \ref{lem:crossing-swaps} concerning crossing cells:
\[
\includegraphics[height=1.2cm]{figs/crossing-swaps.png}
\]
holds for all cells $\alpha$ of $\mathsf{S}(\C)$ precisely because cells $\alpha$ are required to be \emph{strong} natural transformations. In particular, this means that $\mathsf{S}(\C)$ is also a monoidal double category.

We have seen that the double category of stateful transformations has corner and crossing cells, much as the free cornering does. We have previously mentioned that stateful transformations give a \emph{model} of the free cornering. What we mean by this is that there is a structure-preserving double functor from the free cornering of a cartesian closed category into the associated category of stateful transformations. That is, we have:
\begin{lemma}\label{lem:doublefunctor}
  Let $(\C,\otimes,I)$ be a cartesian closed category. Then we have:
  \begin{enumerate}[(i)]
  \item $\mathsf{S}(\C)$ is a proarrow eqipment.
  \item $\mathsf{S}(\C)$ is a monoidal double category. 
  \item There is a double functor $D : \corner{\C} \to \mathsf{S}(\C)$ defined as follows: on the horizontal edge monoid $D$ acts as the identity; on the vertical edge monoid $D$ sends $A^\circ$ and $A^\bullet$ in $\ex{\C}$ to the strong functors $A^\circ$ and $A^\bullet$ in $\C^\C_\tau$ detailed above, and is otherwise defined as in $D(I) = I$ and $D(U \otimes W) = D(U) \otimes D(W)$; on cells $D$ acts on the morphisms $\corner{f}$ as in $D(\corner{f}) = [f]$, and on the corner cells as in:
  \begin{mathpar}
    D(\putr^A) = \putr^A

    D(\getl^A) = \getl^A
    \\
    D(\putl^A) = \putl^A

    D(\getr^A) = \getr^A
  \end{mathpar}
    Moreover, $D : \corner{\C} \to \mathsf{S}(\C)$ preserves the proarrow equipment structure and monoidal double category structure, in the sense that it maps corner cells and crossing cells in $\corner{\C}$ to corner cells and crossing cells in $\mathsf{S}(\C)$. 
  \end{enumerate}
\end{lemma}\qed

\begin{remark}\label{rem:effects-interpretation}
Strong functors, and in particular strong monads, are often used to model computational effects (see e.g.,~\cite{Moggi91,Plotkin01}). For example $A^\bullet$ is sometimes called the \emph{reader monad}: arrows $f : B \to C$ in the Kleisli category are given by arrows $f : B \to C^A = A^\bullet(C)$, which are understood as arrows $B \to C$ that \emph{read input} of type $A$. This input is understood to be provided by the environment of the program. Moreover, $A^\circ$ is the \emph{writer comonad}, $A^\bullet \circ A^\circ$ is the \emph{state monad}, and $A^\circ \circ A^\bullet$ is the \emph{store comonad}.

One way to think of cells $\alpha$ of $\mathsf{S}(\C)$ from the perspective of computational effects is as follows: the right boundary of $\alpha$ represents its \emph{environment}, that is, the context in which $\alpha$ executes. For example if the right boundary of $\alpha$ is $A^\bullet$ then $\alpha$ will read a value (supplied by the environment). The left boundary of $\alpha$ represents the \emph{interior} of $\alpha$, in the sense that $\alpha$ acts as the environment of its interior. For example if the left boundary of $\alpha$ is $A^\bullet$ then the interior of $\alpha$ will read a value, which $\alpha$ must supply. Here effects are understood to be triggered from the left, propagating outwards until resolved. 
\end{remark}

This concludes our discussion of $\mathsf{S}(\C)$ for the time being. We proceed to discuss the addition of choice to the free cornering.
\section{Adding Choice to the Free Cornering}\label{sec:corner-choice}
In this section we extend the free cornering of a symmetric monoidal category with a notion of protocol choice. In addition to symmetric monoidal structure we will require the base category to have distributive binary coproducts, which we review in Section~\ref{subsec:distributive-branching}. We also discuss the way in which this sort of category can be seen as an algebra of sequential branching programs. In Section~\ref{subsec:free-choice} we construct the \emph{free cornering with choice} over a suitable base and discuss its interpretation. In Section~\ref{subsec:choice-properties} we establish a number of elementary properties of the free cornering with choice. In Section~\ref{subsec:choice-crossing} we define crossing cells in the free cornering with choice, and show that they are well-behaved. This is mostly an extension of Section~\ref{subsec:crossing-cells} to the new setting, with the exception of Lemma~\ref{lem:coherent-oplus-choice}. Finally, in Section~\ref{subsec:stateful-choice} we show that with additional assumptions on the base category the double category of stateful transformations from Section~\ref{subsec:stateful-transformations} gives a model of the free cornering with choice. 

\subsection{Distributive Monoidal Categories and Branching Programs}\label{subsec:distributive-branching}
We begin by recapitulating the notion of a category with binary coproducts, largely in order to establish our notation for them:
\begin{definition}
  A category $\A$ is said to \emph{have binary coproducts} in case for each pair $A,B$ of objects of $\A$ there is is an object $A \oplus B$ of $\A$ together with morphisms $\sigma^{A,B}_0 : A \to A \oplus B$ and $\sigma^{A,B}_1 : B \to A \oplus B$ such that for any pair of morphims $f : A \to C$ and $g : B \to C$ there exists a unique morphism $[f,g] : A \oplus B \to C$ with the property that $\sigma^{A,B}_0[f,g] = f$ and $\sigma_1^{A,B}[f,g] = g$. We call $A \oplus B$ the \emph{coproduct of $A$ and $B$}, and call $[f,g]$ the \emph{copairing of $f$ and $g$}. We write $\sigma_0^{A,B} = \sigma_0$ and $\sigma_1^{A,B} = \sigma_1$ when it is unlikely to result in confusion. Note that a category with binary coproducts need not have an initial object.
\end{definition}

Next, we recall the notion of a distributive monoidal category, being a monoidal category with distributive binary coproducts:
\begin{definition}
  A \emph{distributive monoidal category} $(\A,\otimes,\oplus,I)$ is a symmetric monoidal category $(\A,\otimes,I)$ with binary coproducts $A \oplus B$ such that $\otimes$ distributes over $\oplus$. That is, for all objects $A,B,C$ of $\A$ the arrow  $\mu^r = [(\sigma_0 \otimes 1_C),(\sigma_1 \otimes 1_C)] : (A \otimes C) \oplus (B \otimes C) \to (A \oplus B) \otimes C$ has an inverse $\delta^r : (A \oplus B) \otimes C \to (A \otimes C ) \oplus (B \otimes C)$. Diagrammatically:
    \begin{mathpar}
      \includegraphics[height=1.6cm,align=c]{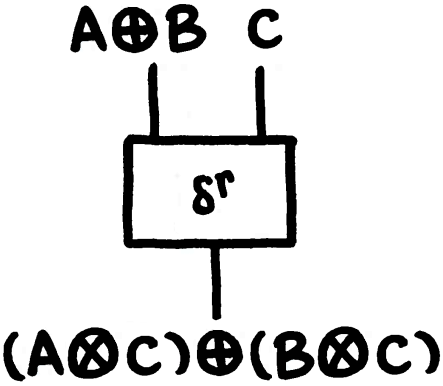}

      \includegraphics[height=1.6cm,align=c]{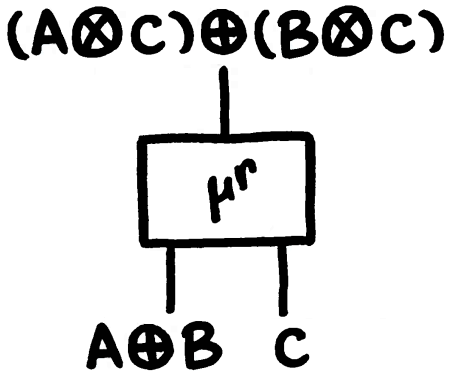}
    \end{mathpar}
Note that in any distributive monoidal category there is necessarily an inverse $\delta^l : C \otimes (A \oplus B) \to (C \otimes A) \oplus (C \otimes B)$ to the arrow $\mu^l = [(1_C \otimes \sigma_0),(1_C \otimes \sigma_1)] : (C \otimes A) \oplus (C \otimes B) \to C \otimes (A \oplus B)$. 
\end{definition}

The resource-theoretic understanding of symmetric monoidal categories extends to distributive monoidal categories, with $A \oplus B$ understood as the collection consisting of the contents of $A$ \emph{or} the contents of $B$. Notice that this interpretation is only really coherent in the presence of distributivity. 

Another way to understand distributive monoidal categories is that they model \emph{branching programs}. In any distributive monoidal category we may define an object $\mathsf{Bool} = I \oplus I$ of booleans with elements $\top = \sigma_0 : I \to \mathsf{Bool}$ and $\bot = \sigma_1 : I \to \mathsf{Bool}$ given by the coproduct injections. Then for any $f,g : A \to B$ the morphism $\delta^r[f,g] : \mathsf{Bool} \otimes A \to B$ models the conditional statement:
\[\texttt{if } b \texttt{ then } f(x) \texttt{ else } g(x)\]
In particular we have both of (see Lemma~\ref{lem:injections-distribute}):
\begin{mathpar}
  (\top \otimes 1_A)\delta^r[f,g] = \sigma_0^{A,A}[f,g] = f

  (\bot \otimes 1_A)\delta^r[f,g] = \sigma_1^{A,A}[f,g] = g
\end{mathpar}
which we should think of as program equivalences:
\begin{mathpar}
  \texttt{if } \top \texttt{ then } f(x) \texttt{ else } g(x) = f(x)

  \texttt{if } \bot \texttt{ then } f(x) \texttt{ else } g(x) = g(x)
\end{mathpar}

Before moving on we record a useful fact about coproduct injections in distributive monoidal categories for later use:
\begin{lemma}\label{lem:injections-distribute}
  In distributive monoidal category:
  \begin{mathpar}
    (\sigma_0^{A,B} \otimes 1_C)\delta^r = \sigma_0^{A \otimes C,B \otimes C}

    \text{and}

    (\sigma_1^{A,B} \otimes 1_C)\delta^r = \sigma_1^{A\otimes C,B \otimes C}
  \end{mathpar}
\end{lemma}
\begin{proof}
  We have:
  \begin{align*}
    & \sigma_0^{A\otimes C,B \otimes C}\mu^r
    = \sigma_0^{A\otimes C,B \otimes C}[(\sigma_0^{A,B} \otimes 1_C),(\sigma_1^{A,B} \otimes 1_C)]
    = \sigma_0^{A,B} \otimes 1_C
  \end{align*}
  It follows immediately that:
  \begin{align*}
    & \sigma_0^{A \otimes C,B \otimes C}
    = \sigma_0^{A \otimes C,B \otimes C}\mu^r\delta^r
    = (\sigma_0^{A,B} \otimes 1_C)\delta^r
  \end{align*}
  Similarly, we have $(\sigma_1^{A,B} \otimes 1_C)\delta^r = \sigma_1^{A\otimes C,B \otimes C}$.
\end{proof}

Distributive monoidal categories are also a good place to model datatypes. For example, if $\A$ is distributive monoidal and $A$ is an object of $\A$, then we may model \emph{stacks of type $A$} as an object $S_A$ of $\A$ equipped with an isomorphism $S_A \cong I \oplus (A \otimes S_A)$ with components $\mathsf{pop} : S_A \to I \oplus (A \otimes S_A)$ and $[\mathsf{nil},\mathsf{push}] : I \oplus (A \otimes S_A) \to S_A$. Then for example the object $S_I$ of stacks of type $I$ is a model of the natural numbers with $\mathsf{nil} = \mathsf{zero}$ and $\mathsf{push} = \mathsf{succ}$. See~\cite{Walters1989} for a more in-depth discussion. 

\subsection{The Free Cornering With Choice}\label{subsec:free-choice}
In this section we extend the free cornering of a monoidal category with a notion of protocol choice. We begin by extending the monoid of exchanges (Definition~\ref{def:monoid-exchanges}) with binary operations $-+-$ and $-\times-$ representing branching protocols:
\begin{definition}\label{def:choice-exchanges}
  Let $\A$ be a symmetric monoidal category. The monoid $\ex{\A}_\oplus$ of \emph{$\A$-valued exchanges with choice} has elements generated by:
  \begin{mathpar}
    \inferrule{A \in \A_0}{A^\circ \in \ex{\A}_\oplus}
    
    \inferrule{A \in \A_0}{A^\bullet \in \ex{\A}_\oplus}

    \inferrule{\text{}}{I \in \ex{\A}_\oplus}

    \inferrule{U \in \ex{\A}_\oplus \\ W \in \ex{\A}_\oplus}{U \otimes W \in \ex{\A}_\oplus}
   
    \inferrule{U \in \ex{\A}_\oplus \\ W \in \ex{\A}_\oplus}{U \times W \in \ex{\A}_\oplus}
    
    \inferrule{U \in \ex{\A}_\oplus \\ W \in \ex{\A}_\oplus}{U + W \in \ex{\A}_\oplus}
  \end{mathpar}
  subject to the following equations:
  \begin{mathpar}
  I \otimes U = U

  U \otimes I = U

  (U \otimes W) \otimes V = U \otimes (W \otimes V)
  \end{mathpar}
\end{definition} 
We extend the interpretation of $\ex{\A}$ from Section~\ref{subsec:free-cornering} to an interpretation of $\ex{\A}_\oplus$, interpreting $-+-$ and $-\times-$ as \emph{choices}. Specifically, For any $U,W \in \ex{\A}_\oplus$ we interpret $U + W \in \ex{\A}_\oplus$ as an exchange which begins with the \emph{left} participant choosing whether the rest of the exchange will be of the form $U$ or of the form $W$, after which the exchange proceeds according to this choice. Dually, we interpret $U \times W \in \ex{\A}_\oplus$ in the same way, except that the \emph{right} participant chooses instead of the left participant. 

For example, suppose $A,B \in \A_0$. For each of the following exchanges, call the left participant $\alice$ and the right participant $\bob$, as before. Now, consider:
\begin{itemize}
\item To carry out $A^\circ \times A^\bullet$, first $\bob$ chooses which of $A^\circ$ and $A^\bullet$ will happen. If $\bob$ chooses $A^\circ$ then $\alice$ sends him an instance of $A$ and the exchange ends. If $\bob$ chooses $A^\circ$ then he sends $\alice$ an instance of $A$ and the exchange ends.

\item To carry our $(A^\circ \times A^\bullet) \otimes B^\bullet$, first $\alice$ and $\bob$ carry out $A^\circ \times A^\bullet$ as above, and then $\bob$ gives $\alice$ an instance of $B$. 
  
\item To carry out $A^\bullet + I$, first $\alice$ chooses which of $A^\bullet$ and $I$ will happen. If $\alice$ chooses $A^\bullet$ the $\bob$ sends her an instance of $A$ and the exchange ends. If $\alice$ chooses $I$ then the exchange ends immediately. 
\item To carry out $A^\circ + (A^\circ \times B^\circ) \in \ex{\A}_\oplus$, first $\alice$ chooses which of $A^\circ$ and $(A^\circ \times B^\circ)$ will happen. If $\alice$ chooses $A^\circ$, then she sends $\bob$ an instance of $A$ and the exchange ends. If $\alice$ chooses $A^\circ \times B^\circ$, then next $\bob$ chooses which of $A^\circ$ and $B^\circ$ will happen. If $\bob$ chooses $A^\circ$ then $\alice$ sends him an instance of $A$ and the exchange ends. If $\bob$ chooses $B^\circ$ then instead $\alice$ sends an instance of $B$ and the exchange ends. 
\end{itemize}

We proceed to extend the rest of the free cornering construction with choice:
\begin{definition}\label{def:free-cornering-choice}
  Let $\A$ be a distributive monoidal category. We define the \emph{free cornering with choice} of $\A$, written $\corner{\A}^{\oplus}$, to be the free single-object double category with horizontal edge monoid $(\A_0,\otimes,I)$, vertical edge monoid $\ex{\A}_\oplus$, and generating cells and equations consisting of:
  \begin{itemize}
  \item The generating cells and equations of $\corner{\A}$ (Definition~\ref{def:free-cornering}).
  \item For each $U,W \in \ex{\A}_\oplus$, horizontal projection cells $\pi_0 : \cells{U \times W}{I}{I}{U}$ and $\pi_1 : \cells{U \times W}{I}{I}{W}$. Further, for each pair of cells $\alpha \in \corner{\A}^\oplus\cells{V}{A}{B}{U}$ and $\beta \in \corner{\A}^\oplus\cells{V}{A}{B}{W}$ a unique cell $\alpha \times \beta \in \corner{\A}^\oplus\cells{V}{A}{B}{U \times W}$ satisfying:
    \begin{mathpar}
      \includegraphics[height=1.2cm,align=c]{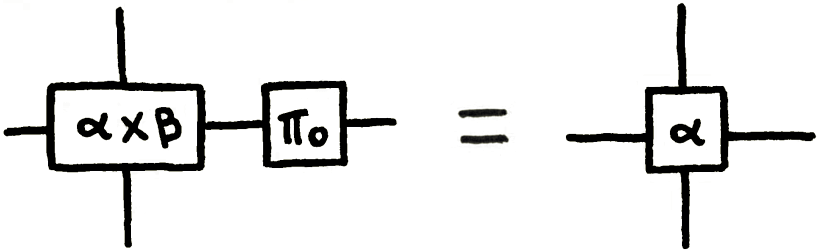}

      \includegraphics[height=1.2cm,align=c]{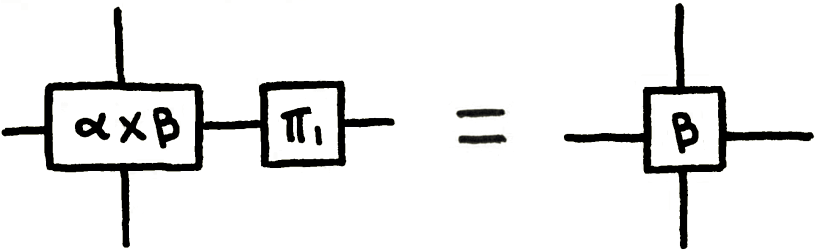}
    \end{mathpar}
  \item Dually, for each $U,W \in \ex{\A}_\oplus$, horizontal injection cells $\ip_0 : \cells{U}{I}{I}{U + W}$ and $\ip_1 : \cells{W}{I}{I}{U+W}$. Further, for each pair of cells $\alpha \in \corner{\A}^\oplus\cells{U}{A}{B}{V}$ and $\beta \in \corner{\A}^\oplus\cells{W}{A}{B}{V}$ a unique cell $\alpha + \beta \in \corner{\A}^\oplus\cells{U+W}{A}{B}{V}$ satisfying:
    \begin{mathpar}
      \includegraphics[height=1.2cm,align=c]{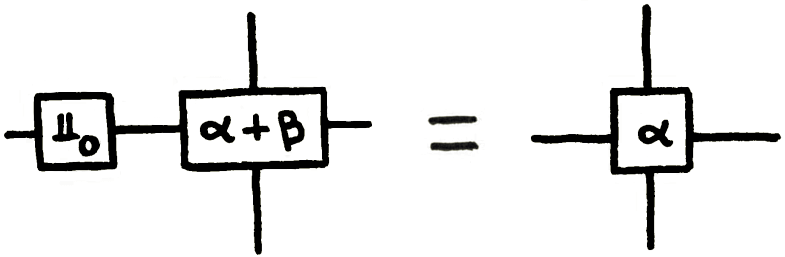}

      \includegraphics[height=1.2cm,align=c]{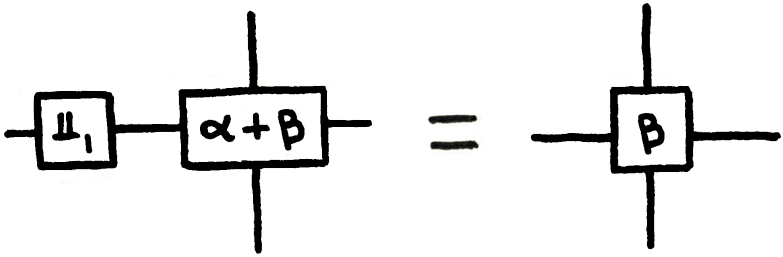}
    \end{mathpar}
  \item Finally, for each pair of cells $\alpha : \corner{\A}^\oplus\cells{U}{A}{C}{W}$ and $\beta : \corner{\A}^\oplus\cells{U}{B}{C}{W}$ a unique cell $[ \alpha , \beta ] : \corner{\A}^\oplus\cells{U}{A \oplus B}{C}{W}$ satisfying:
  \begin{mathpar}
  \includegraphics[height=1.2cm,align=c]{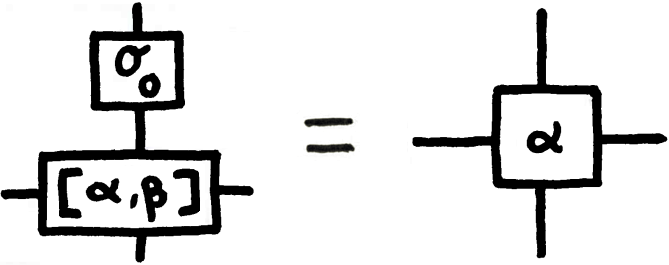}

  \includegraphics[height=1.2cm,align=c]{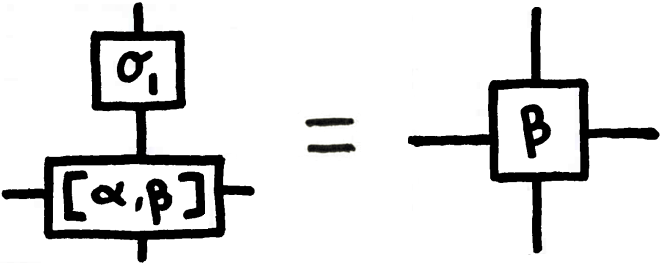}
  \end{mathpar}
  where $\corner{\sigma_0}$ and $\corner{\sigma_1}$ are given by the coproduct injections in $\A$.
\end{itemize}
\end{definition}

We extend our interpretation of cells of $\corner{\A}$ as interacting processes to cells of $\corner{\A}^\oplus$. Recall that $U \times W$ is the exchange in which the right participant chooses whether the exchange will be of the form $U$ or $W$. The projection cells $\pi_0 : \cells{U \times W}{I}{I}{U}$ and $\pi_1 : \cells{U \times W}{I}{I}{W}$ allow a cell, acting as the right participant in the exchange $U \times W$, to \emph{make} such choices. The corresponding cells $\alpha \times \beta : \cells{V}{A}{B}{U \times W}$ allow a cell, acting as the left participant in the exchange $U \times W$, to \emph{react} to such choices by specifying a response to each of the two possible choices. Similarly, $U + W$ is the exchange in which the left participant chooses whether the exchange will be of the form $U$ or $W$. The injections allow a cell, acting as the left participant, to make such choices, and the corresponding cells $\alpha + \beta : \cells{U+W}{A}{B}{V}$ to react to such choices by specifying a response to each possibility.

\begin{example}
  Suppose $\A$ contains a morphism $\texttt{bake} : \texttt{dough} \otimes \texttt{oven} \to \texttt{bread} \otimes \texttt{oven}$ (as in Section~\ref{subsec:free-cornering}), and define:
\[
\texttt{react} = \left(\getl^\texttt{bread} \mid 1_C\right) + \left(\frac{\getl^\texttt{dough} \mid 1_C}{\texttt{bake}}\right) : \floatcells{\texttt{bread}^\circ + \texttt{dough}^\circ}{\texttt{oven}}{\texttt{bread} \otimes \texttt{oven}}{I}
\]
Then $\texttt{react}$ describes a procedure for obtaining $\texttt{bread}$, assuming one posesses an $\texttt{oven}$, by participating in an exchange along the left boundary in which the counterparty chooses whether to supply $\texttt{bread}$ or $\texttt{dough}$. If they choose to supply $\texttt{bread}$ (via $\ip_0$), then the bread has been obtained. If they choose to supply $\texttt{dough}$ (via $\ip_1$), then we instead $\texttt{bake}$ the $\texttt{dough}$ to obtain $\texttt{bread}$. So for example:
\begin{mathpar}
  (\putr^\mathsf{bread} \mid \ip_0) \mid \texttt{react}
  =
  1_{\texttt{bread} \otimes \texttt{oven}}

  (\putr^\mathsf{dough} \mid \ip_1) \mid \texttt{react}
  =
  \texttt{bake}
\end{mathpar}
or, diagrammatically:
\begin{mathpar}
  \includegraphics[height=1.2cm,align=c]{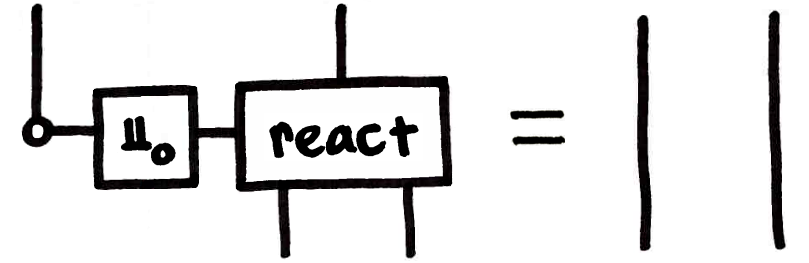}

  \includegraphics[height=1.2cm,align=c]{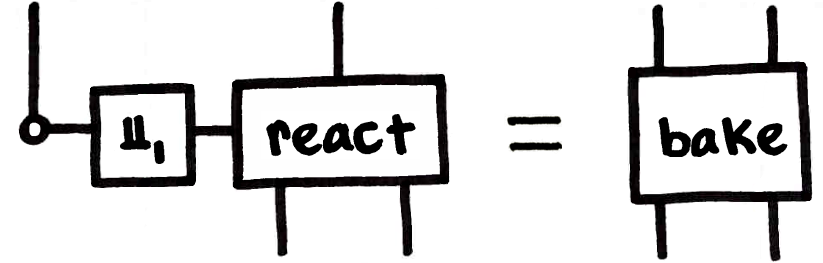}
\end{mathpar}
\end{example}
\begin{example}
  Consider $H : \cells{\texttt{dough}^\circ \times \texttt{oven}^\circ}{I}{\texttt{bread} \otimes \texttt{oven}}{\texttt{dough}^\bullet \times \texttt{oven}^\bullet}$ defined by
  \[
  H = \frac{\left(\pi_1 \mid \frac{\getl^\texttt{oven} \mid \getr^\texttt{dough}}{\sigma_{\texttt{oven},\texttt{dough}}}\right) \times \left(\pi_0 \mid \getl^\texttt{dough} \mid \getr^\texttt{oven}\right)}{\texttt{bake}}
  \]
  Then we have:
  \begin{mathpar}
    \includegraphics[height=1.2cm,align=c]{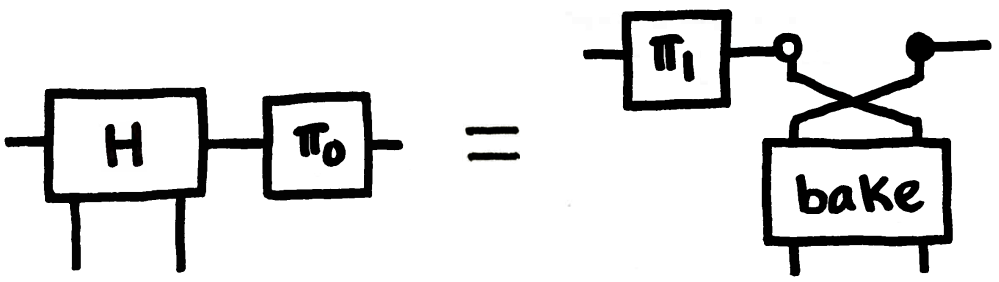}

    \includegraphics[height=1.2cm,align=c]{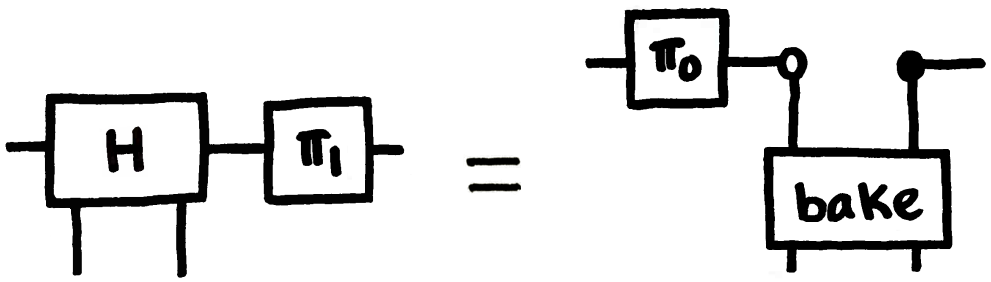}
  \end{mathpar}
  That is, if $\texttt{dough}$ is supplied along the right boundary, then $H$ chooses to obtain an $\texttt{oven}$ along the left boundary, and bakes bread. Otherwise an $\texttt{oven}$ is supplied along the right boundary, in which case $H$ chooses to obtain $\texttt{dough}$ along the left boundary, and bakes bread anyway. 
\end{example}
In this way, cells $\alpha + \beta$ and $\alpha \times \beta$ are understood as procedures that branch according to choices made externally as part of the exchange along their left an right boundary, respectively. We compare this to cells $[\alpha,\beta] : \cells{U}{A \oplus B}{C}{W}$, which we understand as procedures that branch according to their input, much as in Section~\ref{subsec:distributive-branching}. An important feature of $\corner{\A}^\oplus$ is that when such a procedure branches according to its input, this may be reflected in choices made along the left and right boundary. Explicitly, let $\alpha : \cells{U}{A}{C}{W}$ and $\beta : \cells {U'}{B}{C}{W'}$, and consider $[(\pi_0 \mid \alpha \mid \ip_0) , (\pi_1 \mid \beta \mid \ip_1)] : \cells{U + U'}{A \oplus B}{C}{W + W'}$. Then we have:
\begin{mathpar}
  \frac{\sigma_0}{[(\pi_0 \mid \alpha \mid \ip_0) , (\pi_1 \mid \beta \mid \ip_1)]}
  =
  \pi_0 \mid \alpha \mid \ip_0 

  \frac{\sigma_1}{[(\pi_0 \mid \alpha \mid \ip_0) , (\pi_1 \mid \beta \mid \ip_1)]}
  =
  \pi_1 \mid \beta \mid \ip_1
\end{mathpar}
and in this way the choice a procedure makes as part of some exchange along its left and/or right boundary may be determined by its inputs.
\begin{example}
  Suppose our base category has both an object $\texttt{bread}$ as well as an object $S_\texttt{bread} \cong I \oplus (\texttt{bread} \otimes S_\texttt{bread})$ of stacks of bread as in Section~\ref{subsec:distributive-branching}). Then consider the cell $H : \cells{\texttt{bread}^\circ \times I}{S_\texttt{bread}}{S_\texttt{bread}}{I}$ defined as in:
  \[
  H = \frac{\texttt{pop}}{\frac{\left[ (\pi_0 \mid \getl^\texttt{bread} \mid \texttt{nil}) , (\pi_1 \mid 1_{\texttt{bread} \otimes S_\texttt{bread}}) \right]}{\texttt{push}}}
  \]
  Then we have:
  \begin{mathpar}
    \includegraphics[height=1.2cm,align=c]{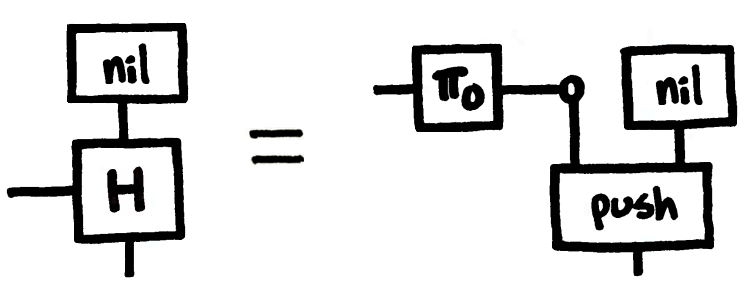}
        
    \includegraphics[height=1.2cm,align=c]{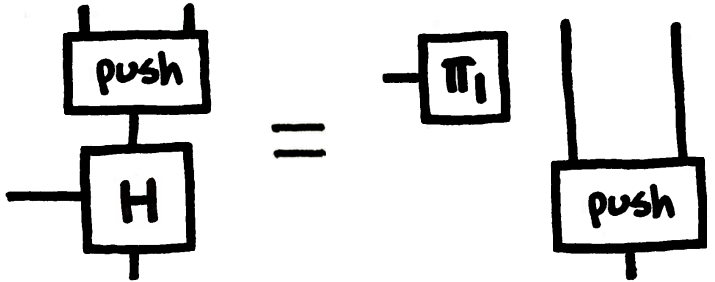}
  \end{mathpar}
  That is, if the input stack of \texttt{bread} is empty then $H$ chooses to obtain $\texttt{bread}$ along the left boundary. If the input stack of \texttt{bread} is nonempty then $H$ chooses to do nothing along the left boundary (presumably since it already has bread and does not need any more). 
\end{example}

\subsection{Elementary Properties}\label{subsec:choice-properties}
In this section we establish a number of elementary properties of the free cornering with choice. First, we observe that where our formation rule for cells $[\alpha,\beta]$ in $\corner{\A}^\oplus$ overlaps with the formation rule for copairing maps in $\A$, the two coincide:
\begin{lemma}\label{lem:copairing-coincide}
  For any $f : A \to C$ and $g : B \to C$ in $\A$, $\left[\,\corner{f},\corner{g}\,\right] = \corner{\,[f,g]\,}$ in $\corner{\A}^\oplus$. 
\end{lemma}
\begin{proof}
  We have:
  \begin{mathpar}
    \frac{\sigma_0}{\corner{\,[f,g]\,}} = \corner{f} = \frac{\sigma_0}{\left[\,\corner{f},\corner{g}\,\right]}

    \frac{\sigma_1}{\corner{\,[f,g]\,}} = \corner{g} = \frac{\sigma_1}{\left[\,\corner{f},\corner{g\,}\right]}
  \end{mathpar}
  and the claim follows.
\end{proof}
Next, we find that cells $[\alpha,\beta]$ enjoy certain absorption properties in $\corner{\A}^\oplus$:
\begin{lemma}\label{lem:copairing-absorbs}
  In $\corner{\A}^\oplus$:
  \begin{enumerate}[(i)]
  \item $\gamma \mid [\alpha,\beta] = \frac{\delta^l}{[(\gamma \mid \alpha),(\gamma \mid \beta)]}$
  \item $[\alpha,\beta] \mid \gamma = \frac{\delta^r}{[(\alpha \mid \gamma), (\beta \mid \gamma)]}$
  \item $\frac{[\alpha,\beta]}{\gamma} = [\frac{\alpha}{\gamma},\frac{\beta}{\gamma}]$
  \end{enumerate}
\end{lemma}
\begin{proof}
  \begin{enumerate}[(i)]
  \item We have:
    \begin{mathpar}
      \frac{\sigma_0}{\frac{\mu^l}{\gamma \mid [\alpha,\beta]}}
      = \gamma \mid \frac{\sigma_0}{[\alpha,\beta]}
      = \gamma \mid \alpha
      = \frac{\sigma_0}{[(\gamma \mid \alpha),(\gamma \mid \beta)]}

      \frac{\sigma_1}{\frac{\mu^l}{\gamma \mid [\alpha,\beta]}}
      = \gamma \mid \frac{\sigma_1}{[\alpha,\beta]}
      = \gamma \mid \beta
      = \frac{\sigma_1}{[(\gamma \mid \alpha),(\gamma \mid \beta)]}
    \end{mathpar}
    and so we have $\frac{\mu^l}{\gamma \mid [\alpha,\beta]} = [(\gamma \mid \alpha),(\gamma \mid \beta)]$. Precomposing vertically with $\delta^l$ on both sides proves the claim.
  \item Similar to (i).
  \item We have:
    \begin{mathpar}
      \frac{\sigma_0}{\frac{[\alpha,\beta]}{\gamma}}
      = \frac{\alpha}{\gamma}
      = \frac{\sigma_0}{[\frac{\alpha}{\gamma},\frac{\beta}{\gamma}]}

      \frac{\sigma_1}{\frac{[\alpha,\beta]}{\gamma}}
      = \frac{\beta}{\gamma}
      = \frac{\sigma_1}{[\frac{\alpha}{\gamma},\frac{\beta}{\gamma}]}
    \end{mathpar}
    and the claim follows.
  \end{enumerate}
\end{proof}
While (iii) is analogous to the naturality of the codiagonal map in a category with finite coproducts, (i) and (ii) do not seem to admit similar analogies. 

When restricted to the category of horizontal cells, the axioms concerning cells $\alpha \times \beta$ and $\alpha + \beta$  are precisely the axioms for binary products and binary coproducts. That is, we have:
\begin{lemma}\label{lem:branching-products-coproducts}
  We have:
  \begin{enumerate}[(i)]
  \item $\bh\,\corner{\A}^{\oplus}$ has binary products $U \stackrel{\pi_0}{\from}U \times W \stackrel{\pi_1}{\to} W$.
  \item $\bh\,\corner{\A}^{\oplus}$ has binary coproducts $U \stackrel{\ip_0}{\to} U + W \stackrel{\ip_1}{\from} W$.
  \end{enumerate}
\end{lemma}

The category of horizontal cells $\bh\,\corner{\A}$ of the free cornering can be understood as a category of exchanges, a perspective developed in \cite{Nester2021a,Nester2023}. In particular, isomorphic objects of $\bh\,\corner{\A}$ correspond to exchanges of $\ex{\A}$ that are \emph{morally equivalent} (Lemma 3 of \cite{Nester2021a}). We show that $\bh\,\corner{\A}^\oplus$ contains two novel pairs of such morally equivalent exchanges:
\begin{lemma}\label{lem:moral-equivalence-choice}
  In $\bh\,\corner{\A}^\oplus$:
  \begin{enumerate}[(i)]
  \item $(A \oplus B)^\circ \cong A^\circ + B^\circ$ and $(A \oplus B)^\bullet \cong A^\bullet \times B^\bullet$
  \item $(A \times B) \times C \cong A \times (B \times C)$ and $(A + B) + C \cong A + (B + C)$
  \end{enumerate}
\end{lemma}
\begin{proof}
  \begin{enumerate}[(i)]
  \item Let $\gamma = \frac{\getl^A}{\sigma_0} + \frac{\getl^B}{\sigma_1} : \cells{A^\circ + B^\circ}{I}{A \oplus B}{I}$. That is, $\gamma$ is the unique cell such that:
\begin{mathpar}
  \includegraphics[height=1.6cm,align=c]{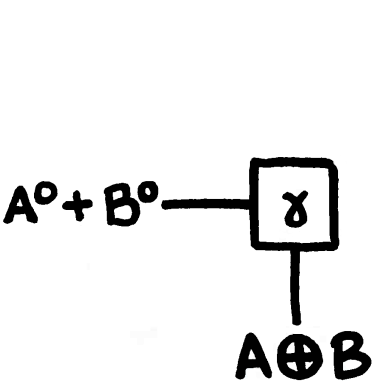}
  
  \includegraphics[height=1.2cm,align=c]{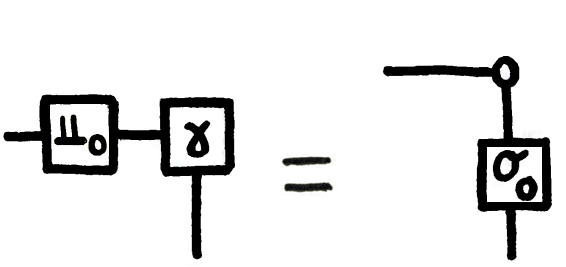}

  \includegraphics[height=1.2cm,align=c]{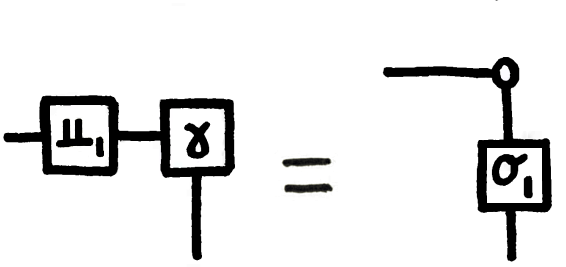}
\end{mathpar}
Next, define $\delta = \left[ \putr^A \mid \ip_0 , \putr^B \mid \ip_1 \right] : \cells{I}{A \oplus B}{I}{A^\circ \otimes B^\circ}$. That is, $\delta$ is the unique cell such that:
\begin{mathpar}
  \includegraphics[height=1.6cm,align=c]{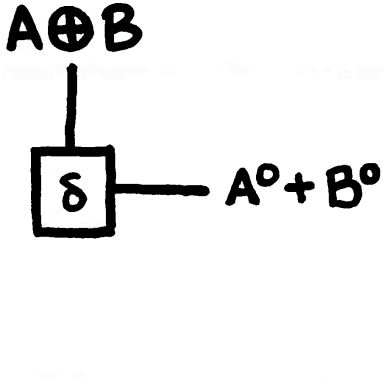}

  \includegraphics[height=1.2cm,align=c]{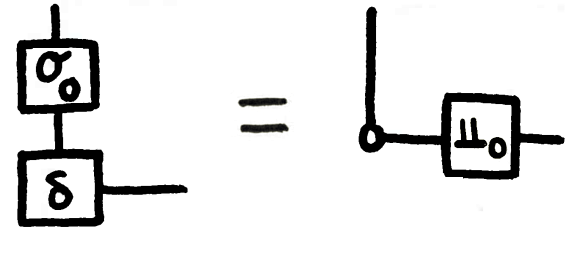}

  \includegraphics[height=1.2cm,align=c]{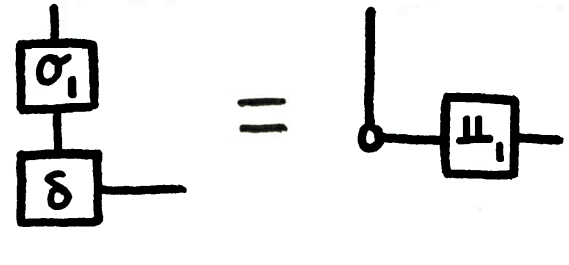}
\end{mathpar}
Then we have
\begin{mathpar}
  \includegraphics[height=1.6cm,align=c]{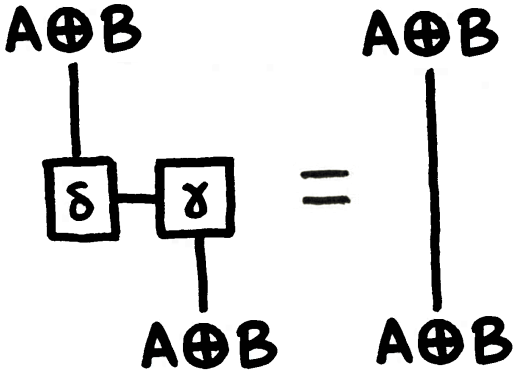}
\end{mathpar}
by the universal property of $\oplus$, since we have both of:
\begin{mathpar}
  \includegraphics[height=1.2cm,align=c]{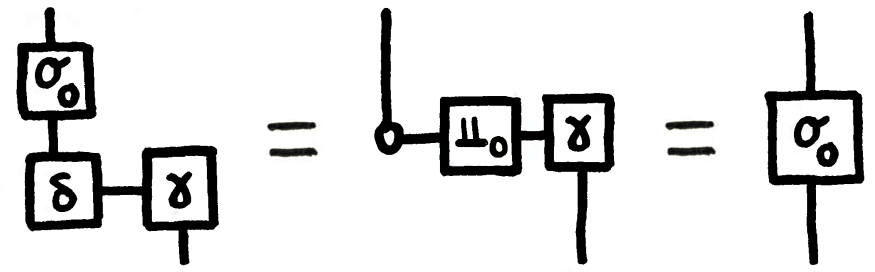}

  \includegraphics[height=1.2cm,align=c]{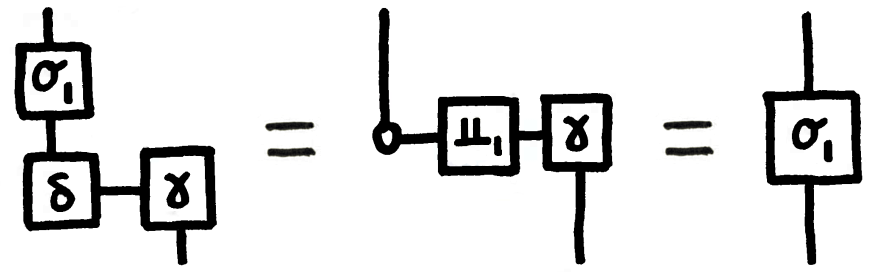}
\end{mathpar}
Similarly, we have:
\begin{mathpar}
  \includegraphics[height=1.6cm,align=c]{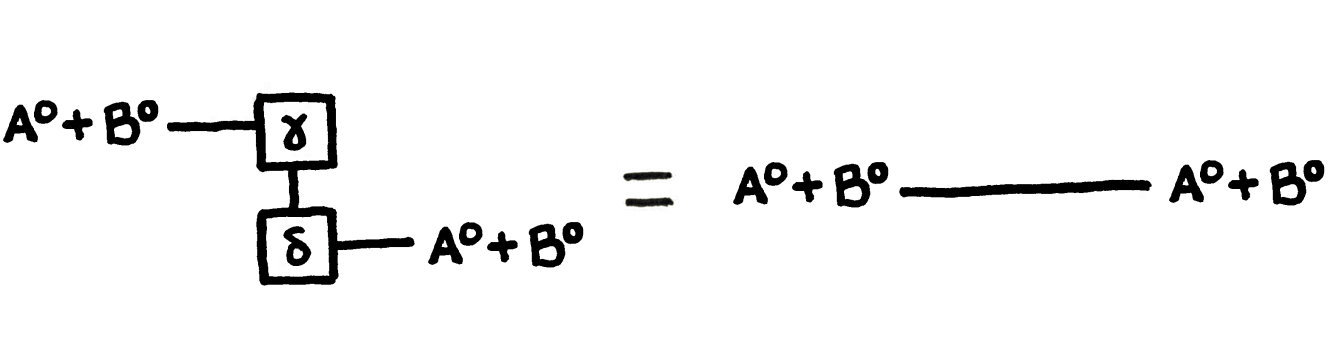}
\end{mathpar}
by the universal property of $+$, as in:
\begin{mathpar}
  \includegraphics[height=1.2cm,align=c]{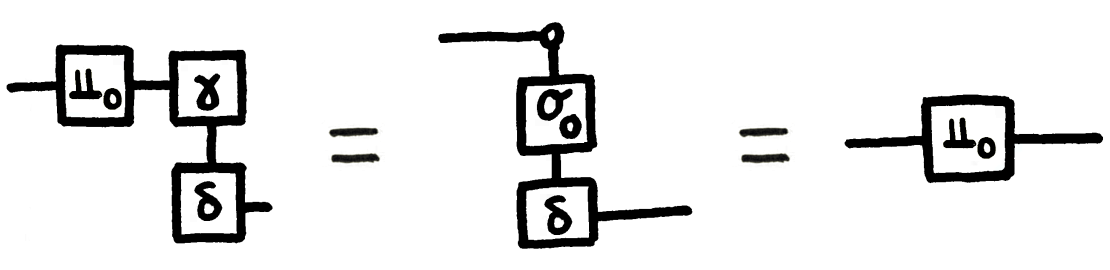}

  \includegraphics[height=1.2cm,align=c]{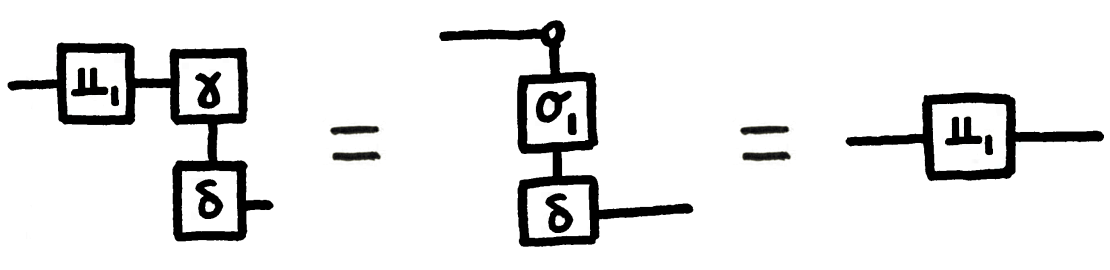}
\end{mathpar}
Then the following arrows of $\bh\,\corner{\A}^\oplus$ are mutually inverse:
\begin{mathpar}
  \includegraphics[height=1.6cm,align=c]{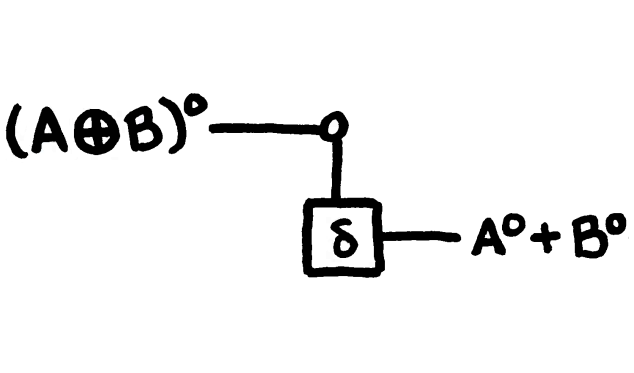}

  \includegraphics[height=1.6cm,align=c]{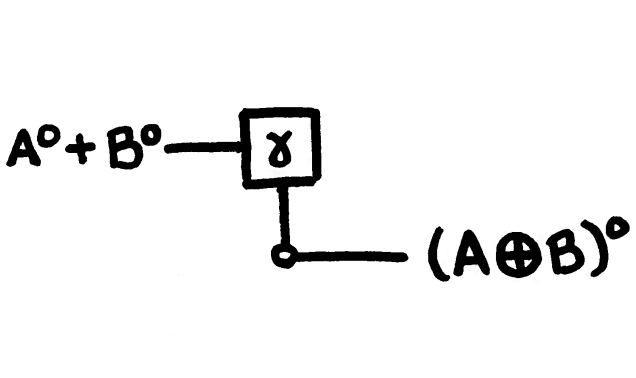}
\end{mathpar}
and the claim follows. The proof that $(A \oplus B)^\bullet = A^\bullet \times B^\bullet$ is similar. 
\item Follows immediately from Lemma~\ref{lem:branching-products-coproducts}.
  \end{enumerate}
\end{proof}
This makes intuitive sense: from the resource-theoretic perspective, an instance of $A \oplus B$ is either an instance of $A$ or an instance of $B$. Then it certainly ought to be the case that $\alice$ giving $\bob$ an instance of $A \oplus B$ is the same as $\alice$ choosing whether to give $\bob$ an instance of $A$ or to give $\bob$ an instance of $B$. The above lemma tells us that this is indeed the case. 
\subsection{Crossing Cells}\label{subsec:choice-crossing}
We extend Definition~\ref{def:crossing-cells} to obtain crossing cells in the free cornering with choice:
\begin{definition}\label{def:choice-crossing}
  Let $\A$ be a distributive monoidal category. For each $A \in \corner{\A}^\oplus_H$ and each $U \in \corner{\A}^\oplus_V$ we define a crossing cell $\chi_{U,A} : \cells{U}{A}{A}{U}$ by induction on the structure of $U$. The cases for $A^\circ$, $A^\bullet$, $I$, and $U \otimes W$ are as in Definition~\ref{def:crossing-cells}. For $U+W$ we define $\chi_{U + W,A} = (\chi_{U,A} \mid \ip_0) + (\chi_{W,A} \mid \ip_1)$. That is, $\chi_{U+W,A}$ is the unique cell satisfying:
\begin{mathpar}
  \includegraphics[height=1.2cm,align=c]{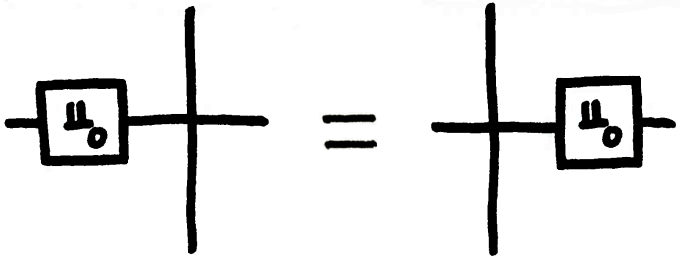}

  \includegraphics[height=1.2cm,align=c]{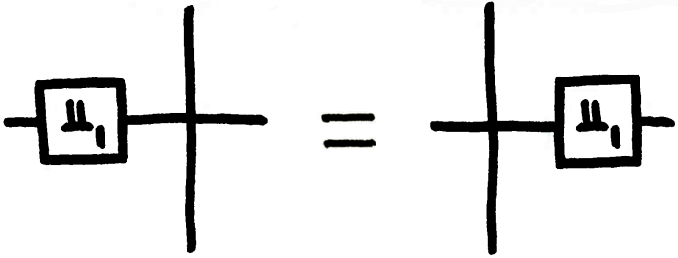}
\end{mathpar}
Similarly, for $U \times W$ we define $\chi_{U \times W,A} = (\pi_0 \mid \chi_{U,A}) \times (\pi_1 \mid \chi_{W,A})$, so $\chi_{U \times W,A}$ is the unique cell satisfying:
\begin{mathpar}
  \includegraphics[height=1.2cm,align=c]{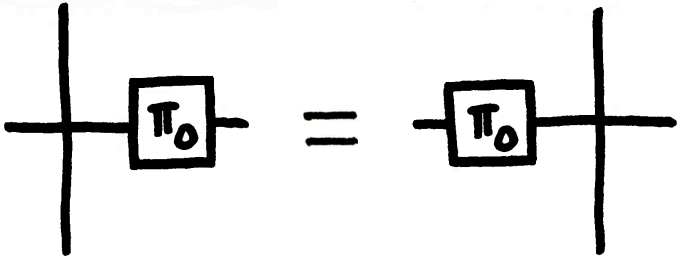}

  \includegraphics[height=1.2cm,align=c]{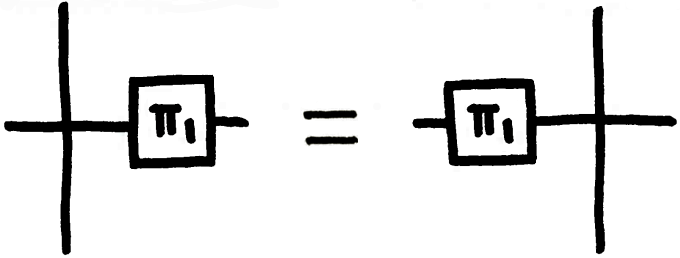}
\end{mathpar}
\end{definition}

We show that the crossing cells remain well-behaved. First, the crossing cells remain coherent with respect to horizontal composition in the free cornering with choice:
\begin{lemma}\label{lem:coherent-crossing-choice}
  For $U \in \ex{A}_\oplus$ and $A,B \in \A_0$ we have 
  \begin{enumerate}[(i)]
  \item $\chi_{U,{A \otimes B}} = \chi_{U,A} \mid \chi_{U,B}$
  \item $\chi_{U,I} = id_U$
  \end{enumerate}
\end{lemma}
\begin{proof}
  We provide the inductive cases necessary to extend the proof of Lemma~\ref{lem:coherent-crossing} to account for the new structure.
  \begin{enumerate}[(i)]
  \item In the inductive case for $U + W$ we have:
    \[
    \ip_0 \mid \chi_{U+W,A \otimes B} = \chi_{U,A \otimes B} \mid \ip_0 = \chi_{U,A} \mid \chi_{U,B} \mid \ip_0 = \ip_0 \mid \chi_{U+W,A} \mid \chi_{U+W,B}
    \]
    Similarly, we have $\ip_1 \mid \chi_{U+W,A \otimes B} = \ip_1 \mid \chi_{U+W,A} \mid \chi_{U+W,B}$. It follows by the universal property of $-+-$ that $\chi_{U+W,A\otimes B} = \chi_{U+W,A} \mid \chi_{U+W,B}$. The inductive case for $U\times W$ is similar. 
  \item In the inductive case for $U + W$ we have:
    \[
    \ip_0 \mid \chi_{U+W,I} = \chi_{U,I} \mid \ip_0 = id_U \mid \ip_0 = \ip_0 \mid id_{U+W}
    \]
    Similarly, we have $\ip_1 \mid \chi_{U+W,I} = \ip_1 \mid id_{U+W}$. It follows that $\chi_{U+W,I} = id_{U+W}$, as required. The case for $U \times W$ is similar. 
  \end{enumerate}
\end{proof}

Further, we find that the core technical lemma concerning crossing cells still holds:
\begin{lemma}~\label{lem:choice-crossing-swaps}
  For any cell $\alpha$ of $\corner{\A}^\oplus$ we have
  \[
  \includegraphics[height=1.2cm]{figs/crossing-swaps.png}
  \]
\end{lemma}
\begin{proof}
  By structural induction on cells of $\corner{\A}^\oplus$. The base cases and the (inductive) cases for cells $\alpha \mid \beta$ and $\frac{\alpha}{\beta}$ are as in the proof of Lemma~\ref{lem:crossing-swaps}. The remaining inductive cases are as follows: For cells $\alpha + \beta$, we have:
  \begin{mathpar}
    \includegraphics[height=1.2cm,align=c]{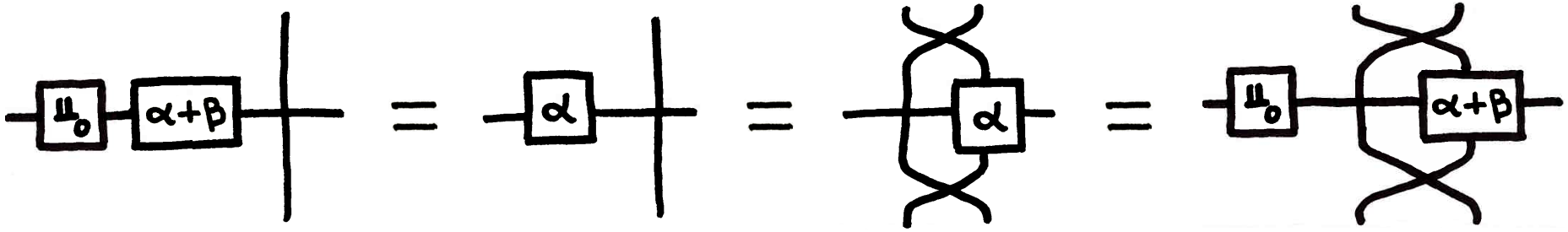}

    \includegraphics[height=1.2cm,align=c]{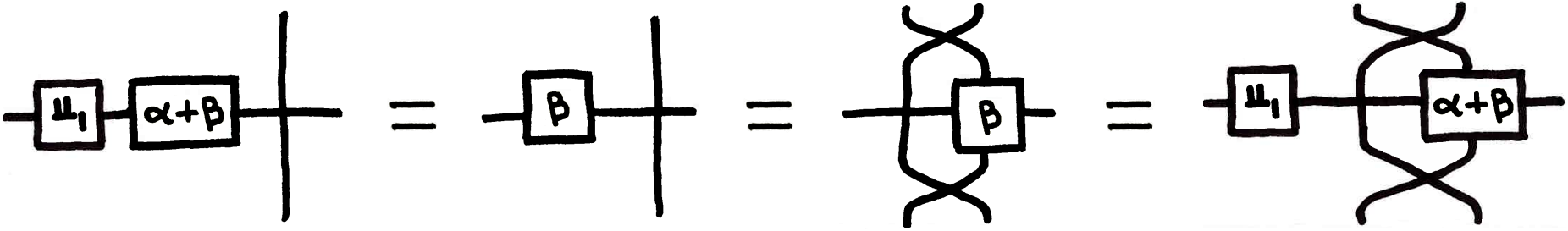}
  \end{mathpar}
  and then by the universal property of $+$ we have
  \[
  \includegraphics[height=1.2cm,align=c]{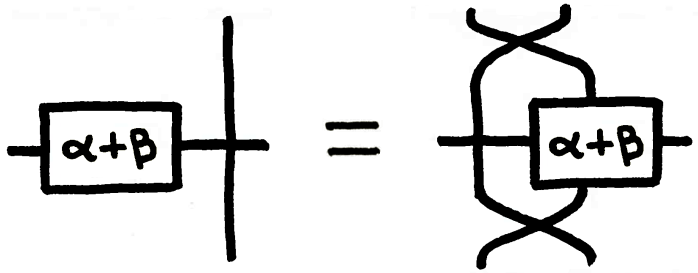}
  \]
  as required. The case for cells $\alpha \times \beta$ is similar. In the inductive case for cells $[\alpha,\beta]$ we have:
  \begin{mathpar}
    \includegraphics[height=1.6cm,align=c]{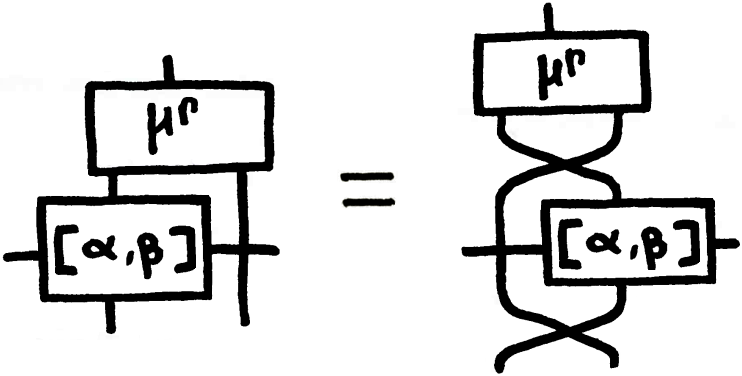}
  \end{mathpar}
  by the universal property of $\oplus$ as in:
  \begin{mathpar}
    \includegraphics[height=2cm,align=c]{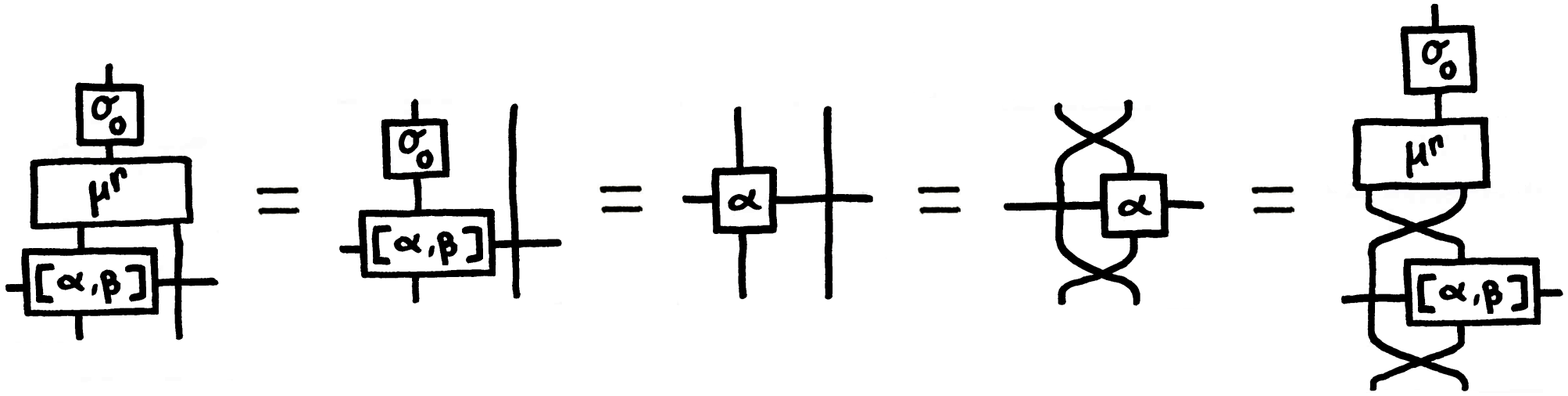}
  \end{mathpar}
  with the case for $\sigma_1$ being similar. Precomposing vertically with $\delta_r$ yields:
  \begin{mathpar}
    \includegraphics[height=1.2cm,align=c]{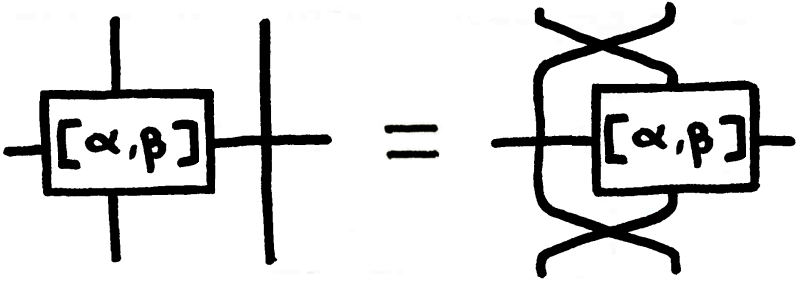}
  \end{mathpar}
\end{proof}

Consequently, $\corner{\A}^\oplus$ is a monoidal double category with the tensor product of cells and proof as in Lemma~\ref{lem:monoidaldouble}. We record:
\begin{lemma}\label{lem:choicemonoidaldouble}
  If $\A$ is a distributive monoidal category then $\corner{\A}^\oplus$ is a monoidal double category. 
\end{lemma}

Further, we find that crossing cells in the free cornering with choice are coherent with respect to $\oplus$ in the following sense:
\begin{lemma}\label{lem:coherent-oplus-choice}
  In $\corner{\A}^\oplus$, $\chi_{U,A \oplus B} = \left[\frac{\chi_{U,A}}{\sigma_0},\frac{\chi_{U,B}}{\sigma_1}\right]$. That is, $\chi_{U,A \oplus B}$ is the unique cell such that:
  \begin{mathpar}
    \frac{\sigma_0}{\chi_{U,A\oplus B}} = \frac{\chi_{U,A}}{\sigma_0}

    \frac{\sigma_1}{\chi_{U,A\oplus B}} = \frac{\chi_{U,B}}{\sigma_1}
  \end{mathpar}
\end{lemma}
\begin{proof}
  By structural induction on $U$. In case $U = C^\circ$ we have $\frac{\sigma_0}{\chi_{C^\circ,A \oplus B}} = \frac{\chi_{C^\circ,A}}{\sigma_0}$ as in:
  \begin{mathpar}
    \includegraphics[height=1.6cm,align=c]{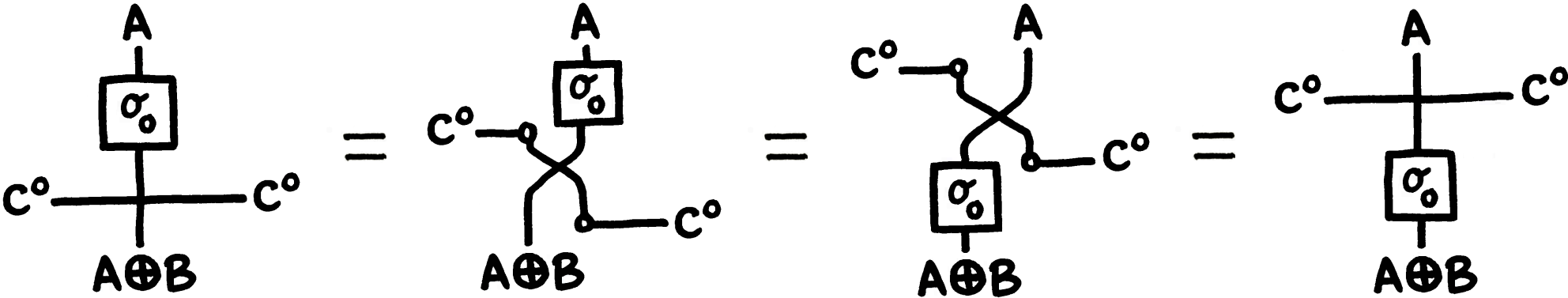}
  \end{mathpar}
  Similarly $\frac{\sigma_1}{\chi_{C^\circ,A \oplus B}} = \frac{\chi_{C^\circ,B}}{\sigma_1}$, and an analogous argument can be made for $U = C^\bullet$. If $U = I$ then we have:
  \begin{align*}
    & \frac{\sigma_0}{\chi_{I,A \oplus B}} = \frac{\sigma_0}{1_{A \oplus B}} = \frac{1_A}{\sigma_0} = \frac{\chi_{I,A}}{\sigma_0}
  \end{align*}
  Similarly $\frac{\sigma_1}{\chi_{I,A\oplus B}} = \frac{\chi_{I,B}}{\sigma_1}$. For $U \otimes W$ we have:
  \begin{align*}
    & \frac{\sigma_0}{\chi_{U \otimes W, A \oplus B}}
    = \frac{\sigma_0}{\frac{\chi_{U,A\oplus B}}{\chi_{W,A \oplus B}}}
    = \frac{\frac{\chi_{U,A}}{\chi_{W,A}}}{\sigma_0}
    = \frac{\chi_{U \otimes W,A}}{\sigma_0}
  \end{align*}
  Similarly, we have $\frac{\sigma_1}{\chi_{U \otimes W, A \oplus B}} = \frac{\chi_{U \otimes W,B}}{\sigma_1}$. For $U + W$ we have $\ip_0 \mid \frac{\sigma_0}{\chi_{U+W,A\oplus B}} = \ip_0 \mid \frac{\chi_{U+W,A}}{\sigma_0}$ as in:
  \begin{mathpar}
    \includegraphics[height=1.6cm,align=c]{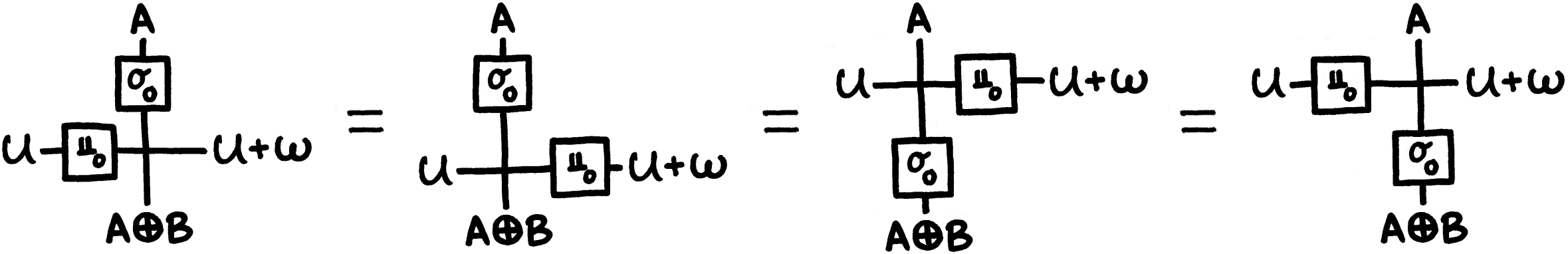}
  \end{mathpar}
  An analogous argument gives $\ip_1 \mid \frac{\sigma_0}{\chi_{U+W,A\oplus B}} = \ip_1 \mid \frac{\chi_{U+W,A}}{\sigma_0}$, and so $\frac{\sigma_0}{\chi_{U+W,A \oplus B}} = \frac{\chi_{U+W,A}}{\sigma_0}$ by the universal property of $+$. Similarly, we have $\frac{\sigma_1}{\chi_{U+W,A\oplus B}} = \frac{\chi_{U+W,B}}{\sigma_1}$. The case for $U \times W$ is similar to the case for $U + W$. 
\end{proof}

\subsection{A Model: Stateful Choice}\label{subsec:stateful-choice}
In this section we return to the double category of stateful transformations defined in Section~\ref{subsec:stateful-transformations} over a cartesian closed category $\C$. We show that if $\C$ has binary coproducts $\oplus$ which are distributive with respect to the cartesian product $\otimes$ then we can define the branching protocols of the free cornering with choice in the double category $\Dst{\C}$. In keeping with our strictness assumptions, the categorical product $\otimes$ is assumed to be strictly associative. As before, we do not need to ask this of our binary coproducts $\oplus$. Of course, the distributive laws hold only up to isomorphism. 

Let $\pl^{X,Y} : X \otimes Y \to X$ and $\pr^{X,Y} : X \otimes Y \to Y$ be the projection maps, and note that they are natural in $X$ and $Y$. For $f : A \to B$ and $g : A \to C$, we write $\langle f,g \rangle : A \to B \otimes C$ to be the unique morphism such that $\langle f,g \rangle \pl = f$ and $\langle f,g \rangle \pr = g$. We note that the injections $\sigma_0^{X,Y} : X \to X \oplus Y$, $\sigma_1^{X,Y} : Y \to X \oplus Y$ given by the coproduct structure are also natural in $X$ and $Y$. Finally, both $\otimes$ and $\oplus$ give bifunctors $\C$ sending $f : A \to C$ and $g: B \to D$ to $(f \otimes g) =  \langle \pl \, f , \pr \, g \rangle : A \otimes B \to C \otimes D$ and $(f \oplus g) = [f \, \sigma_0 , g \, \sigma_1]$. 

Given two strong endofunctors $(F, \tau^F)$ and $(G, \tau^F)$, we define their product $(F \times G, \tau^{F \times G})$ as follows: the endofunctor $F \times G$ is given by $(F \times G)(X) = FX \otimes GX$ and $(F \times G)(f) = F(f) \otimes G(f)$; and the strength $\tau^{F \times G}$ is given by:
\[
(F \times G)(X) \otimes A \xrightarrow{\langle (\pl \otimes A), (\pr \otimes A)\rangle} FX \otimes A \otimes GX \otimes A \xrightarrow{\tau^F_{X,A} \otimes \tau^G_{X, A}} (F \times G)(X \otimes A)
\]

Similarly, the endofunctor $F + G$ is given by $(F + G)(X) = FX \oplus GX$ and $(F + G)(f) = F(f) \oplus G(f)$. Its strength $\tau^{F+G}$ is given by:
\[
(F+G)(X) \otimes A \xrightarrow{\delta^r_{FX,GX,A}} (FX \otimes A) \oplus (GX \otimes A) \xrightarrow{\tau^F_{X,A} \oplus \tau^G_{X,A}} (F+G)(X \otimes A)
\]

These two constructions have the same universal properties as the corresponding definitions of choice protocols in Definition \ref{def:free-cornering-choice}.
Specifically, we have $2$-cells $\pi_0^{F,G} : \cells{F \times G}{1}{1}{F}$ and $\pi_1^{F,G} : \cells{F \times G}{1}{1}{F}$ with components $\pl^{FX,GX}$ and $\pr^{FX,GX}$ respectively. Given $\alpha : \cells{F}{A}{B}{G}$ and $\beta : \cells{F}{A}{B}{H}$, we define $(\alpha \times  \beta) : \cells{F}{A}{B}{G \times H}$ by:
\[
(\alpha \times  \beta)_X = \langle \alpha_X,\beta_X \rangle :  FX \otimes A \to (G \times H)(X \otimes B)
\]

Dually, we have $2$-cells $\ip_0^{F,G} : \cells{F}{1}{1}{F+G}$ and $\ip_1^{F,G} : \cells{G}{1}{1}{F+G}$ with components $\sigma^{FX,GX}_0$ and $\sigma^{FX,GX}_1$ respectively. Given $\alpha : \cells{F}{A}{B}{H}$ and $\beta : \cells{G}{A}{B}{H}$, we define $(\alpha \times  \beta) : \cells{F+G}{A}{B}{H}$ by using the distributivity natural transformation:
\[
(\alpha + \beta)_X = \delta^r[\alpha_X,\beta_X] : (F+G)(X) \otimes A \to H(X \otimes B)
\]

Finally, we have vertical injection cells $[\sigma_0^{A,B}] : \cells{I}{A}{A \oplus B}{I}$ and $[\sigma_1^{A,B}] : \cells{I}{B}{A \oplus B}{I}$. Given $\alpha : \cells{F}{A}{C}{G}$ and $\beta : \cells{F}{B}{C}{G}$, we define the corresponding copairing cell $[\alpha , \beta] : \cells{F}{A \oplus B}{C}{H}$ as in:
\[
[\alpha , \beta]_X = \delta^l[\alpha_X,\beta_X] : FX \otimes (A \oplus B) \to G(X \otimes C)
\]

These constructions have universal properties inherited from the distributive cartesian structure of $\C$, and it is straightforward to show that the projections and injections are strong. We show that these constructions are well-defined in the sense that they give 2-cells of $\mathsf{S}(\C)$:
\begin{lemma}
  If $\alpha$ and $\beta$ are $2$-cells of $\mathsf{S}(\C)$, then $\alpha \times \beta$, $\alpha + \beta$ and $[\alpha , \beta]$ are $2$-cells of $\mathsf{S}(\C)$.
\end{lemma}
\begin{proof}
  We show that the results of the operations are strong by showing that they are a composition of strong natural transformations. First, note that the pairing and copairing operations $\langle -,- \rangle$ and $[-,-]$ when applied to strong natural transformations $\alpha, \beta$  create strong natural transformations, as in:
  \begin{mathpar}
  \begin{tikzcd}[column sep=large]
			FX \otimes A 
			\ar[r,"{\langle \alpha_X, \beta_X \rangle \otimes 1_A}"]
			\ar[dd,"{\tau^F_{X,A}}"']
			\ar[dr,"{\langle \alpha_X \otimes 1_A, \beta_X \otimes 1_A \rangle}" description] 
			&	(G \times H)(X) \otimes A
			\ar[d,"{\langle \pl \otimes 1_A, \pr \otimes 1_A\rangle}"]\\
			&	GX \otimes A \otimes HX \otimes A
			\ar[d,"{\tau^G_{X,A} \otimes \tau^H_{X,A}}"]\\
			F(X \otimes A)
			\ar[r,"{\langle \alpha_{X \otimes A}, \beta_{X \otimes A} \rangle}"']
			&	(G \times H)(X \otimes A)
  \end{tikzcd}
  
  \begin{tikzcd}[column sep=large]
			(F+G)(X) \otimes A
			\ar[r,"{[\alpha_X , \beta_X] \otimes 1_A}"]
			\ar[d,"{\delta^r}"']
			&	HX \otimes A
			\ar[dd,"{\tau^H_{X,A}}"]\\
			(FX \otimes A) \oplus (GX \otimes A)
			\ar[d,"{[\tau^F_{X,A} , \tau^G_{X,A}]}"']
			\ar[ur,"{[\alpha_X \otimes 1_A , \beta_X \otimes 1_A]}" description]
			& \\
			(F+G)(X \otimes A)
			\ar[r,"{[ \alpha_{X \otimes A}, \beta_{X \otimes A} ]}"']
			&	H(X \otimes A)
  \end{tikzcd}
  \end{mathpar}
	Secondly, given strong endofunctors $F$ and $G$, and objects $A$ and $B$, both of the natural transformations $\delta^r : (FX \oplus GX) \otimes A \to (FX \otimes A) \oplus (GX \otimes A)$ and $\delta^l : FX \otimes (A \oplus B) \to (FX \otimes A) \oplus (FX \otimes B)$ are strong since their inverses $\mu^r$ and $\mu^l$ can be constructed using $\sigma_0$, $\sigma_1$ and $[-,-]$.
	We conclude that $\alpha \times \beta$, $\alpha + \beta$ and $[\alpha , \beta]$ are strong since they are compositions of strong natural transformations.
\end{proof}

As before, we have that $\mathsf{S}(\C)$ is a proarrow equipment and also a monoidal double category. Moreover, the double functor of Lemma~\ref{lem:doublefunctor} extends to the setting with choice:
\begin{lemma}\label{lem:doublefunctorchoice}
  Let $\C$ be a cartesian closed category with distributive binary coproducts. Then the double functor $D : \corner{\C} \to \mathsf{S}(\C)$ extends to a double functor $D : \corner{\C}^\oplus \to \mathsf{S}(\C)$ as follows: for the vertical edge monoid we take $D(U + W) = DU + DW$ and $D(U \times W) = DU \times DW$; and for the new cells we take $D([\alpha,\beta]) = [D(\alpha),D(\beta)]$, $D(\alpha + \beta) = D(\alpha) + D(\beta)$, and $D(\alpha \times \beta) = D(\alpha) \times D(\beta)$. 
\end{lemma}\qed

We remark that this double functor maps the cells related to branching communication protocols in $\corner{\C}^\oplus$ to cells carrying similar structure in $\mathsf{S}(\C)$, and speculate that it will therefore preserve whatever kind of structure that turns out to be. In this sense, we consider $\mathsf{S}(\C)$ to give a model of $\corner{\C}^\oplus$.

\begin{remark}\label{rem:effects-interp-choice}
  We may consider the interpretation of branching protocols given by $-+-$ and $-\times-$ from the perspective of computational effects, extending Remark~\ref{rem:effects-interpretation}. As a computational effect, $F+G$ may be triggered by any program which would trigger $F$ or $G$. Dually, to resolve an effect $F+G$ its environment must be able to resolve both $F$ and $G$ independently. Similarly, to trigger $F\times G$ a program must be ready for the environment to resolve either of $F$ or $G$, and dually to resolve $F\times G$ the environment must be able to resolve either $F$ or $G$ independently. For $F+G$ the choice comes from the interior, while for $F \times G$ the choice comes from the environment.
\end{remark}
  
This concludes our discussion of the free cornering with choice. We proceed to add a notion of protocol iteration the free cornering.

\section{Adding Iteration to the Free Cornering}\label{sec:corner-iteration}
In this section we extend the free cornering with choice to include a notion of protocol iteration. In Section~\ref{subsec:free-iteration} we construct the \emph{free cornering with iteration} over a distributice monoidal category and discuss its interpretation. In Section~\ref{subsec:iteration-properties} we establish a number of elementary properties of the free cornering with iteration, and in Section~\ref{subsec:iteration-crossing} we define crossing cells in the free cornering with iteration, and show that they are well-behaved (extending Section~\ref{subsec:choice-crossing} to the new setting). Finally, in Section~\ref{subsec:stateful-iteration} we discuss iteration in terms of the double category of stateful transformations. 

\subsection{The Free Cornering with Iteration}\label{subsec:free-iteration}
In this section we extend the free cornering with choice to include a notion of protocol iteration. We begin by extending the monoid of exchanges with choice (Definition~\ref{def:choice-exchanges}) with unary operations $(-)^+$ and $(-)^\times$ representing iterated protocols:
\begin{definition}\label{def:iteration-exchanges}
    Let $\A$ be a symmetric monoidal category. The monoid $\ex{\A}_*$ of \emph{$\A$-valued exchanges with choice and iteration} has elements generated by:
  \begin{mathpar}
    \inferrule{A \in \A_0}{A^\circ \in \ex{\A}_*}
      
    \inferrule{A \in \A_0}{A^\bullet \in \ex{\A}_*}

    \inferrule{\text{}}{I \in \ex{\A}_*}

    \inferrule{U \in \ex{\A}_* \\ W \in \ex{\A}_*}{U \otimes W \in \ex{\A}_*}
     
    \inferrule{U \in \ex{\A}_* \\ W \in \ex{\A}_*}{U \times W \in \ex{\A}_*}
    
    \inferrule{U \in \ex{\A}_* \\ W \in \ex{\A}_*}{U + W \in \ex{\A}_*}
    
    \inferrule{U \in \ex{\A}_*}{U^\times \in \ex{\A}_*}

    \inferrule{U \in \ex{\A}_*}{U^+ \in \ex{\A}_*}
  \end{mathpar}
  subject to the following equations:
  \begin{mathpar}
  I \otimes U = U

  U \otimes I = U

  (U \otimes W) \otimes V = U \otimes (W \otimes V)

  U^\times = I \times (U \otimes U^\times)

  U^+ = I + (U \otimes U^+)
  \end{mathpar}
Notice in particular that  $\ex{\A}_\oplus$ embeds into $\ex{\A}_*$.
\end{definition}

We extend our interpretation of elements of $\ex{\A}_\oplus$ to elements of $\ex{\A}_*$. To do so we must interpret $U^\times$ and $U^+$. We begin with $U^\times$. Our interpretation of $U^\times$ is informed by the equation $U^\times = I \times (U \otimes U^\times)$. Recall that $V \times W$ is the protocol in which the right participant chooses whether to continue as $V$ or $W$. It follows that $U^\times$ is the protocol in which the right participant chooses whether to continue as $I$ or $U \otimes U^\times$. Thus, $U^\times$ is the protocol in which the right participant chooses whether to do nothing, or to do $U$ and then $U^\times$ again. Put another way, $U^\times$ is the iterated version of $U$, in which the right participant decides when to stop iterating. Our interpretation of $U^+$ is dual. Everything is as above, except that the roles of the right and left participant are swapped. 

For example, suppose $A,B \in \A_0$. For each of the following exchanges, call the left participant $\alice$ and the right participant $\bob$, as before. Now, consider:
\begin{itemize}
\item To carry out $(A^\circ)^\times = I \times (A^\circ \otimes (A^\circ)^\times) \in \ex{\A}_*$, first $\bob$ chooses which of $I$ and $A^\circ \otimes (A^\circ)^\times$ will happen. If $\bob$ chooses $I$ then the exchange ends. If $\bob$ chooses $A^\circ \otimes (A^\circ)^\times$ then $\alice$ sends $\bob$ and instance of $A$ and the two of them carry out $(A^\circ)^\times$ again from the beginning. In other words, $\bob$ can request any number of instances of $A$ from $\alice$.
\item To carry out $(A^\circ \times B^\bullet)^+ \in \ex{\A}_*$, first $\alice$ chooses which of $I$ and $(A^\circ \times B^\bullet) \otimes (A^\circ \times B^\bullet)^+$ will happen. If $\alice$ chooses $I$ then the exchange ends. If $\alice$ chooses $(A^\circ \times B^\bullet) \otimes (A^\circ \times B^\bullet)^+$ then $\bob$ chooses which of $A^\circ$ and $B^\bullet$ will happen, with $\alice$ sending $\bob$ an instance of $A$ if $\bob$ chooses $A^\circ$ and $\bob$ sending $\alice$ an instance of $B$ if he chooses $B^\bullet$. Then, the two of them carry out $(A^\circ \times B^\bullet)^+$ again from the beginning.
\end{itemize}

The key idea in our notion of iteration is the equation $U^\times = I \times (U \otimes U^\times)$, which allows us to use $\pi_0 : U^\times \to I$ and $\pi_1 : U^\times \to U \otimes U^\times$ to express properties of $U^\times$ . Dually, we will be able to use $\ip_0 : I \to U^+$ and $\ip_1 : U \otimes U^+ \to U^+$ to express properties of $U^+$. We proceed to extend the free cornering with the notion of iteration suggested by this:

\begin{definition}\label{def:iter-doub}
  Let $\A$ be a distributive monoidal category. We define the \emph{free cornering with iteration of $\A$}, written $\corner{\A}^*$, to be the free single-object double category with horizontal edge monoid $(\A_0,\otimes,I)$, vertical edge monoid $\ex{\A}_*$, and generating cells and equations consisting of:
  \begin{itemize}
  \item The generating cells and equations of $\corner{\A}^\oplus$ (Definition~\ref{def:free-cornering-choice}).

\item For each trio of cells $\alpha : \corner{\A}^*\cells{V}{A}{A}{U}$, $f : \corner{\A}^*\cells{W}{A}{B}{K}$, and $g : \corner{\A}^*\cells{W}{I}{I}{V \otimes W}$ a unique cell $\alpha^\times_{f,g} : \corner{\A}^*\cells{W}{A}{B}{U^\times \otimes K}$ satisfying:
    \begin{mathpar}
      \includegraphics[height=1.2cm,align=c]{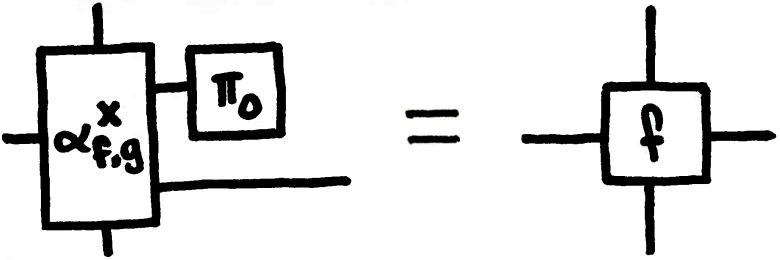}

      \includegraphics[height=1.2cm,align=c]{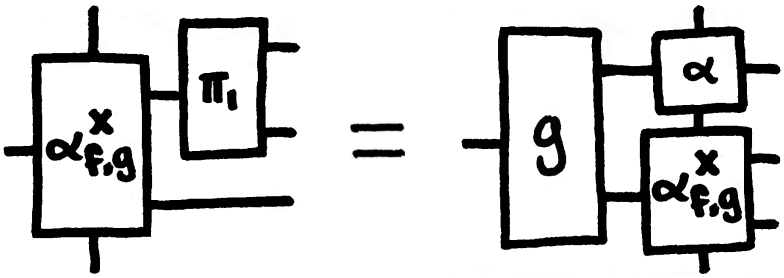}
    \end{mathpar}
  \item Dually, for each trio of cells $\alpha : \corner{\A}^*\cells{U}{A}{A}{V}$, $f : \corner{\A}^*\cells{K}{A}{B}{W}$, and $g : \corner{\A}^*\cells{V \otimes W}{I}{I}{W}$ a unique cell $\alpha^+_{f,g} : \corner{\A}^*\cells{U^+ \otimes K}{A}{B}{W}$ satisfying:
  \begin{mathpar}
    \includegraphics[height=1.2cm,align=c]{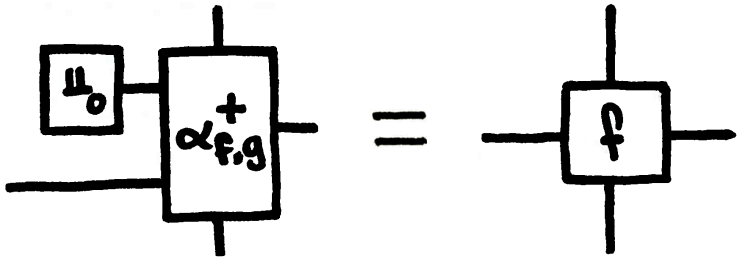}

    \includegraphics[height=1.2cm,align=c]{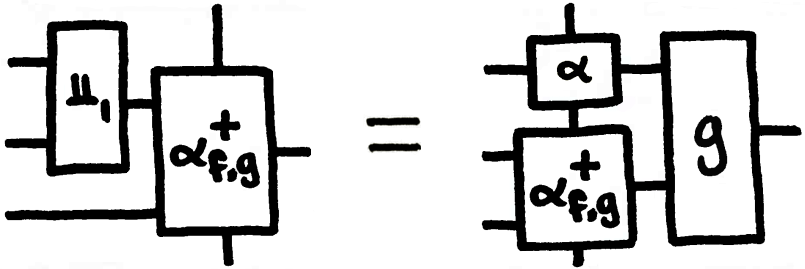}
  \end{mathpar}
\end{itemize}
\end{definition}

We extend our interpretation of cells of $\corner{\A}^\oplus$ as interacting processes to cells of $\corner{\A}^*$. Recall that $U^\times = I \times (U \otimes U^\times)$ is the exchange in which the right participant chooses whether the exchange is over (via $\pi_0$), or is to continue as $U \otimes U^\times$ (via $\pi_1$). The cells $\alpha^\times_{f,g}$ enable our procedures to be the left participant of such exchanges, reacting to the choices of the right participant as specified by the equations. Of course, $\alpha^+_{f,g}$ is the dual version for exchanges $U^+$, with the roles of the left and right participants swapped.

\begin{remark}\label{rem:simple-passive-iteration}
  In developing an intuition about this it is helpful to consider the following, simpler forms of these cell formation rules: Let $\alpha : \cells{U}{A}{A}{W}$. We define:
\begin{mathpar}
  \alpha^\times = \alpha^\times_{(\pi_0 \vert 1_A),\pi_1} : \cells{U^\times}{A}{A}{W^\times}

  \alpha^+ = \alpha^+_{(1_A \vert \ip_0),\ip_1} : \cells{U^+}{A}{A}{W^+}
\end{mathpar}
Then $\alpha^\times$ and is the unique cell such that:
\begin{mathpar}
  \alpha^\times \mid \pi_0 = \pi_0 \mid 1_A

  \alpha^\times \mid \pi_1 = \pi_1 \mid \frac{\alpha}{\alpha^\times}
\end{mathpar}
and similarly $\alpha^+$ is the unique cell such that:
\begin{mathpar}
  \ip_0 \mid \alpha^+ = 1_A \mid \ip_0

  \ip_1 \mid \alpha^+ = \frac{\alpha}{\alpha^+} \mid \ip_1
\end{mathpar}
Interpreted as an interacting process, $\alpha^\times$ reacts to the choice (made along the right boundary) to stop iterating by doing nothing to its inputs and propagating this choice along its left boundary. Similarly, $\alpha^\times$ reacts to the choice to continue iterating, performing $U$ once before performing $U^\times$ again, by acting as $\alpha$ once and then continuing as $\alpha^\times$, propagating this choice along its left boundary. The interpretation of $\alpha^+$ is similar. We note that both $(-)^+$ and $(-)^\times$ give functors $\bh\,\corner{\A}^* \to \bh\,\corner{\A}^*$. 
\end{remark}

\begin{remark}\label{rem:iter-coinduction}
  Before moving on we note that cells $\alpha^\times_{f,g}$, and $\alpha^+_{f,g}$ admit a coinductive reasoning principle. For $\alpha^\times_{f,g}$, we have that if $\gamma$ satisfies:
  \begin{mathpar}
    \gamma \mid \frac{\pi_0}{id} = f

    \gamma \mid \frac{\pi_1}{id} = g \mid \frac{\alpha}{\gamma}
  \end{mathpar}
  then $\gamma = \alpha^\times_{f,g}$, which we say holds by \emph{coinduction}. The coinductive reasoning principle for cells $\alpha^+_{f,g}$ is similar. 
\end{remark}

\begin{example}[Mealy Machines]
  Say that a \emph{mealy machine in a monoidal category $\A$} consists of a morphism $m : A \otimes S \to S \otimes B$. Then the classical notion of mealy machine is recoverable by considering mealy machine in the category of finite sets, with $S$ the set of states, $A$ the input alphabet, and $B$ the output alphabet. Mealy machines are usually understood to operate on a sequence of inputs drawn from $A$, producing a sequence of outputs drawn from $B$. The state of the machine is fed forward to future iterations.

  If $m : A \otimes S \to S \otimes B$ in $\A$ then let $M : \cells{A^\circ}{S}{S}{B^\circ}$ be the cell:
  \begin{mathpar}
    \includegraphics[height=1.2cm,align=c]{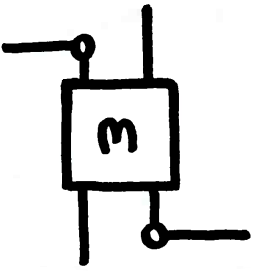}
  \end{mathpar}
  Then the cell $M^+ : \cells{(A^\circ)^+}{S}{S}{(B^\circ)^+}$ exhibits the behaviour of the process that the Mealy machine $m$ is intended to define, as in:
  \begin{mathpar}
    \includegraphics[height=1.2cm,align=c]{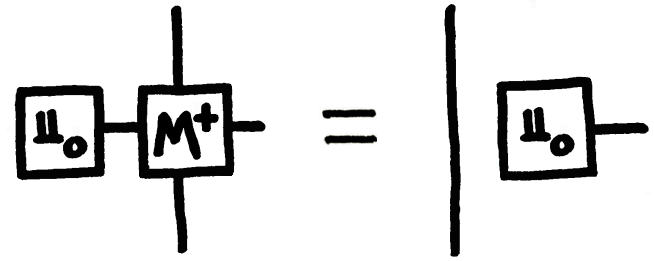}

    \includegraphics[height=1.5cm,align=c]{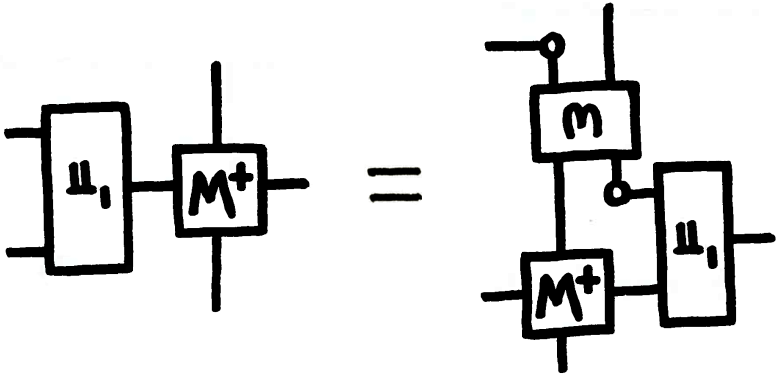}
  \end{mathpar}
  that is, if the there is no more input then the machine produces no more output, and if there is further input then the machine produces output accoring to $m$ and updates its internal state. 
\end{example}

\begin{example}[Memory Cell]
  Consider the cell $H = (\frac{\putr^A}{\getr^A})^\times : \cells{I}{A}{A}{(A^\circ \otimes A^\bullet)^\times}$. This cell behaves as follows:
  \begin{mathpar}
    \includegraphics[height=1.2cm,align=c]{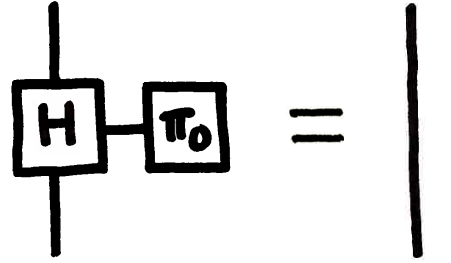}

    \includegraphics[height=1.2cm,align=c]{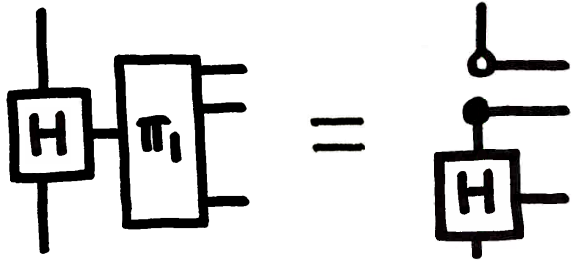}
  \end{mathpar}
  Think of $H$ as a simple sort of memory cell that stores a value of type $A$. When called upon by the environment along its right boundary, the cells outputs its contents, waits for its new contents to be supplied, and waits for further instructions (above right). The cell can also be told to stop (above left). 
\end{example}

\begin{example}\label{ex:bread}
  Suppose our base category $\A$ has objects $\texttt{bread}$ and $\texttt{\$}$, as well as objects representing stacks of each: $S_\texttt{bread} \cong I \oplus (\texttt{bread} \otimes S_\texttt{bread})$ and $S_\texttt{\$} \cong I \oplus (\texttt{\$} \otimes S_\texttt{\$})$ as in Section~\ref{subsec:distributive-branching}. We construct a process that sells bread, $\texttt{sales} : \corner{\A}^*\cells{(\texttt{\$}^\circ \otimes (\texttt{\$}^\bullet \times \texttt{bread}))^+}{S_\texttt{bread} \otimes S_\texttt{\$}}{S_\texttt{bread} \otimes S_\texttt{\$}}{I}$ as follows: Let $\texttt{sale} : \corner{\A}^*\cells{\texttt{\$}^\circ \otimes (\texttt{\$}^\bullet \times \texttt{bread})}{S_\texttt{bread} \otimes S_\texttt{\$}}{S_\texttt{bread} \otimes S_\texttt{\$}}{I}$ be the cell below on the left, with $\gamma_0$ and $\gamma_1$ the cells below on the right.
  \begin{mathpar}
    \includegraphics[height=1.7cm,align=c]{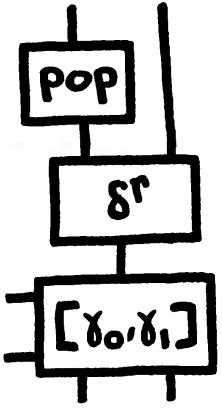}

    \includegraphics[height=1.2cm,align=c]{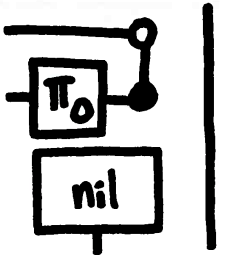}

    \includegraphics[height=1.2cm,align=c]{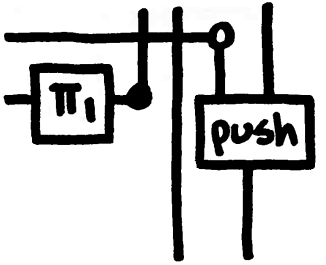}
  \end{mathpar}
  Then define $\texttt{sales} = \texttt{sale}^+_{1_{S_\texttt{bread}\otimes S_\texttt{\$}}, \square_I}$. We have: 
  \begin{mathpar}
    \includegraphics[height=1.2cm,align=c]{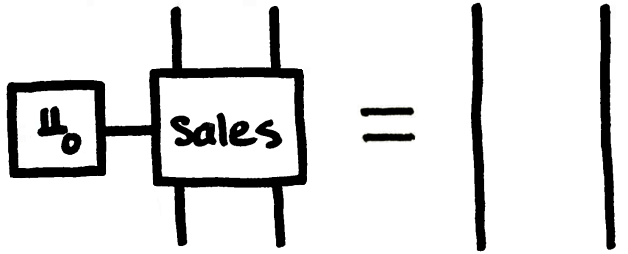}

    \includegraphics[height=1.7cm,align=c]{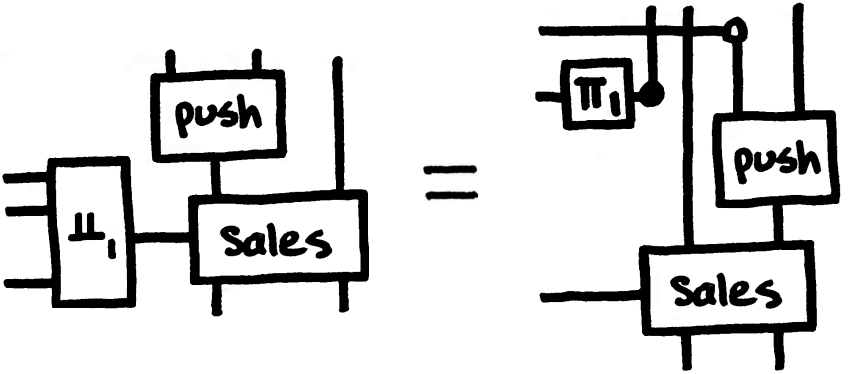}

    \includegraphics[height=1.7cm,align=c]{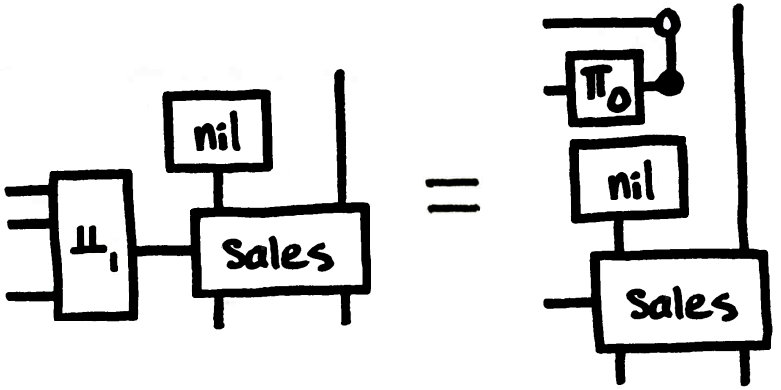}
  \end{mathpar}
  We imagine the left boundary of \texttt{sales} as aa sort of queue of customers, waiting to purchase bread. If there are no more customers ($\ip_0$) then \texttt{sales} simply retains its stacks of \texttt{bread} and \texttt{\$}. If there is at least one more customer ($\pi_1$) then \texttt{sales} receives $\texttt{\$}$ from the first customer in line. If no bread is available (the stack of bread is $\texttt{nil}$) then the money is returned. Otherwise the customer receives the first piece of bread on the stack. This process is repeated until no customers remain.
\end{example}

\subsection{Elementary Properties}\label{subsec:iteration-properties}
We proceed to establish some elementary properties of the free cornering with iteration. First we note that the properties of $\corner{\A}^\oplus$ established in Section~\ref{subsec:choice-properties} all hold in the free cornering with iteration, specifically Lemmas~\ref{lem:copairing-coincide}, \ref{lem:copairing-absorbs}, \ref{lem:branching-products-coproducts}, and \ref{lem:moral-equivalence-choice} all hold in $\corner{\A}^*$. 

Moving on to elementary properties of $\corner{\A}^*$ specifically, we show that $U^\times$ and $U^+$ arise as (co)algebras of a functor on the category of horizontal cells:
\begin{lemma}\label{lem:algebra-coalgebra}
  Consider the category $\bh\,\corner{\A}^*$ of horizontal cells of the free cornering with iteration. For all objects $U$ of $\bh\,\corner{\A}^*$, we have:
  \begin{enumerate}[(i)]
  \item $(U^\times,id_{U^\times} : U^\times \to U^\times = I \times (U \otimes U^\times))$ is the final coalgebra of the functor 
    \[
    I \times (U \otimes -) : \bh\,\corner{\A}^* \to \bh\,\corner{\A}^*
    \]
  \item $(U^+,id_{U^+} : U^+ \to U^+ = I + (U \otimes U^+))$ is the initial algebra of the functor
    \[
    I + (U \otimes -) : \bh\,\corner{\A}^* \to \bh\,\corner{\A}^*
    \]
  \end{enumerate}
\end{lemma}
\begin{proof}
\begin{enumerate}[(i)]
\item Suppose $(W,h : W \to I \times (U \otimes W))$ is a coalgebra for $(I \times (U \otimes -)$. We must show that there is a unique coalgebra morphism $(W,h) \to (U^\times,id_{U^\times})$ in $\bh\,\corner{\A}^*$. Define $\alpha = (id_U)^\times_{(h\mid\pi_0),(h\mid\pi_1)} : \cells{W}{I}{I}{U^\times}$. We must show that $\alpha$ gives a morphism of coalgebras. That is, we must show that in $\bh\,\corner{\A}^*$ we have:
  \[
  \begin{tikzcd}
    W \ar[d,"\alpha"']\ar[r,"h"] & I \times (U \otimes W) \ar[d,"(\pi_0 \times (\pi_1\mid\frac{id_U}{\alpha})) = (I \times (U \otimes -))(\alpha)"] \\
    U^\times \ar[r,"id_{U^\times}"']& U^\times = I \times (U \otimes U^\times)
  \end{tikzcd}
  \]
  This is because we have:
  \begin{align*}
    & h\mid(\pi_0 \times (\pi_1\mid\frac{id_U}{\alpha}))\mid\pi_0 
    = h\mid\pi_0
  \end{align*}
  and
  \begin{align*}
    & h\mid(\pi_0 \times (\pi_1\mid\frac{id_U}{\alpha}))\mid\pi_1
    = h\mid\pi_1\mid\frac{id_U}{\alpha}
  \end{align*}
  and so by coinduction we have that $h \mid (\pi_0 \times (\pi_1 \mid \frac{id_U}{\alpha})) = (1_U)^\times_{(h\mid\pi_0),(h\mid\pi_1)} = \alpha$, and so $\alpha : (W,h) \to (U^\times,id_{U^\times})$ is a morphism of coalgebras. It remains to show that $\alpha$ is the unique such coalgebra morphism. To that end, suppose that $\beta : (W,h) \to (U^\times,id_{U^\times})$ is a coalgebra morphism. That is, suppose we have:
  \[
  \begin{tikzcd}
    W \ar[d,"\beta"'] \ar[r,"h"] & I \times (U \otimes W) \ar[d,"(\pi_0 \times (\pi_1\mid\frac{id_U}{\beta})) = (I \times (U \otimes -))(\beta)"]\\
    U^\times \ar[r,"id_{U^\times}"] & U^\times = I \times (U \otimes U^\times)
  \end{tikzcd}
  \]
  Then we have:
  \begin{align*}
    & h \mid (\pi_0 \times (\pi_1\mid\frac{id_U}{\beta}))\mid\pi_0
    = h\mid\pi_0
  \end{align*}
  and
  \begin{align*}
    h\mid(\pi_0 \times (\pi_1\mid\frac{id_U}{\beta}))\mid\pi_1
    = h \mid \pi_1\mid\frac{id_U}{\beta}
  \end{align*}
  and so by coinduction we have that $\beta = (1_U)^\times_{(h\mid\pi_0),(h\mid\pi_1)} = \alpha$, as required.
\item Similar to the proof of (i).
\end{enumerate}
\end{proof}

Further, we exhibit (co)monoid structures on our iterated protocol types and show that they enjoy a kind of naturality:
\begin{lemma}\label{lem:horizontal-comonoid}
  For all objects $U$ of $\bh\,\corner{\A}^*$:
  \begin{enumerate}[(i)]
  \item $(U^\times, \Delta^\times_U, \pi_0)$ is a comonoid in $\bh\,\corner{\A}^*$ where $\Delta^\times_U = (id_U)^\times_{id_{U^\times},\pi_1}$. Dually, $(U^+,\nabla^+_U,\ip_0)$ is a monoid in $\bh\,\corner{\A}^*$ where $\nabla^+_U = (id_U)^+_{id_{U^+},\ip_1}$. 
  \item For any $h : \cells{U}{I}{I}{W}$,  $\Delta^\times_U \mid \frac{h^\times}{h^\times} = h^\times \mid \Delta^\times_W$. Dually, $\nabla^+_U \mid h^+ = \frac{h^+}{h^+}\mid \nabla^+_W$. 
  \end{enumerate}
\end{lemma}
\begin{proof}
  \begin{enumerate}[(i)]
  \item We must show that $(U^\times,\Delta^\times,\pi_0)$ is coassociative and counital. For coassociativity, we have:
    \begin{mathpar}
      \includegraphics[height=1cm,align=c]{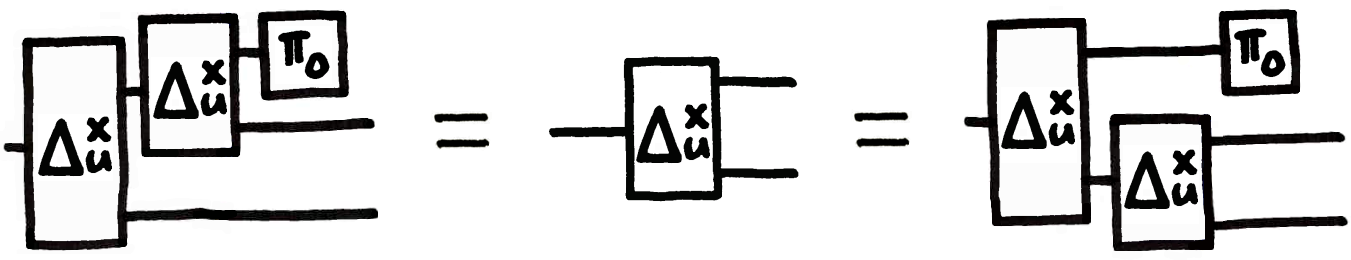}

      \includegraphics[height=1cm,align=c]{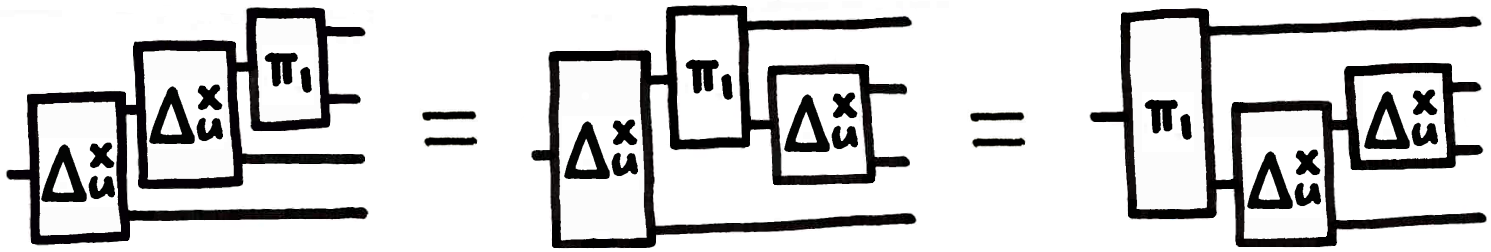}

      \includegraphics[height=1cm,align=c]{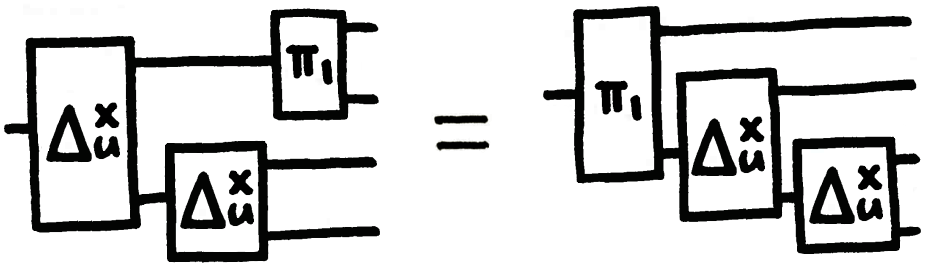}
    \end{mathpar}
    and then by coinduction we have:
    \begin{mathpar}
      \includegraphics[height=1cm,align=c]{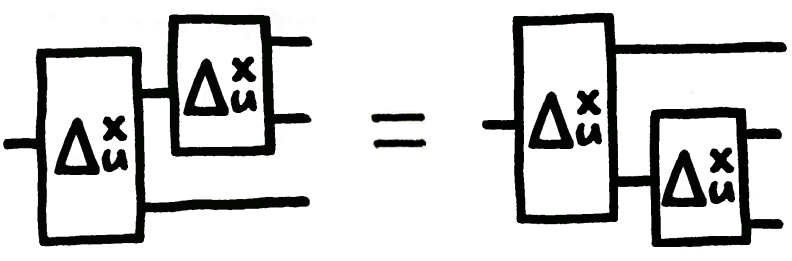}
    \end{mathpar}
    as required. The first counitality axiom holds immediately:
    \begin{mathpar}
      \includegraphics[height=1cm,align=c]{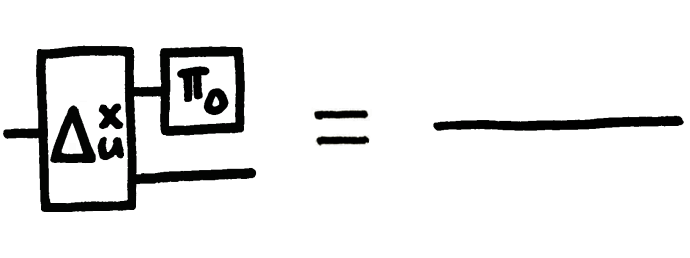}
    \end{mathpar}
    For the second counitality axiom, we have:
    \begin{mathpar}
      \includegraphics[height=1cm,align=c]{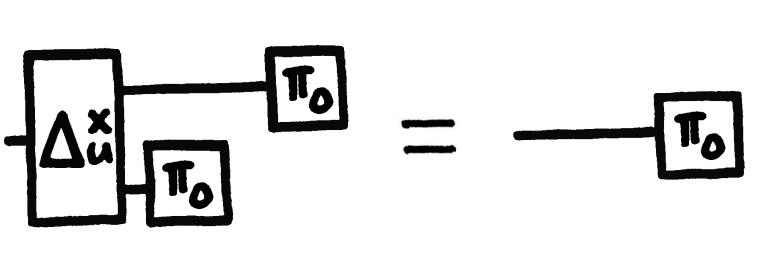}

      \includegraphics[height=1cm,align=c]{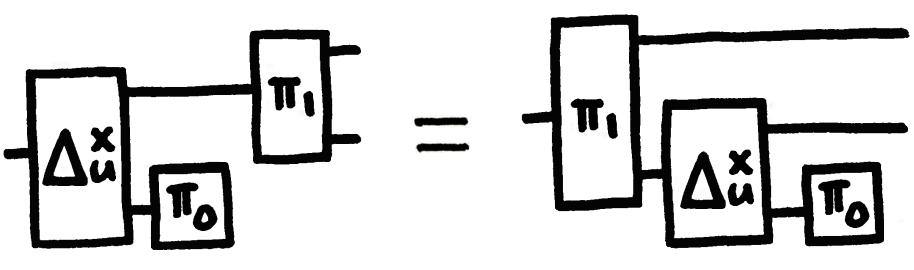}
    \end{mathpar}
    and so, since $id_{U^\times}\mid \pi_1 = \pi_1 \mid \frac{id_{U^\times}}{id_{U^\times}}$, we have by coinduction that the second unitality axiom holds, as in:
    \begin{mathpar}
      \includegraphics[height=1cm,align=c]{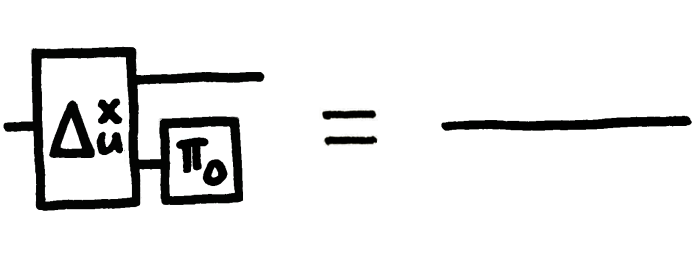}
    \end{mathpar}
    It follows that $(U^\times,\Delta^\times_U,\pi_0)$ is a comonoid. The proof that $(U^+,\nabla^+_U,\ip_0)$ is a monoid is similar. 
  \item Suppose $h : \cells{U}{I}{I}{W}$. Then we have:
    \begin{mathpar}
      h^\times \mid \Delta^\times_W \mid \frac{\pi_0}{id_{U^\times}}
        = h^\times
        = \Delta^\times_U\mid \frac{\pi_0}{h^\times}
        = \Delta^\times_U\mid\frac{h^\times}{h^\times}\mid\frac{\pi_0}{id_{U^\times}}

        h^\times \mid \Delta^\times_W \mid \frac{\pi_1}{id_{W^\times}}
        = h^\times \mid \pi_1\mid \frac{id_W}{\Delta^\times_W}
        = \pi_1 \mid \frac{h}{h^\times}\mid\frac{id_W}{\Delta^\times_W}
        = \pi_1 \mid\frac{h}{h^\times\mid\Delta^\times_W}

        \Delta^\times_U \mid \frac{h^\times}{h^\times} \mid\frac{\pi_1}{id_{U^\times}}
        = \Delta^\times_U\mid \frac{\pi_1\mid\frac{h}{h^\times}}{h^\times}
        = \pi_1 \mid \frac{h}{\Delta^\times_U \mid \frac{h^\times}{h^\times}}
    \end{mathpar}
    and so by coinduction $\Delta^\times_U \mid \frac{h^\times}{h^\times} = h^\times \mid \Delta^\times_W$, as promised. The proof that $\nabla^+_U \mid h^+ = \frac{h^+}{h^+} \mid \nabla^+_W$ is similar. 
  \end{enumerate}
\end{proof}

We end our discussion of the elementary properties of $\corner{\A}^*$ by showing that the functors $(-)^\times$ and $(-)^+$ of Remark~\ref{rem:simple-passive-iteration} are (co)monads:
\begin{lemma}\label{lem:monoid-comonoid}
  Consider the category $\bh\,\corner{\A}^*$. We have:
  \begin{enumerate}[(i)]
  \item The functor
    \[
    (-)^\times : \bh\,\corner{\A}^* \to \bh\,\corner{\A}^*
    \]
    is a comonad with counit $\varepsilon^\times : (-)^\times \to 1_{\bh\,\corner{\A}^*}$ given by components $\varepsilon^\times_U = \pi_1 \mid \frac{id_U}{\pi_0} : \cells{U^\times}{I}{I}{U}$ and comultiplication $\delta^\times : (-)^{\times} \to (-)^{\times\times}$ given by components $\delta^\times_U = (id_{U^\times})^\times_{\pi_0,\Delta^\times_U} : \cells{U^\times}{I}{I}{U^{\times\times}}$.
  \item The functor
    \[
    (-)^+ : \bh\,\corner{\A}^* \to \bh\,\corner{\A}^*
    \]
    is a monad with unit $\eta^+ : 1_{\bh\,\corner{\A}^*} \to (-)^+$ given by components $\eta^+_U  = \frac{id_U}{\ip_0}\mid\ip_1 : \cells{U}{I}{I}{U^+}$ and comultiplication $\mu^+ : (-)^{++} \to (-)^+$ given by components $\mu^+_U = (id_{U^+})^+_{\ip_0,\nabla^+_U} : \cells{U^{++}}{I}{I}{U^+}$.
  \end{enumerate}
\end{lemma}
\begin{proof}
  \begin{enumerate}[(i)]
  \item It is straightforward to verify that $(-)^\times$ is a functor. In order to prove that it is a comonad we first show that $\varepsilon^\times$ and $\delta^\times$ are natural. Explicitly, we require:
  \begin{mathpar}
    \begin{tikzcd}
      U^\times \ar[d,"h^\times"'] \ar[r,"\varepsilon^\times_U"] & U \ar[d,"h"] \\
      W^\times \ar[r,"\varepsilon^\times_W"'] & W
    \end{tikzcd}

    \begin{tikzcd}
      U^\times \ar[r,"\delta^\times_U"] \ar[d,"h^\times"'] & U^{\times\times} \ar[d,"h^{\times\times}"] \\
      W^\times \ar[r,"\delta^\times_W"'] & W^{\times\times}
    \end{tikzcd}
  \end{mathpar}
  for any $h : U \to W$ of $\bh\,\corner{\A}$. For $\varepsilon^\times$ we have:
  \begin{mathpar}
    h^\times \mid \varepsilon^\times_W
    = h^\times \mid \pi_1 \mid \frac{id_W}{\pi_0}
    = \pi_1 \mid \frac{h}{h^\times}\mid \frac{id_U}{\pi_0}
    = \pi_1 \mid \frac{h}{\pi_0}
    = \varepsilon^\times_U \mid h
  \end{mathpar}
  as required. For $\delta^\times$ we have:
  \begin{mathpar}
    h^\times \mid \delta^\times_W \mid \pi_0 = h^\times \mid \pi_0 = \pi_0 = \delta^\times_U \mid \pi_0 = \delta^\times_U \mid h^{\times\times}\mid \pi_0

    h^\times \mid \delta^\times_W \mid \pi_1
    = h^\times \mid \Delta^\times_W \mid \frac{id_{W^\times}}{\delta^\times_W}
    = \Delta^\times_U \mid \frac{h^\times}{h^\times} \mid \frac{id_{W^\times}}{\delta^\times_W}
    = \Delta^\times_U \mid \frac{h^\times}{h^\times \mid \delta^\times_W}

    \delta^\times_U \mid h^{\times\times} \mid \pi_1
    = \delta^\times_U \mid \pi_1 \mid \frac{h^\times}{h^{\times\times}}
    = \Delta^\times_U\mid\frac{id_{U^\times}}{\delta^\times_U}\mid\frac{h^\times}{h^{\times\times}}
    = \Delta^\times_U\mid\frac{h^\times}{\delta^\times_U\mid h^{\times\times}}
  \end{mathpar}
  and then by coinduction we have $h^\times \mid \delta^\times_W = \delta^\times_U \mid h^{\times\times}$ as required.

  It remains to show that the comonad axioms are satisfied. That is, we require:
  \begin{mathpar}
    \begin{tikzcd}
      U^\times \ar[d,"\delta^\times_U"'] \ar[r,"\delta^\times_U"] & U^{\times\times} \ar[d,"\delta^\times_{U^\times}"] \\
      U^{\times\times} \ar[r,"(\delta^\times_U)^\times"'] & U^{\times\times\times}
    \end{tikzcd}

    \begin{tikzcd}
      U^\times & U^{\times\times} \ar[l,"\varepsilon^\times_{U^\times}"' ]\ar[r,"(\varepsilon^\times_U)^\times"] & U^\times \\
      \text{} & U^\times \ar[u,"\delta^\times_U"] \ar[ur,"id_{U^\times}"'] \ar[ul, "id_{U^\times}"] 
    \end{tikzcd}
  \end{mathpar}
  Notice that by coinduction $\delta^\times_U \mid \Delta^\times_{U^\times} = \Delta^\times_U \mid \frac{\delta^\times_U}{\delta^\times_U}$ although, somewhat misleadingly, not as a consequence of Lemma~\ref{lem:horizontal-comonoid}. For coassociativity of $\delta^\times$, we have:
  \begin{mathpar}
    \delta^\times_U \mid \delta^\times_{U^\times} \mid \pi_0 = \delta^\times_U \mid \pi_0 = \delta^\times_U \mid (\delta^\times_U)^\times \mid \pi_0

    \delta^\times_U \mid \delta^\times_{U^\times} \mid \pi_1
    = \delta^\times_U \mid \Delta^\times_{U^\times} \mid \frac{id_{U^{\times\times}}}{\delta^\times_{U^\times}}
      = \Delta^\times_U \mid \frac{\delta^\times_U}{\delta^\times_U}\mid\frac{id_{U^{\times\times}}}{\delta^\times_{U^\times}}
      = \Delta^\times_U \mid \frac{\delta^\times_U}{\delta^\times_U \mid \delta^\times_{U^\times}}

      \delta^\times_U \!\mid\! (\delta^\times_U)^\times \mid \pi_1
      = \delta^\times_U \!\mid\! \pi_1 \!\mid\! \frac{\delta^\times_U}{(\delta^\times_U)^\times}
      = \Delta^\times_U\mid\frac{id_{U^\times}}{\delta^\times_U}\mid\frac{\delta^\times_U}{(\delta^\times_U)^\times}
      = \Delta^\times_U \! \mid \frac{\delta^\times_U}{\delta^\times_U \!\mid\! (\delta^\times_U)^\times}
  \end{mathpar}
  and so by coinduction we have $\delta^\times_U \mid \delta^\times_{U^\times} = \delta^\times_U \mid (\delta^\times_U)^\times$ as required. For the first counit law, we have:
  \begin{mathpar}
    \delta^\times_U \mid \varepsilon^\times_{U^\times} \mid \pi_0
    = \delta^\times_U \mid \pi_1 \mid \frac{\pi_0}{\pi_0}
    = \Delta^\times_U \mid \frac{\pi_0}{\delta^\times_U \mid \pi_0}
    = \pi_0
    = id_{U^\times}\mid \pi_0

    \delta^\times_U \mid \varepsilon^\times_{U^\times} \mid \pi_1
    = \delta^\times_U \mid \pi_1 \mid \frac{\pi_1}{\pi_0}
    = \Delta^\times_U \mid \frac{\pi_1}{\delta^\times_U \mid \pi_0}
    = \pi_1 \mid \frac{id_{U}}{\Delta^\times_U\mid\frac{id_{u^\times}}{\pi_0}}
    = \pi_1
    = id_{U^\times} \mid \pi_1
  \end{mathpar}
  and then since $id_{U^\times} \mid \pi_1 = \pi_1\mid\frac{id_U}{id_{U^\times}}$ we have $\delta^\times\mid \varepsilon^\times_{U^\times} = id_{U^\times}$ by coindution. For the second counit law, we have:
  \begin{mathpar}
    \delta^\times_U \mid (\varepsilon^\times_U)^\times \mid \pi_0
    = \delta^\times_U \mid \pi_0
    = \pi_0
    = id_{U^\times}\mid\pi_0

    \delta^\times_U \mid (\varepsilon^\times_U)^\times \mid \pi_1
    = \delta^\times_U \mid \pi_1 \mid \frac{\varepsilon^\times_U}{(\varepsilon^\times_U)^\times}
    = \Delta^\times_U\mid\frac{id_{U^\times}}{\delta^\times_U}\mid\frac{\varepsilon^\times_U}{(\varepsilon^\times_U)^\times} \\
    = \Delta^\times_U \mid \frac{\pi_1\mid\frac{id_U}{\pi_0}}{\delta^\times_U \mid (\varepsilon^\times_U)^\times}
    = \pi_1 \mid \frac{id_U}{\Delta^\times_U \mid \frac{\pi_0}{\delta^\times_U \mid (\varepsilon^\times_U)^\times}}
    = \pi_1 \mid \frac{id_U}{\delta^\times_U \mid (\varepsilon^\times_U)^\times}
  \end{mathpar}
  and then since $id_{U^\times}\mid \pi_1 = \pi_1 \mid \frac{id_U}{id_{U^\times}}$ we have that $\delta^\times_U(\varepsilon^\times_U)^\times = id_{U^\times}$ by coinduction. Thus, $((-)^\times,\delta^\times,\varepsilon^\times)$ is a comonad on $\bh\,\corner{\A}^*$.
\item Similar to the proof of (i).
\end{enumerate}
\end{proof}

\subsection{Crossing Cells}\label{subsec:iteration-crossing}

We extend Definition~\ref{def:choice-crossing} to obtain crossing cells in the free cornering with iteration:
\begin{definition}
Let $\A$ be a distributive monoidal category. For each $A \in \corner{\A}^*_H$ and each $U \in \corner{\A}^*_V$ We define crossing a crossing cell $\chi_{U,A} : \cells{U}{A}{A}{U}$ by induction on the structure of $U$. The cases for $A^\circ, A^\bullet, I, U \otimes W, U \times W$, and $U + W$ are as in Definition~\ref{def:choice-crossing}. For $U^\times$ we define $\chi_{U^\times,A} = (\chi_{U,A})^\times$ and for $U^+$ we define $\chi_{U^+,A} = (\chi_{U,A})^+$. That is, $\chi_{U^\times,A}$ is the unique cell such that:
\begin{mathpar}
  \includegraphics[height=1.2cm,align=c]{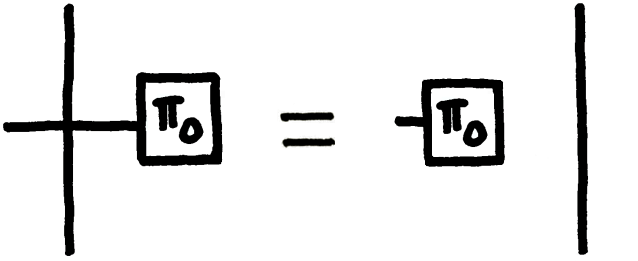}

  \includegraphics[height=1.2cm,align=c]{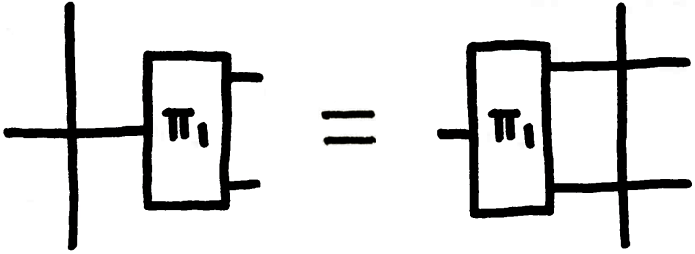}
\end{mathpar}
Similarly, $\chi_{U^+,A}$ is the unique cell such that:
\begin{mathpar}
  \includegraphics[height=1.2cm,align=c]{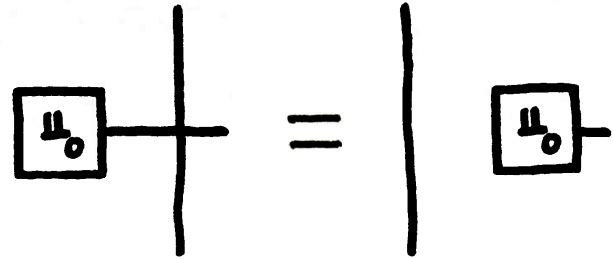}

  \includegraphics[height=1.2cm,align=c]{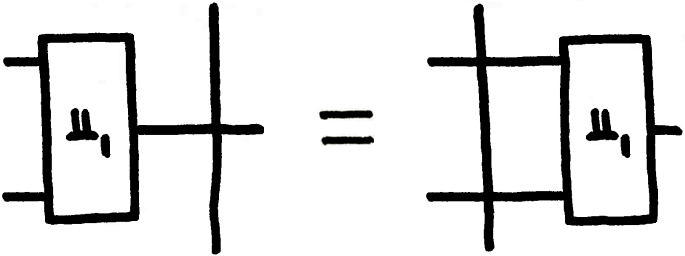}
\end{mathpar}
\end{definition}

The crossing cells remain coherent with respect to horizontal composition:
\begin{lemma}\label{lem:iteration-crossing-horiz-coherent}
  For $U \in \ex{A}_*$ and $A,B \in \A_0$ we have
  \begin{enumerate}[(i)]
  \item $\chi_{U,{A \otimes B}} = \chi_{U,A} \mid \chi_{U,B}$
  \item $\chi_{U,I} = 1_U$
  \end{enumerate}
\end{lemma}
\begin{proof}
  We extend the proof of Lemma~\ref{lem:coherent-crossing-choice} with the necessary inductive cases:
  \begin{enumerate}[(i)]
  \item For $U^\times$ we have:
    \begin{mathpar}
      \chi_{U^\times,A \otimes B} \mid \pi_0 = \pi_0 \mid 1_{A \otimes B} = \pi_0 \mid 1_A \mid 1_B = \chi_{U^\times,A} \mid \chi_{U^\times,B} \mid \pi_0

      \chi_{U^\times,A \otimes B} \mid \pi_1 = \pi_1 \mid \frac{\chi_{U,A \otimes B}}{\chi_{U^\times,A \otimes B}} = \pi_1 \mid \frac{\chi_{U,A} \mid \chi_{U,B}}{\chi_{U^\times,A \otimes B}}

      \chi_{U^\times,A} \mid \chi_{U^\times,B} \mid \pi_1  = \pi_1 \mid \frac{\chi_{U,A}}{\chi_{U^\times,A}} \mid \frac{\chi_{U,B}}{\chi_{U^\times,B}} = \pi_1 \mid \frac{\chi_{U,A} \mid \chi_{U,B}}{\chi_{U^\times,A} \mid \chi_{U^\times,B}}
    \end{mathpar}
    and so $\chi_{U^\times,A \otimes B} = \chi_{U^\times,A} \mid \chi_{U^\times,B}$ by coinduction. A similar argument gives $\chi_{U^+,A \otimes B} = \chi_{U^+,A} \mid \chi_{U^+,B}$. 
  \item For $U^\times$ we have:
    \begin{mathpar}
      \chi_{U^\times,I} \mid \pi_0 = \pi_0 \mid 1_I = \pi_0 = id_{U^\times} \mid \pi_0

      \chi_{U^\times,I} \mid \pi_1 = \pi_1 \mid \frac{\chi_{U,I}}{\chi_{U^\times,I}} = \pi_1 \mid \frac{id_{U}}{\chi_{U^\times,I}}

      id_{U^\times} \mid \pi_1 = \pi_1 = \pi_1 \mid id_{U \otimes U^\times} = \pi_1 \mid \frac{id_U}{id_{U^\times}}
    \end{mathpar}
    It follows that $\chi_{U^\times,I} = id_{U^\times}$. The case for $U^+$ is similar.
  \end{enumerate}
\end{proof}

Next, we show that the technical lemma concerning crossing cells holds in the free cornering with iteration:
\begin{lemma}\label{lem:crossing-swaps-iteration}
  For any cell $\alpha$ of $\corner{\A}^*$ we have
  \[
  \includegraphics[height=1.2cm]{figs/crossing-swaps.png}
  \]
\end{lemma}
\begin{proof}
  We extend the proof of Lemma~\ref{lem:choice-crossing-swaps} with the necessary inductive cases. For cells $\alpha^\times_{f,g}$ we have:
  \begin{mathpar}
    \includegraphics[height=1.4cm,align=c]{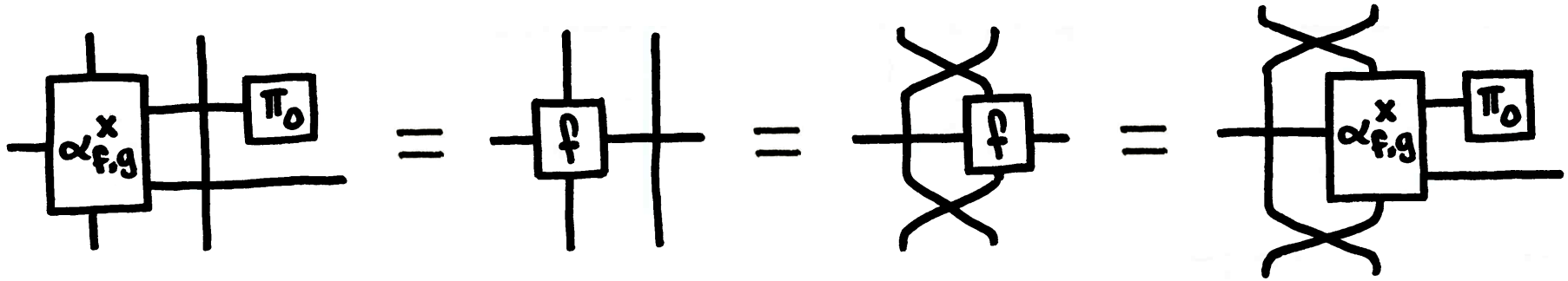}

    \includegraphics[height=1.2cm,align=c]{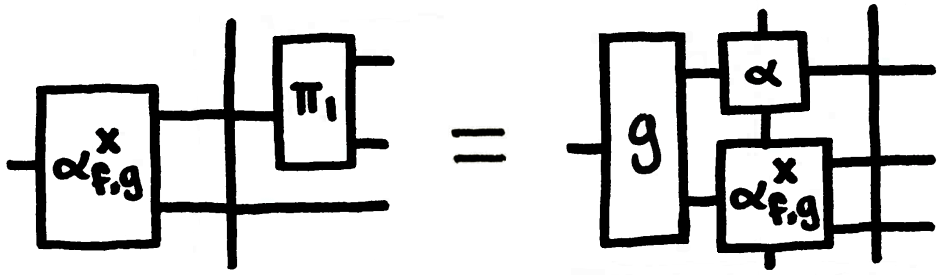}

    \includegraphics[height=1.8cm,align=c]{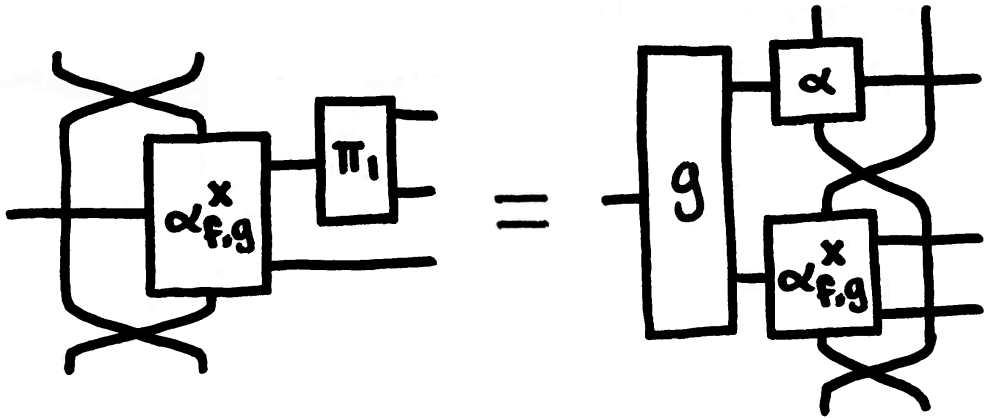}
  \end{mathpar}
  which gives, using the proof technique of Remark~\ref{rem:iter-coinduction}:
  \begin{mathpar}
    \includegraphics[height=1.4cm,align=c]{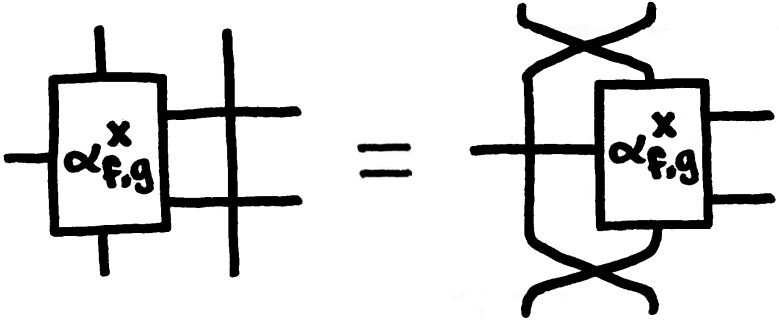}
  \end{mathpar}
  as required. The case for cells $\alpha^+_{f,g}$ is similar. 
\end{proof}

Consequently, $\corner{\A}^*$ is a monoidal double category with the tensor product of cells and proof as in Lemma~\ref{lem:monoidaldouble}. We record:
\begin{lemma}\label{lem:iterationmonoidaldouble}
  If $\A$ is a distributive monoidal category then $\corner{\A}^*$ is a monoidal double category. 
\end{lemma}

Further, the crossing cells remain coherent with respect to $\oplus$ in $\corner{\A}^*$:
\begin{lemma}\label{crossing-copairing-coherent-iteration}
  In $\corner{\A}^*$, $\chi_{U,A \oplus B} = \left[\frac{\chi_{U,A}}{\sigma_0},\frac{\chi_{U,B}}{\sigma_1}\right]$. That is, $\chi_{U,A \oplus B}$ is the unique cell such that:
  \begin{mathpar}
    \frac{\sigma_0}{\chi_{U,A\oplus B}} = \frac{\chi_{U,A}}{\sigma_0}

    \frac{\sigma_1}{\chi_{U,A\oplus B}} = \frac{\chi_{U,B}}{\sigma_1}
  \end{mathpar}
\end{lemma}
\begin{proof}
  By structural induction on $U$. We supply the necessary inductive cases to extend the proof of Lemma~\ref{lem:coherent-oplus-choice} to a proof of the present claim. For $U^\times$ we have:
  \begin{mathpar}
    \includegraphics[height=1.6cm,align=c]{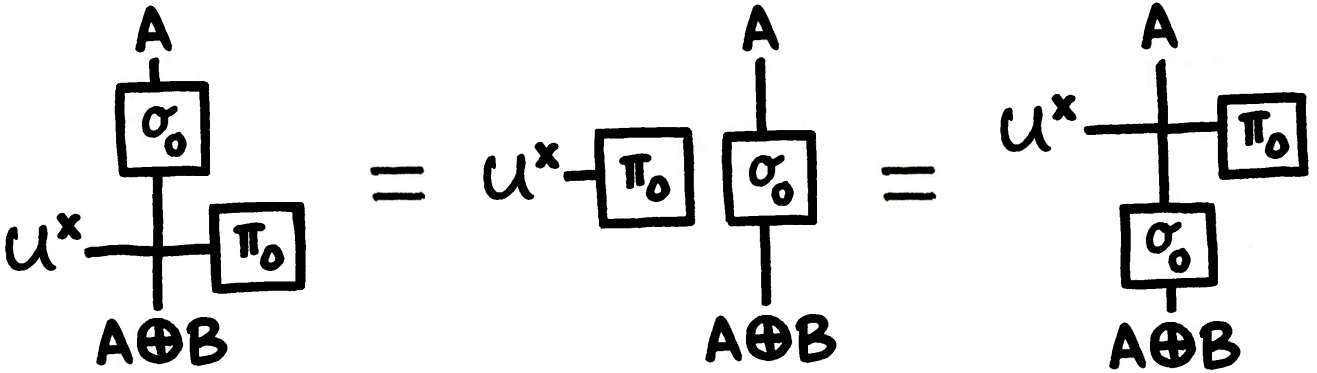}

    \includegraphics[height=1.6cm,align=c]{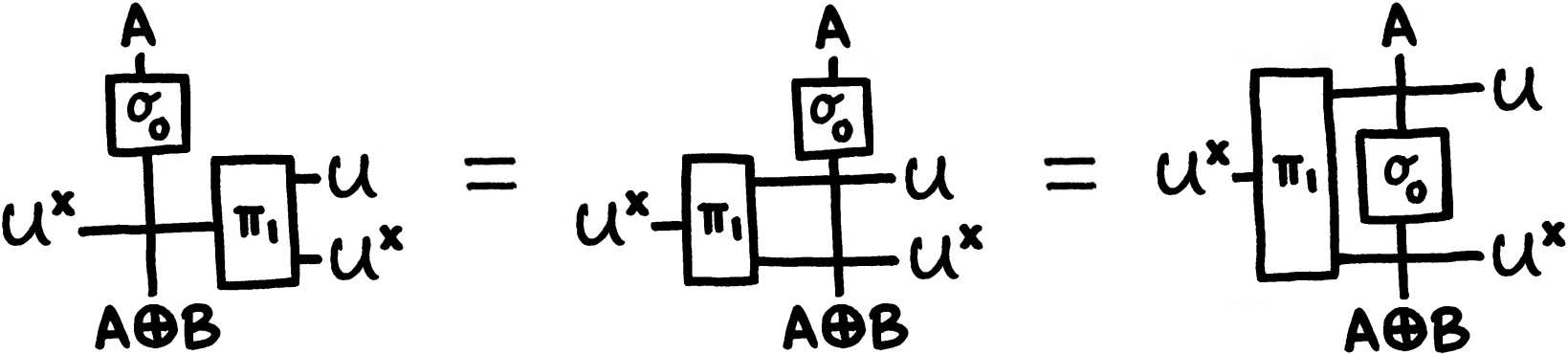}

    \includegraphics[height=1.6cm,align=c]{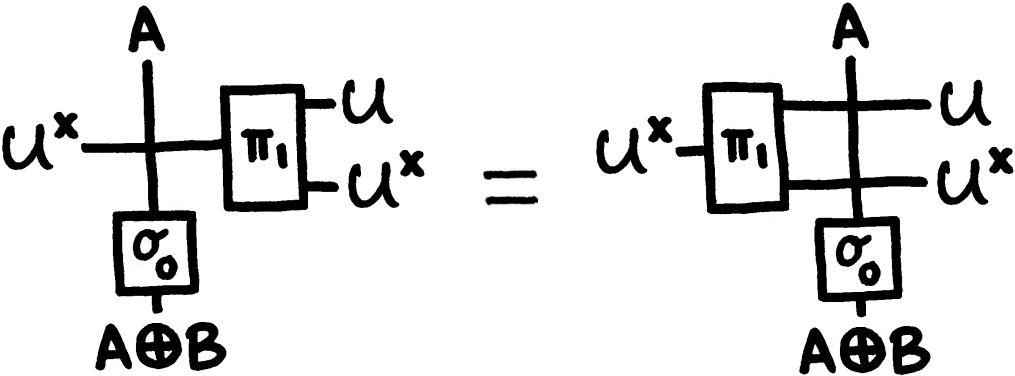}
  \end{mathpar}
  and so by coinduction we have $\frac{\sigma_0}{\chi_{U^\times,A \oplus B}} = \frac{\chi_{U^\times,A}}{\sigma_0}$. Similarly, $\frac{\sigma_1}{\chi_{U^\times,A\oplus B}} = \frac{\chi_{U^\times,B}}{\sigma_1}$. The case for $U^+$ is analogous. 
\end{proof}

\subsection{A Model: Iteration in Stateful Transformations}\label{subsec:stateful-iteration}
We return to the double category $\Dst{\C}$ of stateful transformations over a cartesian closed category with distributive binary coproducts. In order to define $F^+$ for a strong endofunctor $F$ in $\C^\C_\tau$ we require the existence of an initial algebra for the functor $(X \oplus F(-)) : \C \to \C$ for each object $X$ of $\C$. Then we may define $F^+$ to be the functor mapping $X$ to the carrier of the corresponding initial algebra $a^+_{F,X} : X \oplus F(F^+X) \to F^+ X$. Dually, in order to define $F^\times$ we require the existence of a final coalgebra for the functor $(X \otimes F(-)) : \C \to \C$ for each object $X$ of $\C$. Then we may define $F^\times$ to be the functor mapping $X$ to the carrier of the corresponding final coalgebra $c^\times_{F,X} : F^\times X \to X \otimes F(F^\times X)$.

We show that when they exist, both $F^+$ and $F^\times$ are strong. For $F^\times$, this can be shown by constructing a coalgebra for $(X \otimes Y \otimes F(-))$ with carrier $F^{\times}X \otimes Y$, thereby constructing a coalgebra morphism: $\tau^{F^{\times}}_{X,Y} : F^{\times}X \otimes Y \to F^{\times}(X \otimes Y)$ which acts as the strength. 
The coalgebra is defined as follows:
\[F^{\times}X \otimes Y \xrightarrow{c^{\times} \otimes Y} X \otimes F(F^{\times}X) \otimes Y \xrightarrow{X \otimes \langle \pr , \tau^F \rangle } (X \otimes Y \otimes F(F^{\times}X \otimes Y))
\]

Dually, $F^{+}$ is proven to be strong by defining an algebra for $(X \oplus F(-))$ with carrier $(F^+(X \otimes Y))^Y$, thereby constructing an algebra morphism $m : F^+X \to (F^+(X \otimes Y))^Y$ which induces the strength $\tau^{F^+}_{X,Y} = (m \otimes Y) \, \ev^Y : F^+X \otimes Y \to F^+(X \otimes Y)$.
The algebra is defined as
\[
[\lambda[l_{X,Y}], \lambda[r_{X,Y}]] : X \oplus F((F^+(X \otimes Y))^Y) \to (F^+(X \otimes Y))^Y
\]
where
\[
l_{X,Y} : X \otimes Y 
\xrightarrow{\sigma_0}
(X \otimes Y) \oplus F(F^+(X \otimes Y))
\xrightarrow{a^+_{F,X \otimes Y}} 
F^+(X \otimes Y)
\]
\[
r_{X,Y} : F((F^+(X \otimes Y))^Y) \otimes Y \xrightarrow{\tau^F \, F(\ev^Y)} 
F(F^+(X \otimes Y))
\xrightarrow{\sigma_1 \, a^+_{F,X \otimes Y}}
F^+(X \otimes Y)	
\]
It follows that $F^+$ and $F^\times$ are strong, as promised.

Now, given cells $\alpha : \cells{G}{A}{A}{H}$, $g : \cells{F}{1}{1}{G \circ F}$ and $f : \cells{F}{A}{B}{K}$ of $\mathsf{S}(\C)$ such that $H^\times$ exists we may take $\alpha^{\times}_{f,g} : \cells{F}{A}{B}{H^{\times} \circ K}$ to have components $(\alpha^{\times}_{f,g})_X : FX \otimes A \to H^{\times}K(X \otimes B)$ given by the unique coalgebra morphism from the coalgebra $\langle f, (g \otimes A) \alpha \rangle : FX \otimes A \to K(X \otimes B) \otimes H(FX \otimes A)$ to $c^{\times}_{H, K(X \otimes B)}$ .

The dual case is slightly more involved. Given cells $\alpha : \cells{F}{A}{A}{G}$, $g : \cells{G \circ H}{1}{1}{H}$ and $f : \cells{K}{A}{B}{H}$ of $\mathsf{S}(\C)$ such that $F^+$ exists we may construct $\alpha^+_{f,g} : \cells{F^+ \circ K}{A}{B}{H}$ as follows. First, for each object $X$ of $\C$ let $h_X$ be defined by:
\[
h_X = F((H(X \otimes B))^A) \otimes A \xrightarrow{\alpha \, G(\ev^A_{X \otimes B})} GH(X \otimes B) \xrightarrow{g_{X \otimes B}} H(X \otimes B) 
\]
Then $\lambda[f_X] : KX \to (H(X \otimes B))^A$ and $\lambda[h_X] : F((H(X \otimes B))^A) \to (H(X \otimes B))^A$ gives us an algebra: $[\lambda[f_X], \lambda[h_X]] : KX \oplus F((H(X \otimes B))^A) \to (H(X \otimes B))^A$ of $KX \oplus F(-) : \C \to \C$. Since $a^+_{F,KX}$ is the initial algebra of $KX \oplus F(-) : \C \to \C$, we have an algebra morphism $m : F^+(KX) \to (H(X \otimes B))^A$ from $a^+_{F,KX}$ to $[\lambda[f_X],\lambda[h_X]]$. Define $(\alpha^+_{f,g})_X = (m \otimes A)\ev^A_{H(X \otimes B)} : F^+(KX) \otimes A \to H(X \otimes B)$. 

In fact, $\alpha^\times_{f,g}$ and $\alpha^+_{f,g}$ define strong natural transformations, although the proof is somewhat involved. It follows that we may interpret $\corner{\C}^*$ into $\mathsf{S}(\C)$ to the extent that $F^+$ and $F^\times$ exist.

In general, $F^+$ and $F^\times$ need not exist in $\mathsf{S}(\C)$, although in at least one instance it is possible to extend the double functor of Lemma~\ref{lem:doublefunctor} and Lemma~\ref{lem:doublefunctorchoice} to the free cornering with iteration. Specifically, the \emph{containers} over the category $\mathsf{Set}$ of sets and functions\footnote{Admittedly the monoidal structure on $\mathsf{Set}$ is not typically taken to be strict, so to fit the setting of this paper one would technically have to work with its strictification.} are strong and satisfy the above requirements for the existence of $F^+$ and $F^\times$, which are also containers \cite{Abbott05,Ahman_2014}. Moreover, containers over $\mathsf{Set}$ contain the functors $A^\circ$, $A^\bullet$, and $I$, and are closed under $+$, $\times$, and composition. It follows that there is a double functor $D : \corner{\mathsf{Set}}^* \to \mathsf{S}(\mathsf{Set})$ whose image lies in the part of $\mathsf{S}(\mathsf{Set})$ obtained by restricting the vertical edge monoid to the part consisting of only the containers. We leave the project of precisely characterising those $\C$ for which this double functor exists for future work. 

\begin{remark}\label{rem:effects-interp-iteration}
  We consider the interpretation of iterated protocols given by $(-)^+$ and $(-)^{\times}$ from the perspective of computational effects, extending Remarks~\ref{rem:effects-interpretation} and \ref{rem:effects-interp-choice}. As a computational effect, $F^+$ enables a program to trigger the effect $F$ any number of times it chooses. Dually, to resolve an effect $F^{+}$ its environment must be able to resolve $F$ any number of times. Similarly, to trigger $F^{\times}$ a program must provide a method for triggering $F$ any number of times, as required by the environment. Dually, to resolve $F^{\times}$ the environment must commit to resolving $F$ a number of times of its choosing, and then resolve those effects.
\end{remark}

\section{Concluding Remarks}\label{sec:concluding-remarks}

We have shown how to extend the free cornering of a symmetric monoidal category to support both branching communication protocols and iterated communication protocols, bringing it closer to existing systems of session types. Specifically, we have constructed the free cornering with choice (Definition~\ref{def:free-cornering-choice}) and the free cornering with iteration (Definition~\ref{def:iter-doub}) of a distributive monoidal category, shown that they inherit significant categorical structure from the free cornering, and provided some evidence that they fit well into the categorical landscape. Further, we have constructed the double category of stateful transformations (Definition~\ref{def:stateful-transformations}) --- a model of the structure found in the free cornering, free cornering with choice, and free cornering with iteration.

While our work constitutes a significant step, the path is long, and we envision our work here as a small part of a much larger research project surrounding the free cornering. In this final section we elucidate this project by outlining a number of directions for future work:

\paragraph{Active Iteration} There is a mismatch between our constructions of the free cornering with choice and the free cornering with iteration. In the free cornering with choice, we have a pair of dual operations $-+-$ and $-\times-$ on cells corresponding to \emph{reactive} protocol choice, and a third operation $[-,-]$ which allows \emph{active} protocol choice. In the free cornering with iteration we again have a pair of dual operations $(-)^+$ and $(-)^\times$ corresponding to \emph{reactive} protocol iteration, but we are missing their active counterpart. That is, in the free cornering with iteration we cannot model processes that choose whether or not to continue iterating a given protocol as a function of their input. Put another way, we ought to be able to control the iteration of a communication process with a ``while loop'', but this would require a notion of ``while loop'' in the vertical direction. 

We briefly speculate about the form that the axioms for active iteration ought to take. For each quartet of cells $\alpha : \corner{\A}^*\cells{U}{A}{B \oplus A}{W}$, $f : \corner{\A}^*\cells{S}{B}{C}{K}$, $g : \corner{\A}^*\cells{S}{I}{I}{U \otimes S}$, and $h : \corner{\A}^*\cells{W \otimes K}{I}{I}{K}$ we seem to require a cell $\alpha^*_{f,g,h} : \corner{\A}^*\cells{S}{B \oplus A}{C}{K}$ satisfying:
\begin{mathpar}
  \includegraphics[height=1.2cm,align=c]{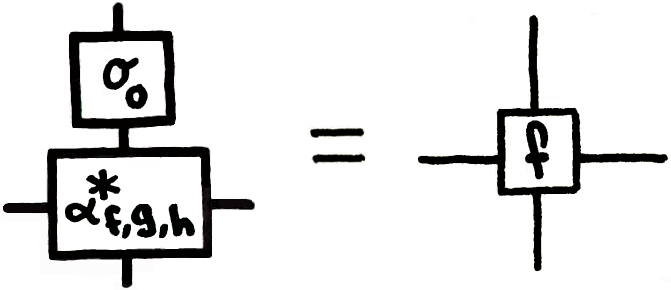}
  
  \includegraphics[height=1.4cm,align=c]{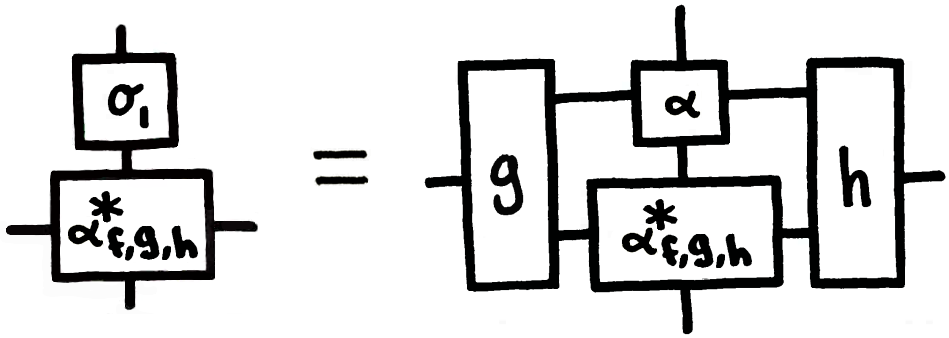}
\end{mathpar}
Significantly, asking for $\alpha^*_{f,g,h}$ to be the unique such cell is too strong, and collapses the hom-sets of the resulting category of vertical cells.
While this sort of active iteration is convenient for constructing examples, it is unclear what sort of properties we ought to ask for in order to obtain e.g., well-behaved crossing cells. 
We speculate that a double-categorical analogue of the notion of \emph{uniform trace operator} (see e.g,~\cite{Hasegawa2003}) will suffice, but how such an analogue should look has not yet been fully worked out. 

In the presence of cells $\alpha^*_{f,g,h}$ we may extend Example~\ref{ex:bread} as follows: let $\texttt{buy} : \cells{I}{S_\texttt{bread} \otimes \texttt{\$} \otimes S_\texttt{\$}}{S_\texttt{bread} \oplus (S_\texttt{bread} \otimes \texttt{\$} \otimes S_\texttt{\$})}{\texttt{bread}^\bullet \otimes \texttt{bread}^\bullet \otimes \texttt{\$}^\circ}$ be the cell below on the left, then define $\texttt{buys'} = \texttt{buy}^*_{1_{S_\texttt{bread}},\square_I,\ip_1} : \cells{I}{S_\texttt{bread} \oplus (S_\texttt{bread} \otimes \texttt{\$} \otimes S_\texttt{\$})}{S_\texttt{bread}}{(\texttt{bread}^\bullet \otimes \texttt{bread}^\bullet \otimes \texttt{\$}^\circ)^+}$, and let $\texttt{buys} : \cells{I}{S_\texttt{bread} \otimes S_\texttt{\$}}{S_\texttt{bread} \otimes S_\texttt{\$}}{(\texttt{bread}^\bullet \otimes \texttt{bread}^\bullet \otimes \texttt{\$}^\circ)^+}$ be the cell below on the right:
  \begin{mathpar}
    \includegraphics[height=2cm,align=c]{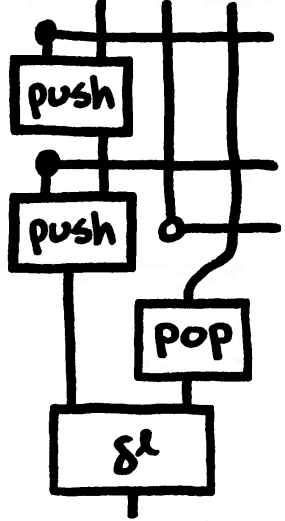}

    \includegraphics[height=2cm,align=c]{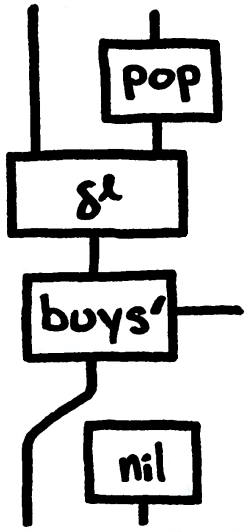}
  \end{mathpar}
  Now the behaviour of $\texttt{buys}$ depends on how much money it has. Specifically, we have:
  \begin{mathpar}
    \includegraphics[height=1cm,align=c]{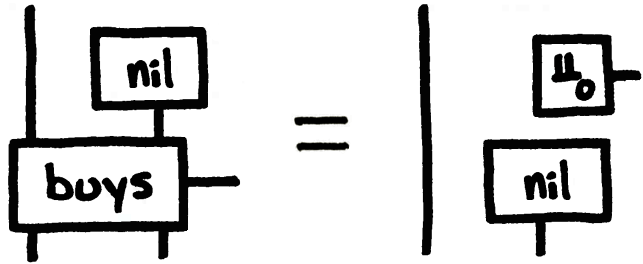}

    \includegraphics[height=1.5cm,align=c]{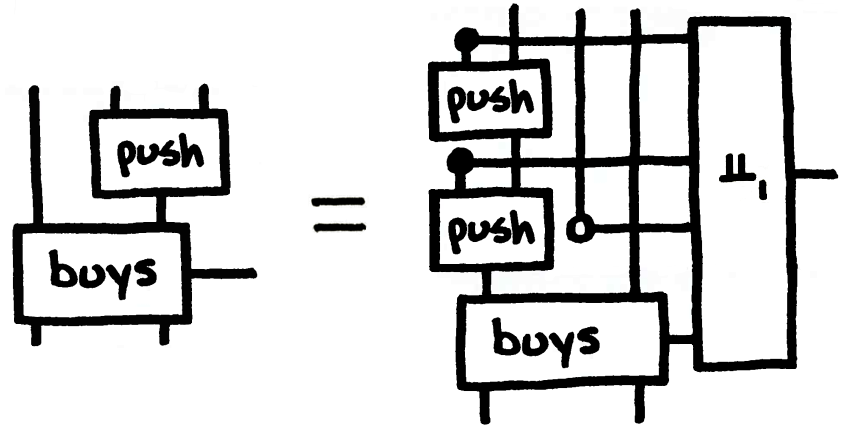}
  \end{mathpar}
  so $\texttt{buys}$ buys bread until it is out of money. Now, we may also consider $\frac{\texttt{sales}}{\texttt{buys}}$, the process which first sells bread until instructed to switch modes, and then buys bread until it is out of money.

\paragraph{Abstract Definitions} In the free cornering, the corner cells carry the structure of a proarrow equipment. A natural question is what structure is carried by the cells of the free cornering with choice and free cornering with iteration. While the structure of the free cornering with choice is clearly some sort of product or coproduct on a single-object double category, it is not clear what sort of limit this is. We remark that it does not seem to be directly related to the double-categorical limits studied in~\cite{Grandis1999}. Similarly, it is unclear what double-categorical structure the cells of the free cornering with iteration carry. We suspect this to be a fruitful direction for future work. 

\paragraph{Term Logic for Cells} Any comparison of the free cornering (with or without choice and iteration) to existing models of concurrent computation is made somewhat awkward by the lack of a term calculus and accompanying term rewriting system for the free cornering. Most existing process calculi and process algebras are first and foremost term calculi, and do not tend to have an accompanying categorical semantics. The free cornering exists only as categorical semantics. Thus, in order to better situate our work in the literature on concurrent computation we would seem to require a term calculus for the free cornering.

While the terms of a rewriting system often form a category, we are not aware of any rewriting systems in which the terms form a double category. In particular, while systems of \emph{tile logic} (see e.g.,~\cite{Bruni2002}) form double categories, there the cells of the double category in question correspond to the \emph{rewrites}, while for us the cells must correspond to the \emph{terms}. There is an evident notion of \emph{$\mathsf{Cat}$-enriched double category}, in which the cell-sets of the double category in question are in fact categories --- in which morphisms correspond to rewrites --- and the composition operations are given by functors. $\mathsf{Cat}$-enriched categories are known to model rewriting systems in which the terms form a category~\cite{Power2005}, and so we expect that the rewriting systems appropriate to our setting will form $\mathsf{Cat}$-enriched double categories. This requirement should guide future developments in in the direction of a term logic for the free cornering. 
  
\paragraph{Coherence and Vertical Cells} An important property of the free cornering is that the vertical cells are the base category:
\begin{proposition}[\cite{Nester2021a}]\label{prop:verticaloriginal}
Let $A$ be a symmetric monoidal category. Then there is an isomorphism of categories $\bv\,\corner{\A} \cong \A$.
\end{proposition}
We think of this as a kind of coherence. We conjecture that the free cornering with choice and free cornering with iteration are also coherent in this way:
\begin{conjecture}
  Let $\A$ be a distributive monoidal category. Then:
  \begin{enumerate}[(i)]
  \item There is an isomorphism of categories $\bv\,\corner{\A}^\oplus \cong \A$.
  \item There is an isomorphism of categories $\bv\,\corner{\A}^* \cong \A$.
  \end{enumerate}
\end{conjecture}
While we believe it to be true, we currently lack the machinery necessary to prove our conjecture. The most promising approach looks to be through the sort of term calculus and rewriting system for the free cornering discussed above, further motivating its development. 

\paragraph{Additional Effects of Stateful Transformations}
We have given a model of the free cornering in terms of strong functors and strong natural transformations. Besides the protocols in the image of the double functor from the free cornering of $\C$ into $\Dst{\C}$, other effects modelled by strong monads occur as protocols in $\mathsf{S}(\C)$ as well, such as monads for nondeterminism and probability. One direction for future work would be to axiomatise selected effects directly in the free cornering. Furthermore, the representation of computational effects in the form of a double category could help us describe both operational and denotational semantics of effectful programs. On the operational side, double cells function as a kind of effect handler \cite{Pretnar13}, translating effects received from its interior into effects invoked in its environment. On the denotational side, models related to stateful nondeterministic runners \cite{Voorneveld22} and other program-environment interaction laws \cite{Katsumata20} seem particularly suitable for interpretation in this model.

\bibliographystyle{plain}
\bibliography{citations}

\end{document}